\documentclass[11pt]{article}

\usepackage{geometry}
\usepackage{amssymb,amsmath,latexsym,amsfonts,amsthm}
\usepackage{mathtools}

\usepackage{xcolor,graphicx}

\usepackage{bussproofs}

\usepackage{lscape}

\usepackage{tikz}

\title{Uniform interpolation for interpretability logic}
\author{Sebastijan Horvat,\\
  {\small Department of Mathematics, Faculty of Science, University of Zagreb,
  Croatia}, \and
Borja Sierra Miranda,\\
{\small University of Bern, Switzerland},\footnote{Research supported by the Swiss National Science Foundation project 200021\_214820.}\and
Thomas Studer,\\
{\small University of Bern, Switzerland}.}
\date{}

\newtheorem{theorem}{Theorem}[section]
\newtheorem{lemma}[theorem]{Lemma}
\newtheorem{corollary}[theorem]{Corollary}
\newtheorem{proposition}[theorem]{Proposition}
\theoremstyle{definition}
\newtheorem{definition}[theorem]{Definition}

\newcommand\restricts{\mathord{\upharpoonright}}

\DeclareMathOperator\vocab{Voc}

\newcommand\interp{\mathbin{\rhd}}
\newcommand\nec{\Box}
\newcommand\bnec{\mathop{\rule[0.5pt]{6pt}{6pt}}}
\newcommand\dnec{\boxdot}
\newcommand\pos{\Diamond}

\newcommand\IL{\mathsf{IL}}
\newcommand\IKT{\mathsf{IK4}}
\newcommand\ILP{\mathsf{ILP}}

\newcommand\ILM{\mathsf{ILM}}
\newcommand\GL{\mathsf{GL}}
\newcommand\domain{\mathop{\mathrm{Dom}}}
\newcommand\image{\mathop{\mathrm{Im}}}

\newcommand\fgentzenIL{\mathcal{G}\IL}
\newcommand\gentzenIL{\mathcal{G}^\infty \IL}
\newcommand\auxgentzenIL{\mathcal{G}^\infty \IL'}
\newcommand\cgentzenIL{\mathcal{G}^\circ \IL}
\newcommand\slimgentzenIL{\mathcal{G}^{\mathrm{slim}}\IL}

\newcommand\toWff[1]{{#1}^\sharp}

\newcommand\cut{\mathrm{Cut}}
\DeclareMathOperator\wk{wk}

\DeclareMathOperator\ctr{Ctr}

\newcommand\ax{\mathrm{ax}}
\newcommand\Ax{\mathrm{Ax}}
\newcommand\botL{\bot\mathrm{L}}
\newcommand\botR{\bot\mathrm{R}}
\newcommand\toL{{\to}\mathrm{L}}
\newcommand\toR{{\to}\mathrm{R}}
\newcommand\negL{{\neg}\mathrm{L}}
\newcommand\negR{{\neg}\mathrm{R}}
\newcommand\veeL{{\vee}\mathrm{L}}
\newcommand\veeR{{\vee}\mathrm{R}}
\newcommand\wedgeL{{\wedge}\mathrm{L}}
\newcommand\wedgeR{{\wedge}\mathrm{R}}
\newcommand\interpIKT{\interp_{\IKT}}
\newcommand\interpIL{\interp_{\IL}}
\newcommand\emp{\mathrm{Empty}}
\newcommand\rep{\mathrm{Repeat}}
\newcommand\ninv{\mathrm{\interp^*_{\IKT}}}
\newcommand\Wk{\mathrm{Wk}}
\newcommand\Ctr{\mathrm{Ctr}}
\newcommand\Eq{\mathrm{Eq}}

\DeclareMathOperator\height{hg}
\DeclareMathOperator\lheight{lhg}
\newcommand\sub{\mathop{\mathrm{Sub}}\nolimits_{\interp}}

\DeclareMathOperator\nodes{Node}

\newcommand\an{\mathrm{an}}
\newcommand\su{\mathrm{su}}

\usepackage[backend=biber]{biblatex}
\usepackage[title,toc,titletoc,page]{appendix}

\begin{document}

\maketitle

\begin{abstract}
  We present a proof-theoretical study of the interpretability logic \(\IL\), providing a wellfounded and a non-wellfounded sequent calculus for \(\IL\).
  The non-wellfounded calculus is used to establish a cut elimination argument for both calculi.  In addition, we show that the non-wellfounded proof theory of \(\IL\) is well-behaved, i.e., that cyclic proofs suffice.
  This makes it possible to prove uniform interpolation for~\(\IL\).
  As a corollary we also provide a proof of uniform interpolation for the interpretability logic~\(\ILP\).
\end{abstract}

\section{Introduction}
\label{sec:introduction}

This paper is concerned with the proof theory of interpretability logic \(\IL\) (see \cite{Visser1990}), i.e., the extension of provability logic with a binary modality formalizing interpretability. 
We introduce three calculi for \(\IL\): a wellfounded Gentzen calculus \(\fgentzenIL\), a non-wellfounded local progress calculus \(\gentzenIL\) and the regularization of the previous calculus, i.e., a cyclic local progress calculus \(\cgentzenIL\).
We show proof-theoretically the equivalence of these three calculi and also their equivalence to the usual Hilbert-style calculus for \(\IL\). 
Our procedure is displayed in Figure~\ref{fig:plan}.

Additionally, we will use the non-wellfounded calculus in order to provide a proof of uniform interpolation for \(\IL\).
We will also show uniform interpolation for the interpretability logic \(\ILP\) by interpreting it inside \(\IL\).
To the best knowledge of the authors, these two results were unknown previous to this work.



\begin{figure}
	\centering
	\begin{tikzpicture}
	\node at (0,2) {\(\IL\)};
	\draw[-latex] (0.5,2) -- (2.5,2); \node at (1.5,2.3) {\small Thm~\ref{tm:IL-and-fgentzenILcut}};

	\node at (3.5,2) {\(\fgentzenIL+\cut\)};
	\draw[-latex] (4.7,2) -- (7.2,2); \node at (5.9,2.3) {\small Thm~\ref{tm:fgentzenILcut-to-genteznILcut}};

	\node at (8.5,2) {\(\gentzenIL+\cut\)};
	\draw[-latex] (2.3,2) -- (0.5,2);
	\draw[-latex] (8,1.6) -- (8.0,0.4); \node at (7.3,1.1) {\small Cor~\ref{tm:gentzenILcut-to-gentzenIL}};

	\node at (3.5,0) {\(\fgentzenIL\)};
  \draw[-latex] (3.5,0.4) -- (3.5,1.6); \node at (4.2,1) {\small Trivial};

	\node at (8,0) {\(\gentzenIL\)};
	\draw[-latex] (7.2,0) -- (4,0); \node at (5.5,-0.3) {\small Thm~\ref{thm:gentzenIL-to-fgentzenIL}};
	\draw[-latex] (8.6,0) -- (10.2,0);  \node at (9.3,-0.3) {\small Thm~\ref{thm:gentzenIL-to-slimgentzenIL}};

	\node at (11,2) {\(\cgentzenIL\)};
  \draw[-latex] (10.5,1.8) -- (8.5,0.2);  \node [rotate=39] at (9.5,1.3) {\small Unfolding};

	\node at (11,0) {\(\slimgentzenIL\)};
  \draw[-latex] (11,0.4) -- (11,1.6); \node at (12,1) {\small Thm~\ref{thm:slimgentzenIL-to-cgentzenIL}};
	\end{tikzpicture}
	\caption{The plan. Arrows without labels are omitted in this paper.}
  	\label{fig:plan}
\end{figure}
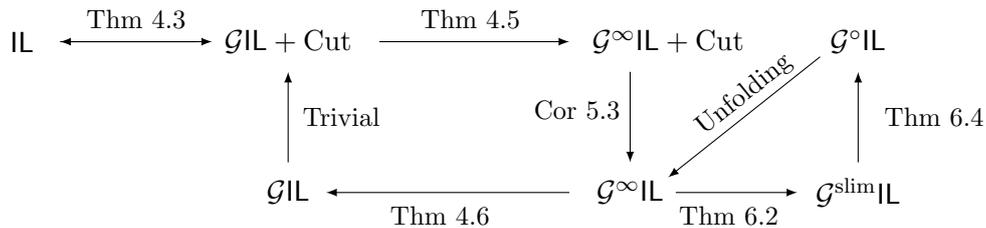

Thus, the contributions of this paper are threefold:
\begin{enumerate}
\item We refine our previous work on developing a general proof theory of non-wellfounded local progress calculi. In particular, we introduce the notions of admissible, locally admissible, eliminable, and locally eliminable rules and study their relationship.
\item We present a simple syntactic cut elimination method for interpretability logic. In particular, the cut reductions will mimic the cut reductions of \(\IKT\), i.e., \(\IL\) without L\"ob's axiom. To do so, we introduce a traditional Gentzen-style sequent calculus for \(\IL\) and a non-wellfounded version of it. 
\item Our non-wellfounded proofs exhibit a regular structure (i.e., they lead to cyclic proofs). This allow us to use them to establish uniform interpolation for \(\IL\).
  The definition of the interpolant will be far from trivial, due to the shape of the rules that are necessary for the calculus.
  This uniform interpolation result also makes it possible  to derive uniform interpolation for the interpretability logic \(\ILP\).
\end{enumerate}
The first two points partly appear in our previous work \cite{tableaux}.
We include them here in order to present their full proofs as well as to make this paper self-contained.

\textbf{Related Work.} There are three directions of closely related work.
The first one is non-wellfounded and cyclic proof theory.
The structure and methodology of this paper has been inspired by the seminal paper~\cite{ShamkanovGrz}.
We follow the trend started in that paper of defining a non-wellfounded Gentzen calculi from a finite one where cut elimination becomes easier to show.
There are many proposed methods for cut elimination in non-wellfounded and cyclic proofs.
The interested reader may consult \cite{acclavio2024infinitarycuteliminationfiniteapproximations, cut-cyclic, afshari2025demystifyingmu, das_et_al:LIPIcs.CSL.2018.19, miranda2025cuteliminationnonwellfoundedmaster, alexis, ShamkanovGrz, shamkanov2023structuralprooftheorymodal, coalgebraic-translations}, among others.

We use our own method of cut elimination, described in detail in \cite{coalgebraic-translations}, as it simplifies the non-wellfounded cut elimination to the point of making it completely analogue to the finitary case. 

The second one is the proof-theoretical study of interpretability logics.
Sasaki's work~\cite{sasaki2002, sasaki2002p, sasaki2003} has been a fundamental reference for this paper.
Part of our motivation was to simplify his approach with the use of modern tools (e.g. non-wellfounded proof theory) and build from them.
More recently, \cite{iwata} has also studied the proof theory of subsystems of \(\IL\).

Finally, the last direction is the study of uniform interpolation.
Uniform interpolation was first considered by Pitts~\cite{Pitts1992-PITOAI},  who established it  for propositional intuitionistic logic.
Usually,   methods to prove uniform interpolation are divided into semantical and syntactical.
Pitts' method is syntactical and it is based, implicitely, on proof search.
The method we are going to use is also syntactical, it is also based on proof search, but it has two big differences compared  to Pitts'.
Firstly, the proof search will be explicit in the construction, which we call interpolation template.
Secondly, the proof search may contain loops.  This defines a system of  equations  of modal formulas,  which we have to solve to find the interpolant.
This methodology for uniform interpolation, in its modern shape, first appeared on \cite{uniform-interp-mu}.
The reader interested in interpolation for provability logics can also consult \cite{uniform-interp-mu, marta, Visser_2017, glp, kogure2023interpolationpropertiesbimodalprovability, Areces1998}, among others.

\textbf{Summary of Sections.}
In the next section we will introduce the basic concepts of interpretability logic and non-wellfounded proof theory needed for the rest of the paper.
Section~\ref{sec:sequent-calculi} will introduce the Gentzen calculi \(\gentzenIL\) and \(\fgentzenIL\).
Section~\ref{sec:transformations} is devoted to showing different translations between the calculi.
Section~\ref{sec:cut-elim} provides cut elimination for \(\gentzenIL\). This result together with the translations of Section~\ref{sec:transformations} provides a cut elimination method for \(\fgentzenIL\) and the equivalence of \(\IL\), \(\fgentzenIL\) and \(\gentzenIL\).
Section~\ref{sub:cyclic-systems} establishes  the triangle on the right of Figure~\ref{fig:plan},  i.e., we will prove that any non-wellfounded proof can be transformed into a regular proof.
In Section~\ref{sec:interpolation}, we will prove the uniform interpolation property for \(\IL\) using the non-wellfounded proof theory for \(\IL\) we developed in the previous sections.
In addition, we will obtain uniform interpolation for \(\ILP\) from the uniform interpolation of \(\IL\).

\section{Preliminaries}
\label{sec:preliminaries}

In this section we will introduce the basic concepts needed for future sections.

\subsection{Interpretability Logic}

In this subsection we will define the interpretability logic that we will be working with. We will also prove that certain formulas, which will be useful to us in the next sections, are theorems of this logic.

The syntax of interpretability logic is given by 
$$\phi::=p\ |\ \bot\ |\ \phi\to\phi\ |\ \phi\rhd\phi,$$ 
where $p$ ranges over a fixed infinite countable set of propositional variables. 
We call formulas of this language \(\IL\)-formulas.
When it is clear from the context that we are talking about \(\IL\)-formulas, we will just write formula instead of \(\IL\)-formula.
Other Boolean connectives can be defined as abbreviations as usual, i.e., \(\neg \phi = \phi \to \bot\), \(\phi \vee \psi = \neg \phi \to \psi\), \(\phi \wedge \psi = \neg(\phi \to \neg \psi)\).
$\Box\phi$ can be defined as an abbreviation, namely $\Box \phi = \neg \phi\rhd\bot$ and we set $\Diamond\phi = \neg(\phi \interp \bot)$. 
We will also use the abbreviation \(\bnec \phi = (\phi \interp \bot) \wedge \phi\) and
\(\dnec \phi = \phi \wedge \nec \phi\).

A formula of the form $\phi\rhd\psi$ will be called a $\rhd$-formula.
Given a \(\interp\)-formula \(\phi \interp \psi\), we wil say that \(\phi\) is its \emph{antecedent}, also denoted as \(\an(\phi \interp \psi)\), and \(\psi\) is its \emph{succedent}, also denoted as \(\su(\phi \interp \psi)\).
Given a multiset \(\Sigma\) of \(\interp\)-formulas, we will write \(\an(\Sigma)\) to mean the \emph{multiset} of antecedents in \(\Sigma\) and \(\su(\Sigma)\) to mean the \emph{multiset} of succedents in \(\Sigma\).\label{antecedent-succedent}.

We use lower case Latin letters $p$, $q$, ..., possibly with subscripts, for propositional variables
and lower case Greek letters $\phi$, $\psi$, ..., possibly with subscripts,
for \(\IL\)-formulas.
To avoid too many parentheses in longer formulas, 
we treat $\rhd$ as having higher priority than $\rightarrow$, but lower than other Boolean connectives.
Unary operators \(\nec\), \(\pos\) and \(\neg\) have the highest priority.

The idea of interpretability logics originates by extending the usual interpretation of modal logic inside arithmatic \(T\) by adding a binary modality \(\interp\).
Then \(\phi \interp \psi\) is understood as \(T + \psi\) is relative interpretable in \(T + \phi\).
Some interpretability logics are sound and complete with respect to this semantics, e.g., \(\ILM\) when we choose \(T\) to be Peano Arithmetic.
\(\IL\), the logic we are going to study, is sound with respect to many arithmetical theories, but incomplete.
However, it contains a good portion of the rest of interpretability logics and it is an appealing logic from the modal point of view.
For details the reader is encouraged to read \cite{Visser1990}.

In some proofs we will use the following auxiliary definition of a size of an \(\IL\)-formula.
\begin{definition}
  The \emph{size} \(|\phi|\) of an \(\IL\)-formula \(\phi\) is defined recursively as follows:
  \[ |\bot| = 1,\qquad |p| = 1,\qquad |\phi \to \psi| = |\phi \interp \psi| = |\phi| + |\psi| + 1.
  \]
\end{definition}

We define the interpretability logic we will consider in this paper.
\begin{definition}
Interpretability logic \(\IL\) is the smallest set of \(\IL\)-formulas that contains all the classical propositional tautologies and axioms
\begin{align*}
  &(\mathrm{K})\quad \Box(\phi\rightarrow\psi)\rightarrow(\Box\phi\rightarrow\Box\psi),
  && (\mathrm{4})\quad \Box\phi\rightarrow\Box\Box\phi,\\
  &(\mathrm{L})\quad \nec(\nec\phi \to \phi) \to \nec \phi,
  &&(\mathrm{J1})\quad \Box(\phi\rightarrow\psi)\rightarrow(\phi\rhd\psi), \\
  &(\mathrm{J2})\quad (\phi\rhd\chi)\wedge(\chi\rhd\psi)\rightarrow(\phi\rhd\psi),
  &&(\mathrm{J3})\quad (\phi\rhd\psi)\wedge(\chi\rhd\psi)\rightarrow(\phi\vee\chi)\rhd\psi,\\
  &(\mathrm{J4})\quad \phi\rhd\psi\rightarrow(\Diamond\phi\rightarrow\Diamond\psi).
  &&(\mathrm{J5})\quad \Diamond\phi\rhd\phi
\end{align*}
and is closed under modus ponens and necessitation:
\[
        \AxiomC{\(\phi \to \psi\)}
        \AxiomC{\(\phi\)}
        \BinaryInfC{\(\psi\)}
        \DisplayProof
        \ ,\qquad
        \AxiomC{\(\phi\)}
        \UnaryInfC{\(\nec \phi\)}
        \DisplayProof\ .
\]
Sometimes we will be referring to axiom (L) as \emph{Löb axiom}.
Note that the definition of \(\IL\) is not minimal.
In particular \((\mathrm{J4})\) is derivable from \((\mathrm{J2})\).
\end{definition}

In the following lemma we will put together some basic properties of \(\IL\).
These results will be used in some proofs in the remainder of this paper.

\begin{lemma}\label{lm:basics-of-IL}
  Let \(\phi,\psi\) be formulas and \(\Sigma\) be a non-empty finite multiset of formulas.
  Then
  \begin{enumerate}
    \item \(\IL \vdash \phi \to \psi\) implies \(\IL \vdash \phi \interp \psi\).
    \item (L\"ob's rule in \(\IL\)) \(\IL \vdash \psi \wedge \bigwedge (\Sigma \interp \bot) \to \bigvee \Sigma\) implies \(\IL \vdash \psi \interp \bigvee \Sigma\).
    \item \(\IL \vdash \phi \interp \bnec \phi\).
  \end{enumerate}
\end{lemma}
\begin{proof}
  The proof of 1.\ is trivial using necessitation and (J1).
  Let us prove 2.\ and 3.\ in detail.

  Assume \(\IL \vdash\psi \wedge \bigwedge (\Sigma \interp \bot) \to \bigvee \Sigma\). 
  From classical propositional reasoning we obtain 
  \(\IL \vdash \psi \to \bigvee \Sigma \vee \bigvee (\pos \Sigma)\), so by using 1.\ we obtain
  \begin{equation}\label{IL-theorem:1}
    \tag{i}
    \IL \vdash \psi \interp \left(\bigvee \Sigma \vee \bigvee \pos \Sigma\right)
  \end{equation}
  Let \(\Sigma = \{\phi_0,\ldots, \phi_m\}\) and notice that \(\IL \vdash \pos \phi_j \interp \phi_j\) for each \(j \leq m\) thanks to (J5).
  For each \(j \leq m\) we also have by 1.\ that \(\IL \vdash \phi_j \interp \phi_j\), 
  so we get by (J3) that \(\IL \vdash (\phi_j \vee \pos \phi_j) \interp \phi_j\).
  Also from \(\IL \vdash \phi_j \to \bigvee_{i \leq m} \phi_i\) 
  by 1.\ we get \(\IL \vdash \phi_j \interp \bigvee_{i \leq m} \phi_i\).
  Then, using (J2) we obtain \(\IL \vdash (\phi_j \vee \pos \phi_j) \interp \bigvee_{i \leq m} \phi_i\) for each \(j \leq m\) and, by (J3), 
  	\begin{equation} \label{IL-theorem:2}
      \tag{ii}
  	\IL \vdash \left(\bigvee_{i \leq m}(\phi_i \vee \pos \phi_i) \right) \interp \bigvee_{i \leq m}\phi_i.
	\end{equation}
  Since the formulas \(\left(\bigvee_{i \leq m} \phi_i\right) \vee \left(\bigvee_{i \leq m} \pos \phi_i\right)\) and \(\bigvee_{i \leq m} (\phi_i \vee \pos \phi_i)\) are equivalent in classical propositional logic by 1.\ we obtain 
  \begin{equation} \label{IL-theorem:3}
    \tag{iii}
  	\IL \vdash \left(\bigvee \Sigma \vee \bigvee \pos \Sigma\right) \interp \left(\bigvee_{i \leq m} (\phi_i \vee \pos \phi_i)\right)
  \end{equation}
  where we used that \(\left(\bigvee_{i \leq m} \phi_i\right) \vee \left(\bigvee_{i \leq m} \pos \phi_i\right)\) is just the same formula as \(\bigvee \Sigma \vee \bigvee \pos \Sigma\).
  So using (\ref{IL-theorem:1}), (\ref{IL-theorem:2}), (\ref{IL-theorem:3}) and the (J2) axiom gives us
  \[
    \IL \vdash \psi \interp \bigvee_{i \leq m}\phi_i,
  \]
  as desired.

  Proof of 3.
  By L\"ob's axiom we obtain that \(\IL \vdash \nec(\nec \neg \phi \to \neg \phi) \to \nec \neg \phi\).
  Unfolding some definitions of \(\nec\) we get \(\IL \vdash \neg(\nec \neg \phi \to \neg \phi) \interp \bot \to \neg\neg\phi \interp \bot\).
  Since \(\IL \vdash \neg\neg \phi \leftrightarrow \phi\) and \(\IL \vdash \neg(\nec \neg \phi \to \neg \phi) \leftrightarrow (\nec \neg \phi \wedge \phi)\), using 1.\ and (J2) we get \(\IL \vdash (\nec \neg \phi \wedge \phi) \interp \bot \to \phi \interp \bot\), or in other words \(\IL \vdash ((\neg\neg\phi \interp \bot) \wedge \phi) \interp \bot \to \phi \interp \bot\).
  Using again that \(\IL \vdash \neg\neg\phi \leftrightarrow \phi\) with 1.\ and (J2) (multiple times) we obtain \(\IL \vdash ((\phi \interp \bot) \wedge \phi) \interp \bot \to \phi \interp \bot\), i.e., \(\IL \vdash \bnec \phi \interp \bot \to \phi \interp \bot\).
  Adding \(\phi\) on both sides we have \(\IL \vdash \phi \wedge (\bnec \phi \interp \bot) \to \phi \wedge \phi \interp \bot\), or analogously, \(\IL \vdash \phi \wedge (\bnec \phi \interp \bot) \to \bnec \phi\).
  Using 2.\ we conclude the desired \(\IL \vdash \phi \interp \bnec \phi\).
\end{proof}

\subsection{Non-wellfounded Proof Theory}

We introduce the basic concepts of (non-wellfounded) proof theory that we are going to use. 
The details can be found in \cite{coalgebraic-translations}.
We start with the definition of non-wellfounded finitely branching trees, from now own simply called \emph{trees}.

\begin{definition}
  A tree with labels in \(A\) is a function \(T\) such that
\begin{enumerate}
  \item \(\domain(T) \subseteq \mathbb{N}^{<\omega}\) is closed under prefixes and \(\image(T) \subseteq A\).
  \item For each \(w \in \domain(T)\) there is an unique \(k\), called the \emph{arity of \(w\)}, such that \(wi \in \domain(T)\) if and only if \(i < k\).
\end{enumerate}
The elements of \(\domain(T)\) are called \emph{nodes of \(T\)}.
Given two nodes \(w,v\) of \(T\) we will write \(w \leq v\) to mean that \(v = wu\) for some \(u \in \mathbb{N}^{<\omega}\) (\(u\) may be the empty sequence).
Also, we will writem \(w < v\) to mean that \(w \leq v\) and \(w \neq v\).
We will say that \(w\) and \(v\) are \emph{incomparable} if \(w \not\leq v\) and \(v \not\leq w\).

Given a tree \(T\) an \emph{(infinite) branch} is an infinite sequence \(b \in \mathbb{N}^\omega \) such that for each  \(i \in \mathbb{N}\), \(b \restricts i \in \domain(T)\), where \(b \restricts i = b_0 \cdots b_{i-1}\).
\end{definition}

\subsubsection{Basics of Local Progress Calculi.}\label{subsub:local-progress-calculi}

We use upper case Greek letters $\Gamma$, $\Delta$, $\Sigma$, $\Gamma'$, $\Delta'$, ...,
possibly with subscripts, for finite multisets of formulas.
The expression $\Gamma\rhd\bot$ denotes the multiset $\{\phi\rhd\bot\ |\ \phi\in\Gamma\}$.
By a sequent, we mean an ordered pair \((\Gamma,\Delta)\) usually denoted as \(\Gamma\Rightarrow\Delta\).
We use upper case Latin letters $S$, $S'$, ..., possibly with subscripts, for sequents.
We will write \(\Gamma, \Delta\) to mean \(\Gamma \cup \Delta\) and \(\phi,\Gamma\) or \(\Gamma,\phi\) to mean \(\{\phi\} \cup \Gamma\), as usual.\footnote{Note that in particular \(\Gamma_0,\ldots,\Gamma_n\) could be either \(\Gamma_0 \cup \cdots \cup \Gamma_n\) or a sequence of multisets of formulas with first element \(\Gamma_0\) and last element \(\Gamma_n\). The meaning of this expression should be clear by context.}
Also, we will write expressions like \((\Gamma,\phi,\Delta) \interp \bot\) to mean \((\Gamma \interp \bot) \cup \{\phi \interp \bot\} \cup (\Delta \interp \bot)\).
Sequences of sequents like \(S_{n}, \ldots, S_{0}\) will be denoted as \([S_i]_{n \ldots i \ldots 0}\).
The \emph{size} of a sequent \(\Gamma \Rightarrow \Delta\), denoted as \(|\Gamma \Rightarrow \Delta|\) will be the sum of the sizes of the formulas ocurring in it, taking into account repetitions.
For example, \(|\phi,\phi, \psi \Rightarrow \psi, \chi| = 2|\phi| + 2|\psi| + |\chi|\).

\begin{definition}
  An \emph{\(n\)-ary rule} is a set of \(n+1\)-tuples \((S_0,\ldots,S_n)\) where each \(S_i\) is a sequent.
  The elements of a rule are called its \emph{instances}.

  A \emph{local progress sequent calculus} is a pair \(G = (\mathcal{R}, L)\) where
  \begin{enumerate}
    \item \(\mathcal{R}\) is a set of rules.
    \item \(L\) is a function such that given a \(n\)-ary rule \(R\) and a rule instance \((S_0,\ldots, S_n)\) of \(R\) returns a subset of \(\{0,\ldots, n-1\}\), called \emph{progressing premises}.
      \(L\) is called the \emph{progressing function}.
  \end{enumerate}
\end{definition}

\begin{definition}
  Let \(\mathcal{G}\) be a local progress sequent calculus.
  A \emph{prederivation} \(\pi\) in \(\mathcal{G}\) is a non-wellfounded tree, whose internal nodes are annotated by a sequent and a rule of \(\mathcal{G}\) and the leafs are annotated by a sequent and a rule of \(\mathcal{G}\) or by a sequent only.
  The leafs which are annotated simply by a sequent are called \emph{assumptions} of the prederivation.
  In addition, for any \(n\)-ary node \(w\) of \(\pi\) annotated with a sequent \(S\) and a rule \(R\), we have that \((S_0,\ldots,S_{n-1},S) \in R\), where each \(S_i\) is the sequent at \(wi\) (the \(i\)-th successor of \(w\)).

  Given a prederivation \(\pi\) in \(\mathcal{G}\) and an infinite branch \(b\) in \(\pi\) we will say that \emph{\(b\) progresses at \(i\)} iff \(b_{i+1} \in L_R(S_0,\ldots,S_{n-1},S)\) where the node \(b \restricts i\) is \(n\)-ary, \(R\) is the rule at node \(b \restricts i\), \(S\) is the sequent at node \(b \restricts i\) and \(S_j\) is the sequent at node \((b\restricts i)j\) for \(j < n\).
  A prederivation \(\pi\) in \(\mathcal{G}\) is said to be a \emph{derivation} in \(\mathcal{G}\) iff for any infinite branch \(b\) of \(\pi\) the set \(\{i \in \mathbb{N} \mid b \text{ progresses at }i\}\) is infinite.

  A (pre)proof is a (pre)derivation without assumptions.
  We will write \(\mathcal{G} \vdash S\) to mean that there is a proof in \(\mathcal{G}\) whose conclusion is \(S\) and \(\pi \vdash_{\mathcal{G}} S\) to mean that \(\pi\) is a proof in \(\mathcal{G}\) with conclusion \(S\) (we will omit \(\mathcal{G}\) when it is clear from the context).

  A local progress calculus is said to be \emph{wellfounded} if its local progress function is the constant function always returning \(\varnothing\).\footnote{Derivations and proofs in a wellfounded calculus must be wellfounded, i.e., no infinite branches would be allowed so we recover the usual definitions of proof and derivation.}
  Given a local progress calculus \(\mathcal{G}\) and a rule \(R\) not in \(\mathcal{R}\) we will define the local progress calculus \(G + R\) by adding the rule \(R\) to the calculus and extending the local progress function such that no premise of an instance of \(R\) is a progressing premise.
\end{definition}

Given prederivations \(\pi_0,\ldots,\pi_{n-1}\) whose roots are annotated with the sequent \(S_i\), respectively, and a \(n\)-ary rule \(R\) such that \((S_0,\ldots,S_{n-1},S) \in R\) it will be common to write
\[
  \AxiomC{\(\pi_0\)}
  \noLine
  \UnaryInfC{\(S_0\)}
  \AxiomC{\(\cdots\)}
  \AxiomC{\(\pi_{n-1}\)}
  \noLine
  \UnaryInfC{\(S_{n-1}\)}
  \RightLabel{\(R\)}
  \TrinaryInfC{\(S\)}
  \DisplayProof
\]
to mean the prederivation whose root is annotated with sequent \(S\) and rule \(R\) and whose subtree at the \(i\)-th successor of the root is \(\pi_i\).
In case the prederivation is an assumption we will write it without the line.
If we write
\[
  \AxiomC{\(\pi_0\)}
  \noLine
  \UnaryInfC{\(S_0\)}
  \AxiomC{\(\cdots\)}
  \AxiomC{\(\pi_{n-1}\)}
  \noLine
  \UnaryInfC{\(S_{n-1}\)}
  \RightLabel{\(R_0, \ldots, R_{m}\)}
  \doubleLine
  \TrinaryInfC{\(S\)}
  \DisplayProof
\]
we mean the prederivation with conclusion \(S\) obtained from \(\pi_0, \ldots, \pi_{n-1}\) via multiple applications of the rules \(R_0\) to \(R_{m}\).
And in case we write
\[
  \AxiomC{\(\pi\)}
  \noLine
  \UnaryInfC{\(S\)}
  \dashedLine
  \UnaryInfC{\(\hspace{.5cm}S'\hspace{.5cm}\)}
  \DisplayProof
\]
we just mean the prederivation \(\pi\) where the sequent \(S\) is equal to the sequent \(S'\) but has been rewritten to ease the reading.
For example if \(\phi = \phi'\) and \(\pi\) is a derivation of \(\phi, \Gamma \Rightarrow \Delta\) then we may write
\[
  \AxiomC{\(\pi\)}
  \noLine
  \UnaryInfC{\(\phi, \Gamma \Rightarrow \Delta\)}
  \dashedLine
  \UnaryInfC{\(\psi, \Gamma \Rightarrow \Delta\)}
  \DisplayProof
\]
to make explicit that \(\pi\) is a prederivation of \(\psi, \Gamma \Rightarrow \Delta\).

\subsubsection{The Method of Translations.}\label{subsub:translations}

In \cite{coalgebraic-translations} we developed a method to construct translations between local progress calculi, i.e., to provide functions transforming proofs of one calculus into proofs (not necessarily of the same sequent) in another calculus.
Here, we will introduce informally the concepts and methods, the interested reader should consult \cite{coalgebraic-translations} for more details.

The idea goes as follows.
Given a proof \(\pi\) in a local progress calculus \(G\) we can define a partition of its nodes, the elements of the partition will be called \emph{local fragments}.
Two nodes will belong to the same local fragment if the smallest path between them does not go through progress.
Here, with passing through progress we mean going from the premise to the conclusion of a rule instance, or from conclusion to premise, such that the premise is progressing in the rule instance.
Thanks to the condition that any infinite branch progresses infinitely often, it is easy to see that each local fragment will be a finite tree, in other words, this slices the non-wellfounded tree into (possibly infinitely many) finite trees.
Figure \ref{fig:fragments} shows how the slicing can look in this setting, where each triangle represents a local fragment.

\begin{figure}
	\centering
	\begin{tikzpicture}[scale=0.4]
		\filldraw[gray!50] (0,-0.5) -- (3.5,2.5) -- (-3.5,2.5) -- cycle;
		\filldraw (0,0) circle (2pt);

		\filldraw (0,1) circle (2pt); \draw (0,0) -- (0,1);

		\filldraw[gray!50] (-3,2.5) -- (-4,4.5) -- (-2,4.5) -- cycle;
		\filldraw[gray!50] (-1,2.5) -- (-2,3.5) -- (0,3.5) -- cycle;
		\filldraw[gray!50] (2,2.5) -- (0,4.5) -- (4,4.5) -- cycle;
		\filldraw (-2,2) circle (2pt); \draw (0,1) -- (-2,2);
		\filldraw (2,2) circle (2pt); \draw (0,1) -- (2,2);

		\filldraw (-3,3) circle (2pt); \draw (-2,2) -- (-3,3);
		\filldraw (-1,3) circle (2pt); \draw (-2,2) -- (-1,3);
		\filldraw (2,3) circle (2pt); \draw (2,2) -- (2,3);

		\filldraw (-3,4) circle (2pt); \draw (-3,3) -- (-3,4); 
		\node at (-3,5) {\large\(\vdots\)};
		\filldraw (1,4) circle (2pt); \draw (2,3) -- (1,4); 
		\filldraw (3,4) circle (2pt); \draw (2,3) -- (3,4); 
	\end{tikzpicture}
	\caption{Structure of proofs in local progress calculi}
  	\label{fig:fragments}
\end{figure}
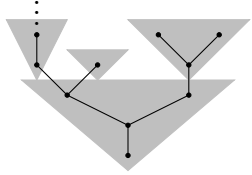

The bottom-most local fragment, i.e., the one to which the root belongs to, is called the \emph{main local fragment}.
We define the \emph{local height} of a proof \(\pi\), denoted as \(\lheight(\pi)\), as the height of its main local fragment (which is a finite tree, so indeed it has a height).

Finally, the translation method goes as follows.
To define a function from local progress Gentzen calculus \(G\) to local progress Gentzen calculus \(G'\), it suffices to provide another function (called \emph{corecursive step}) that, given a proof \(\pi\) in \(G\), returns:
\begin{enumerate}
  \item a local fragment in \(G'\), i.e., a finite tree generated by the rules of \(G'\) where every leaf is either axiomatic or a progressing premise and every progressing premise is a leaf;
  \item for each non-axiomatic leaf (of the local fragment) with sequent \(S\), a proof of \(S\) in \(G\).
\end{enumerate}
Then, the desired translation function is obtained by extending this corecursive step via corecursion.
The procedure is displayed in Figure \ref{fig:corecursion}.

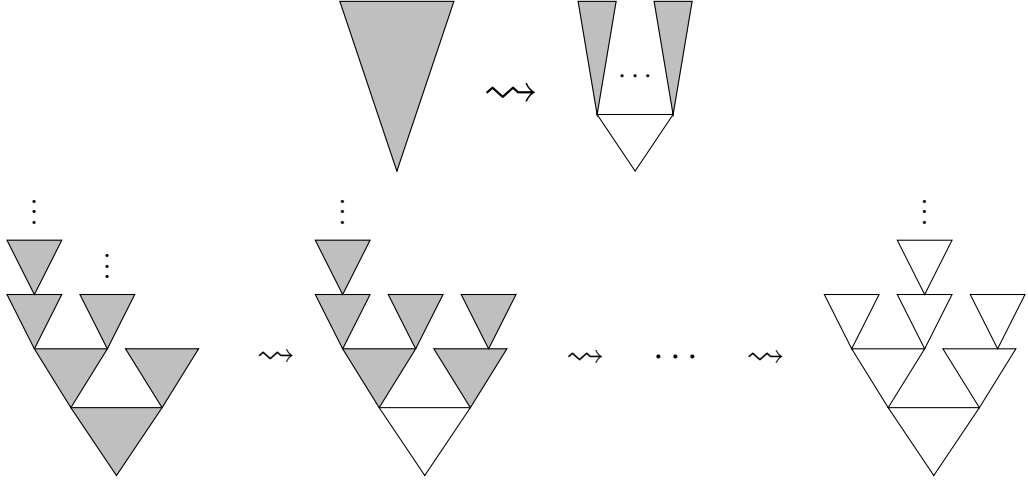
\begin{figure}
	\centering
	\begin{tikzpicture}[scale=0.5]
	\filldraw[gray!50] (-2.5,4.5) -- (-1,0) -- (0.5,4.5) -- cycle;
        \draw (-2.5,4.5) -- (-1,0) -- (0.5,4.5) -- cycle;
        \node at (2,2) {\huge\(\rightsquigarrow\)};
    \end{tikzpicture}\quad
    \begin{tikzpicture}[scale=0.5]
        \filldraw[gray!50] (-1.5,4.5) -- (-1,1.5) -- (-0.5,4.5) -- cycle;
        \draw (-1.5,4.5) -- (-1,1.5) -- (-0.5,4.5) -- cycle;
        \node at (0,2.5) {\(\cdots\)};
        \filldraw[gray!50] (0.5,4.5) -- (1,1.5) -- (1.5,4.5) -- cycle;
        \draw (0.5,4.5) -- (1,1.5) -- (1.5,4.5) -- cycle;
        \draw (-1,1.5) -- (0,0) -- (1,1.5) -- cycle;
	\end{tikzpicture}
	
	\begin{tikzpicture}[scale=0.6]
        \filldraw[gray!50] (2,1.5) -- (3,0) -- (4,1.5) -- cycle;
        \draw (2,1.5) -- (3,0) -- (4,1.5) -- cycle;

        \filldraw[gray!50] (1.2,2.8) -- (2,1.5) -- (2.8,2.8) -- cycle;
        \draw (1.2,2.8) -- (2,1.5) -- (2.8,2.8) -- cycle;
        \filldraw[gray!50] (3.2,2.8) -- (4,1.5) -- (4.8,2.8) -- cycle;
        \draw (3.2,2.8) -- (4,1.5) -- (4.8,2.8) -- cycle;

        \filldraw[gray!50] (0.6,4) -- (1.2,2.8) -- (1.8,4) -- cycle;
        \draw (0.6,4) -- (1.2,2.8) -- (1.8,4) -- cycle;
        \filldraw[gray!50] (2.2,4) -- (2.8,2.8) -- (3.4,4) -- cycle;
        \draw (2.2,4) -- (2.8,2.8) -- (3.4,4) -- cycle;

        \filldraw[gray!50] (0.6,5.2) -- (1.2,4) -- (1.8,5.2) -- cycle;
        \draw (0.6,5.2) -- (1.2,4) -- (1.8,5.2) -- cycle;
        \node at (2.8,4.8) {\(\vdots\)};
        
        \node at (1.2,6) {\(\vdots\)};

        \node at (6.5,2.6) {\Large\(\rightsquigarrow\)};
    \end{tikzpicture}
    \begin{tikzpicture}[scale=0.6]
        \draw (2,1.5) -- (3,0) -- (4,1.5) -- cycle;

        \filldraw[gray!50] (1.2,2.8) -- (2,1.5) -- (2.8,2.8) -- cycle;
        \draw (1.2,2.8) -- (2,1.5) -- (2.8,2.8) -- cycle;
        \filldraw[gray!50] (3.2,2.8) -- (4,1.5) -- (4.8,2.8) -- cycle;
        \draw (3.2,2.8) -- (4,1.5) -- (4.8,2.8) -- cycle;

        \filldraw[gray!50] (0.6,4) -- (1.2,2.8) -- (1.8,4) -- cycle;
        \draw (0.6,4) -- (1.2,2.8) -- (1.8,4) -- cycle;
        \filldraw[gray!50] (2.2,4) -- (2.8,2.8) -- (3.4,4) -- cycle;
        \draw (2.2,4) -- (2.8,2.8) -- (3.4,4) -- cycle;
        \filldraw[gray!50] (3.8,4) -- (4.4,2.8) -- (5,4) -- cycle;
        \draw (3.8,4) -- (4.4,2.8) -- (5,4) -- cycle;

        \filldraw[gray!50] (0.6,5.2) -- (1.2,4) -- (1.8,5.2) -- cycle;
        \draw (0.6,5.2) -- (1.2,4) -- (1.8,5.2) -- cycle;
        
        \node at (1.2,6) {\(\vdots\)};

        \node at (8.5,2.6) {\Large\(\rightsquigarrow\quad\cdots\quad\rightsquigarrow\)};
    \end{tikzpicture}\quad
    \begin{tikzpicture}[scale=0.6]
        \draw (2,1.5) -- (3,0) -- (4,1.5) -- cycle;

        \draw (1.2,2.8) -- (2,1.5) -- (2.8,2.8) -- cycle;
        \draw (3.2,2.8) -- (4,1.5) -- (4.8,2.8) -- cycle;

        \draw (0.6,4) -- (1.2,2.8) -- (1.8,4) -- cycle;
        \draw (2.2,4) -- (2.8,2.8) -- (3.4,4) -- cycle;
        \draw (3.8,4) -- (4.4,2.8) -- (5,4) -- cycle;

        \draw (2.2,5.2) -- (2.8,4) -- (3.4,5.2) -- cycle;
        
        \node at (2.8,6) {\(\vdots\)};
    \end{tikzpicture}
	\caption{Corecursive step function (top) and its extension from proofs to proofs (bottom). Tall gray (white) triangles represent proofs in \(G\) (\(G'\)) and short gray (white) triangles represent local fragments in \(G\) (\(G'\)). }
  	\label{fig:corecursion}
\end{figure}

\noindent
\textbf{Properties of Rules.}
Finally we introduce some properties of rules and proofs that will be fundamental to show cut elimination.

\begin{definition}
  Let \(R\) be an \(n\)-ary rule, \(\mathcal{G}\) be a local progress Gentzen calculus and \(\pi\) a proof in \(\mathcal{G} + R\).
  We say that
  \begin{enumerate}
    \item \(R\) is \emph{derivable} in \(\mathcal{G}\) if for any \((S_0,\ldots,S_{n-1}, S) \in R\) there is a derivation in \(\mathcal{G}\) with assumptions \(\{S_0,\ldots, S_{n-1}\}\) and conclusion \(S\).
    \item \(R\) is \emph{admissible} in \(\mathcal{G}\) if for any instance \((S_0,\ldots,S_{n-1},S)\) of the rule \(R\), \(\mathcal{G} \vdash S_0, \ldots, \mathcal{G} \vdash S_{n-1}\) implies that \(\mathcal{G} \vdash S\).
    \item \(R\) is \emph{invertible} if for each \(i < n\), the rule
      \[
        R^{-1}_i = \{(S_n, S_i) \mid \text{Exists } S_0,\ldots,S_{i-1},S_{i+1}, \ldots S_{n-1}.\ (S_0,\ldots,S_n) \in R\}
      \]
      is admissible.
      In words, if each of the rules which says that from the conclusion you can infer the premises is admissible.
    \item \(R\) is \emph{eliminable} in \(\mathcal{G}\) if for any sequent \(S\) if \(\mathcal{G} + R\vdash S\) then \(\mathcal{G} \vdash S\).
    \item \(\pi\) is \emph{locally \(R\)-free} if it contains no instances of \(R\) in its main local fragment.
    \item \(R\) is \emph{locally admissible} in \(\mathcal{G}\) if for any instance \((S_0,\ldots,S_{n-1},S)\) of the rule if \(\mathcal{G} \vdash S_0, \ldots, \mathcal{G} \vdash S_{n-1}\) with locally \(R\)-free proofs, then there is a locally \(R\)-free proof of \(\mathcal{G} \vdash S\).
    \item \(R\) is \emph{locally eliminable} if for any \(S\), if \(\mathcal{G} + R\vdash S\) then there is a locally \(R\)-free proof in \(\mathcal{G} + R\) of \(S\).
  \end{enumerate}
  The (local) admissibility/eliminability properties can be understood as asserting the existence of a proof \(\pi\) from the assumption that some proofs \(\pi_0,\ldots, \pi_{n-1}\) exist.
  Let \(P\) be a property of proofs, we say that any of the properties above holds \emph{preserving \(P\)} if, adding the extra assumption that \(\tau_0,\ldots, \tau_{n-1}\) fulfill \(P\), \(\pi\) also fulfills \(P\).
  In particular, we will say that we have (local) admissibility/eliminability of a rule \(R\) preserving height if \(\height(\pi) \leq \max(\height(\pi_0), \ldots, \height(\pi_{n-1}))\) and similarly for local height.
\end{definition}

Note that derivability implies eliminability which implies admissibility.
The fundamental lemma to show cut elimination is the following.

\begin{lemma}\label{lm:admissibility-and-eliminability}
  For any local progress sequent calculi, the following holds
  \[
    R \text{ eliminable iff }R \text{ locally eliminable iff } R \text{ locally admissible}.
  \]
\end{lemma}
\begin{proof}
  That \(R\) is eliminable trivially implies that \(R\) is locally admissible.
  To show that \(R\) locally admissible implies \(R\) locally eliminable it suffices to do an induction in the local height. 
  Finally, to show that \(R\) locally eliminable implies that \(R\) is eliminable it suffices to apply the method of translations using local eliminability to define a corecursive step.
\end{proof}

In addition we notice the following facts.
\begin{enumerate}
  \item In a local-progress calculus \(\mathcal{G}\), \(R\) is locally admissibile preserving local height implies that \(R\) is eliminable preserving height.
  \item In a wellfounded calculus \(\mathcal{G}\), \(R\) is admissibile preserving local height implies that \(R\) is eliminable preserving height.
  \item In a local-progress calculus \(\mathcal{G}\), \(R\) is locally admissibile preserving local \(R'\)-freeness implies that \(R\) is eliminable preserving local \(R'\)-freeness.
  \item In a wellfounded calculus \(\mathcal{G}\), \(R\) is admissibile preserving \(R'\)-freeness implies that \(R\) is eliminable preserving \(R'\)-freeness.
\end{enumerate}

\section{Sequent Calculi for \(\IL\)}
\label{sec:sequent-calculi}

In this section we introduce two sequent calculi for \(\IL\).
Let us introduce a useful convention for describing the rules of these calculi.
In case \(X \subseteq \mathbb{N}\) we will define the sets
\[
  \Phi_X := \{\phi_i \mid i \in X\}\qquad \text{and}\qquad \Psi_X := \{\psi_i \mid i \in X\}.
\]
In particular \(X\) will always be an interval like \((i,j)\), \([i,j]\) or \([i, j)\).

\begin{figure}
  \[
    \AxiomC{}
    \RightLabel{\(\text{ax}\)}
    \UnaryInfC{\(p,\Gamma \Rightarrow p, \Delta\)}
    \DisplayProof
    \qquad
    \AxiomC{}
    \RightLabel{\(\bot\text{L}\)}
    \UnaryInfC{\(\bot,\Gamma\Rightarrow \Delta\)}
    \DisplayProof
    \qquad
    \AxiomC{\(\Gamma\Rightarrow \Delta\)}
    \RightLabel{\(\bot\text{R}\)}
    \UnaryInfC{\(\Gamma\Rightarrow \bot,\Delta\)}
    \DisplayProof
  \]

  \[
    \AxiomC{\(\Gamma \Rightarrow \Delta, \phi\)}
    \AxiomC{\(\psi, \Gamma \Rightarrow \Delta\)}
    \RightLabel{\(\to\text{L}\)}
    \BinaryInfC{\(\phi \to \psi, \Gamma \Rightarrow \Delta\)}
    \DisplayProof
    \qquad
    \AxiomC{\(\phi, \Gamma \Rightarrow \Delta, \psi\)}
    \RightLabel{\(\to\text{R}\)}
    \UnaryInfC{\(\Gamma\Rightarrow \Delta, \phi \to \psi\)}
    \DisplayProof
  \]

  \[
    \AxiomC{\([\psi_i, (\psi_i,\Phi_{[0,i)}, \phi) \interp \bot \Rightarrow \Phi_{[0,i)}, \phi]_{m\ldots i\ldots 0}\)}
    \RightLabel{\(\interp_{\IL}\)}
    \UnaryInfC{\(\{\phi_i \interp \psi_i\}_{i < m}, \Gamma \Rightarrow \psi_m \interp \phi, \Delta\)}
    \DisplayProof
    \qquad
    \AxiomC{\([\psi_i, (\Phi_{[0,i)}, \phi) \interp \bot \Rightarrow \Phi_{[0,i)}, \phi]_{m...i...0}\)}
    \RightLabel{\(\interp_{\IKT}\)}
    \UnaryInfC{\(\{\phi_i \interp \psi_i\}_{i < m}, \Gamma \Rightarrow \psi_m \interp \phi, \Delta\)}
    \DisplayProof
  \]

  \[
    \AxiomC{\(\Gamma \Rightarrow \Delta, \chi\)}
    \AxiomC{\(\chi, \Gamma \Rightarrow \Delta\)}
    \RightLabel{Cut}
    \BinaryInfC{\(\Gamma \Rightarrow \Delta\)}
    \DisplayProof
  \]
  \caption{Sequent rules}
  \label{fig:rule}
\end{figure}

\begin{definition}
  We define the sequent calculus \(\fgentzenIL\) as the wellfounded calculus given by the rules of Figure~\ref{fig:rule} without rules \(\interp_{\IKT}\) and \(\cut\).

  We define the sequent calculus \(\gentzenIL\) as the local progress sequent calculus given by the rules of Figure~\ref{fig:rule} without rules \(\interp_{\IL}\) and \(\cut\).
  Progress only occurs at the premises of \(\interp_{\IKT}\).
\end{definition}

In the rules ax, \(\bot L\), \({\rightarrow} L\) and \({\rightarrow} R\) of Figure~\ref{fig:rule}  the explicitly displayed formula in the conclusion is called the \emph{principal formula}.
In \(\interp_{\IL}\) and \(\interp_{\IKT}\) the formula \(\psi_m\rhd\phi\) is called \emph{principal},
and multisets of formulas \(\Gamma\) and \(\Delta\) are called the \emph{weakening part} of these rules.
In \(\interp_{\IL}\) the formula \(\psi_m \interp \bot\) appearing at the left hand side of the first premise is called \emph{diagonal formula}.
The absence of diagonal formula at \(\interp_{\IKT}\) is what simplifies the treatment of \(\cut\) elimination.
The explicitly displayed formula in the \(\cut\) rule is called the cut formula.

The calculus \(\fgentzenIL\) is inspired from the calculus for \(\IKT\) in \cite{sasaki2003}.
It provides a simplification of the calculus defined there, as we are capable of give a much more concrete shape to the modal rule.
However, we notice a peculiar property of our calculus: the premises depend on an ordering of the \(\interp\)-formulas of the conclusion.
This implies that the same conclusion could have been obtained in multiple ways, depending on the ordering chosen.
The necessity of an order comes from the axiom (J2) of \(\IL\).

We want to notice that the rule \(\bot\mathrm{R}\) is a particular instance of weakening, which below we show to be eliminable.
Some readers may wonder why we add it to our calculus.
The reason is the following result:
\begin{lemma}
  Let \(\mathcal{C}\) be a local progress calculus with the rules \(\bot\mathrm{L}, \bot\mathrm{R}, {\to}\mathrm{L}, {\to}\mathrm{R}\).
  The rules
  \[
    \AxiomC{\(\Gamma \Rightarrow \phi, \Delta\)}
    \RightLabel{\(\neg\mathrm{L}\)}
    \UnaryInfC{\(\neg \phi, \Gamma \Rightarrow \Delta\)}
    \DisplayProof
    \qquad
    \AxiomC{\(\phi, \Gamma \Rightarrow \Delta\)}
    \RightLabel{\(\neg\mathrm{R}\)}
    \UnaryInfC{\(\Gamma \Rightarrow \neg \phi,\Delta\)}
    \DisplayProof
  \]
  \[
    \AxiomC{\(\phi, \Gamma \Rightarrow \Delta\)}
    \AxiomC{\(\psi, \Gamma \Rightarrow \Delta\)}
    \RightLabel{\(\vee\mathrm{L}\)}
    \BinaryInfC{\(\phi \vee \psi, \Gamma \Rightarrow \Delta\)}
    \DisplayProof
    \qquad
    \AxiomC{\(\Gamma \Rightarrow \phi,\psi,\Delta\)}
    \RightLabel{\(\vee\mathrm{R}\)}
    \UnaryInfC{\(\Gamma \Rightarrow \phi \vee \psi, \Delta\)}
    \DisplayProof
  \]
  \[
    \AxiomC{\(\phi,\psi,\Gamma \Rightarrow \Delta\)}
    \RightLabel{\(\wedge\mathrm{L}\)}
    \UnaryInfC{\(\phi \wedge \psi, \Gamma \Rightarrow \Delta\)}
    \DisplayProof
    \qquad
    \AxiomC{\(\Gamma \Rightarrow \phi, \Delta\)}
    \AxiomC{\(\Gamma \Rightarrow \psi, \Delta\)}
    \RightLabel{\(\wedge\mathrm{R}\)}
    \BinaryInfC{\(\Gamma \Rightarrow \phi \wedge \psi, \Delta\)}
    \DisplayProof
  \]
  are derivable.
\end{lemma}
\begin{proof}
  We have the following derivations
  \[
    \AxiomC{\(\Gamma \Rightarrow \Delta, \phi\)}
    \AxiomC{}
    \RightLabel{\(\bot\mathrm{L}\)}
    \UnaryInfC{\(\bot, \Gamma \Rightarrow \Delta\)}
    \RightLabel{\(\to\mathrm{L}\)}
    \BinaryInfC{\(\phi \to \bot, \Gamma \Rightarrow \Delta\)}
    \dashedLine
    \UnaryInfC{\(\neg \phi, \Gamma \Rightarrow \Delta\)}
    \DisplayProof
    \qquad
    \AxiomC{\(\phi, \Gamma \Rightarrow \Delta\)}
    \RightLabel{\(\bot\mathrm{R}\)}
    \UnaryInfC{\(\phi, \Gamma \Rightarrow \Delta, \bot\)}
    \RightLabel{\({\to}\mathrm{R}\)}
    \UnaryInfC{\(\Gamma \Rightarrow \Delta, \phi \to \bot\)}
    \dashedLine
    \UnaryInfC{\(\Gamma \Rightarrow \Delta, \neg \phi\)}
    \DisplayProof
  \]
  \[
    \AxiomC{\(\phi, \Gamma \Rightarrow \Delta\)}
    \RightLabel{\(\neg\mathrm{R}\)}
    \UnaryInfC{\(\Gamma \Rightarrow \neg \phi, \Delta\)}
    \AxiomC{\(\psi, \Gamma \Rightarrow  \Delta\)}
    \RightLabel{\({\to}\mathrm{L}\)}
    \BinaryInfC{\(\neg \phi \to \psi, \Gamma \Rightarrow \Delta\)}
    \dashedLine
    \UnaryInfC{\(\phi \vee \psi, \Gamma \Rightarrow \Delta\)}
    \DisplayProof
    \qquad
    \AxiomC{\(\Gamma \Rightarrow \phi, \psi, \Delta\)}
    \RightLabel{\(\neg \mathrm{L}\)}
    \UnaryInfC{\(\neg \phi, \Gamma \Rightarrow \psi, \Delta\)}
    \RightLabel{\({\to}\mathrm{R}\)}
    \UnaryInfC{\(\Gamma \Rightarrow \neg \phi \to \psi, \Delta\)}
    \dashedLine
    \UnaryInfC{\(\Gamma \Rightarrow \phi \vee \psi, \Delta\)}
    \DisplayProof
  \]
  \[
    \AxiomC{\(\phi, \psi, \Gamma \Rightarrow  \Delta\)}
    \RightLabel{\(\neg\mathrm{R}\)}
    \UnaryInfC{\(\phi, \Gamma \Rightarrow \neg\psi, \Delta\)}
    \RightLabel{\({\to}\mathrm{R}\)}
    \UnaryInfC{\(\Gamma \Rightarrow \phi \to \neg\psi, \Delta\)}
    \RightLabel{\({\neg}\mathrm{L}\)}
    \UnaryInfC{\(\neg(\phi \to \neg\psi), \Gamma \Rightarrow \Delta\)}
    \dashedLine
    \UnaryInfC{\(\phi \wedge \psi, \Gamma \Rightarrow \Delta\)}
    \DisplayProof
    \qquad
    \AxiomC{\(\Gamma \Rightarrow \phi, \Delta\)}
    \AxiomC{\(\Gamma \Rightarrow \psi, \Delta\)}
    \RightLabel{\(\neg\mathrm{L}\)}
    \UnaryInfC{\(\neg\psi, \Gamma \Rightarrow \Delta\)}
    \RightLabel{\({\to}\mathrm{L}\)}
    \BinaryInfC{\(\phi \to \neg\psi, \Gamma \Rightarrow \Delta\)}
    \RightLabel{\(\neg\mathrm{R}\)}
    \UnaryInfC{\(\Gamma \Rightarrow \neg(\phi \to \neg\psi), \Delta\)}
    \dashedLine
    \UnaryInfC{\(\Gamma \Rightarrow \phi \wedge \psi, \Delta\)}
    \DisplayProof
  \]
\end{proof}

The following lemma will be used in many proofs in the rest of this paper, as usual it is proven by induction on the size of \(\phi\).
When we use this lemma in a proof we will simply write Ax just as we write ax for the rule in Figure~\ref{fig:rule}.
\begin{lemma}\label{lm:axiom-sequent}
  Let \(\phi\) be a formula. Then in \(\fgentzenIL\) and in \(\gentzenIL\) we have that
  \[
    \vdash \phi, \Gamma \Rightarrow \phi, \Delta.
  \]
\end{lemma}
\begin{proof}
  This is a simple proof by induction on \(|\phi|\).
  Cases where \(\phi\) is \(\bot\) or an atomic variable are trivial, and the case where \(\phi\) is an implication is as usual.

  Assume \(\phi = \phi_0 \interp \phi_1\). Then we provide the following proof for \(\fgentzenIL\)
  \[
    \AxiomC{}
    \RightLabel{I.H.}
    \UnaryInfC{\(\phi_0 \interp \bot, \phi_0, \phi_0 \interp \bot, \phi_1 \interp \bot \Rightarrow \phi_0, \phi_1\)}
    \AxiomC{}
    \RightLabel{I.H.}
    \UnaryInfC{\(\phi_1 \interp \bot, \phi_1, \phi_1 \interp \bot \Rightarrow \phi_1\)}
    \RightLabel{\(\interp_{\IL}\)}
    \BinaryInfC{\(\Gamma, \phi_0 \interp \phi_1 \Rightarrow \phi_0 \interp \phi_1, \Delta\)}
    \DisplayProof
  \]
  and this other proof for \(\gentzenIL\)
  \[
    \AxiomC{}
    \RightLabel{I.H.}
    \UnaryInfC{\(\phi_0, \phi_0 \interp \bot, \phi_1 \interp \bot \Rightarrow \phi_0, \phi_1\)}
    \AxiomC{}
    \RightLabel{I.H.}
    \UnaryInfC{\(\phi_1, \phi_1 \interp \bot \Rightarrow \phi_1\)}
    \RightLabel{\(\interp_{\IKT}\)}
    \BinaryInfC{\(\Gamma, \phi_0 \interp \phi_1 \Rightarrow \phi_0 \interp \phi_1, \Delta\)}
    \DisplayProof
  \]
  where we applied the rule \(\interp_{\IL}\) and the rule \(\interp_{\IKT}\), respectively, with ordering \(\phi_0 \interp \phi_1\) and principal formula \(\phi_0 \interp \phi_1\).
\end{proof}

We state the eliminability of some rules that will be useful, they are proved by showing admissibility or local admissibility (depending on the calculus) which is shown by induction on the height or local height, respectively.

\begin{lemma}\label{lm:weakening}
  Let us define the weakening rule as
  \[
    \AxiomC{\(\Gamma \Rightarrow \Delta\)}
    \RightLabel{\(\mathrm{Wk}\)}
    \UnaryInfC{\(\Gamma, \Gamma' \Rightarrow \Delta, \Delta'\)}
    \DisplayProof
  \]
  Then we have that
  \begin{enumerate}
    \item \(\Wk\) is admissible in \(\fgentzenIL (+\cut)\) preserving height.
    \item \(\Wk\) is eliminable in \(\fgentzenIL (+\cut)\).
    \item \(\Wk\) is admissible in \(\gentzenIL (+\cut)\) preserving local height and local \(\cut\)-freeness.
    \item \(\Wk\) is eliminable in \(\fgentzenIL (+\cut)\).
  \end{enumerate}
\end{lemma}
\begin{proof}
  The proof of 1.\ is as usual by induction on the height of the proof, and then 2.\ follows straightforwardly.
  The proof of 3.\ is by induction on the local height of the proof, since when the local height is \(0\) we have either an axiomatic sequent or an application \(\interp_{\IKT}\) rule and in both cases we can weaken straightforwardly.
  The proof of 4.\ can be done by showing local admissibility of \(\Wk\), which can be proven again by induction on the local height.
\end{proof}

\begin{lemma}\label{lm:invertibility-of-implication}
  The rules \({\to}\text{L}\), \({\to}\text{R}\) and \(\bot\mathrm{R}\) are invertible in \(\fgentzenIL (+ \cut)\), preserving height; and in \(\gentzenIL (+\cut)\), preserving local height and local \(\cut\)-freeness.
\end{lemma}
\begin{proof}
  The proof is just by induction on the height or on the local height, depending on if we are working with \(\fgentzenIL (+\cut)\) or with \(\gentzenIL (+\cut)\).
\end{proof}

\begin{lemma}\label{lm:nec-rule}
  The rule
  \[
    \AxiomC{\(\phi, \Sigma \interp \bot \Rightarrow \Sigma\)}
    \RightLabel{\(\mathrm{Nec}\)}
    \UnaryInfC{\(\Sigma \interp \bot, \Gamma \Rightarrow \phi \interp \bot, \Delta\)}
    \DisplayProof
  \]
  is admissible in \(\fgentzenIL (+ \cut)\) and in \(\gentzenIL (+\cut)\).
\end{lemma}
\begin{proof}
  We show it for \(\fgentzenIL (+\cut)\), the other proof being similar.
  Assume \(\pi \vdash \phi, \Sigma \interp \bot \Rightarrow \Sigma\) in \(\fgentzenIL (+ \cut)\) and let us enumerate \(\Sigma\) as \(\{\phi_0,\ldots, \phi_{m-1}\}\) (note that then \(\Sigma = \Phi_{[0,m)}\)).
  Then, the desired proof for \(\fgentzenIL\) is
  \[
    \AxiomC{\(\pi\)}
    \noLine
    \UnaryInfC{\(\phi, \Phi_{[0,m)} \interp \bot \Rightarrow \Phi_{[0,m)}\)}
    \RightLabel{Wk}
    \UnaryInfC{\(\phi, (\phi,\Phi_{[0,m)}, \bot) \interp \bot \Rightarrow \Phi_{[0,m)},\bot\)}
    \AxiomC{}
    \RightLabel{\(\bot\)L}
    \UnaryInfC{\(\bot, \ldots\Rightarrow \ldots\)}
    \AxiomC{\(\cdots\)}
    \RightLabel{\(\interp_{\IL}\)}
    \TrinaryInfC{\(\Sigma \interp \bot, \Gamma \Rightarrow \phi \interp \bot, \Delta\)}
    \DisplayProof
  \]
  where in the right-most dots we are omitting some proofs by \(\bot\mathrm{L}\)
  and we applied \(\interp_{\IL}\) with ordering \(\phi_0 \interp \bot, \ldots, \phi_{m-1} \interp \bot\) and prinicipal formula \(\phi \interp \bot\).
\end{proof}

Finally, we note some nice properties of the \(\cut\)-free calculi.

\begin{proposition}\label{lm:no-branch-condition}
  Any preproof of \(\gentzenIL\) is a proof of \(\gentzenIL\).
\end{proposition}
\begin{proof}
  The rules \(\to \mathrm{L}\), \(\to \mathrm{R}\) and \(\bot\mathrm{R}\) reduce the size of the sequent (which is just the multiset of the sizes of each formula ocurrence in it).
  So any infinite branch in a preproof must have infinitely many instances of \(\interp_{\IKT}\).
\end{proof}

Due to the shape of the rules we need to slightly change the usual definition of subformula set.
This definition allows to establish the subformula property for \(\fgentzenIL\) and \(\gentzenIL\).

\begin{definition}
  Let \(\phi\) be a formula. We define the \emph{set} \(\sub(\phi)\) as follows:
  \begin{align*}
    &\sub(p) = \{p\}, \qquad \sub(\bot) = \{\bot\}, \\
    &\sub(\phi \to \psi) = \{\phi \to \psi\} \cup \sub(\phi) \cup \sub(\psi), \\
    &\sub(\phi \interp \psi) = \{\phi \interp \psi, \phi \interp \bot, \psi \interp \bot, \bot\} \cup \sub(\phi) \cup \sub(\psi).
  \end{align*}
  If \(\Gamma\) is a multiset, \(\sub(\Gamma)= \bigcup\{\sub(\phi) \mid \phi \in \Gamma\}\); and if \(S = (\Gamma \Rightarrow \Delta)\) is a sequent, then \(\sub(S) = \sub(\Gamma \cup \Delta)\).
\end{definition}

\begin{proposition}[Subformula property]\label{lem:sub}
  Let \(\pi \vdash S\) in \(\gentzenIL\) or \(\fgentzenIL\) and \(\phi\) be a formula occurring in \(\pi\).
  Then \(\phi \in \sub(S)\).
\end{proposition}
\begin{proof}
  The proof is trivial by observing the shape of the rules and using an induction in the length of the node where \(\phi\) is taken from.
\end{proof}

\section{Transformations between calculi}
\label{sec:transformations}

In this section we will show how to transform proofs between the Hilbert calculus, the wellfounded sequent calculus and the non-wellfounded sequent calculus.
In order to show the equivalence among all the system we will need to also show \(\cut\) elimination.
This last step will be done in the next section.

\subsection{Equivalence of Hilbert calculus and \(\fgentzenIL + \cut\)}
We show the equivalence of Hilbert style proofs in \(\IL\) and sequent proofs in the calculus \(\fgentzenIL+\cut\).
First we remember the interpretation of sequents as formulas.

\begin{definition}
  Given a sequent \(S = (\Gamma \Rightarrow \Delta)\), 
  we define \(\toWff{S} = \left(\bigwedge \Gamma \to \bigvee \Delta\right)\).
\end{definition}

\begin{lemma}\label{lm:from-hilbert-to-gentzen}
  Let \(\IL \vdash \phi\), then \(\fgentzenIL + \cut \vdash {\Rightarrow \phi}\).
\end{lemma}
\begin{proof}
  By induction on the length of the Hilbert-style proof of \(\phi\).
  The case of classical propositional tautologies is trivial, the proofs in \(\fgentzenIL\) of the modal axioms are easy to construct (the interested reader can consult Section~\ref{sec:completeness-of-gil} in the Appendix).
  For modus ponens case it suffices to use Lemma~\ref{lm:invertibility-of-implication} and \(\cut\).
  For necessitation case it suffices to use Lemma~\ref{lm:nec-rule}.
\end{proof}

The converse of the previous lemma is a simple consequence of the following.

\begin{theorem}\label{tm:IL-and-fgentzenILcut}
  For any sequent \(S\), \( \IL \vdash \toWff{S} \text{ if and only if } \fgentzenIL + \cut \vdash S \).
\end{theorem}
\begin{proof}
  Let \(S = (\Gamma\Rightarrow\Delta)\).
  Using Lemma~\ref{lm:from-hilbert-to-gentzen}, we have that \(\IL \vdash \toWff{S}\) implies \(\fgentzenIL + \cut \vdash {\Rightarrow \bigwedge \Gamma \to \bigvee \Delta}\). 
  Then, using invertibility of \(\to\mathrm{L}\), \(\to\mathrm{R}\) and \(\bot\mathrm{R}\), we obtain \(\fgentzenIL + \cut \vdash \Gamma \Rightarrow \Delta\).
  
  For the other direction, let \(\pi \vdash S\) in \(\fgentzenIL + \cut\). 
  We proceed by induction on the height of \(\pi\) and cases in the last rule of \(\pi\).
  The cases where the last rule of \(\pi\) is either \(\mathrm{ax}, \bot\mathrm{L}, \bot\mathrm{R}, {\to}\mathrm{L}, {\to}\mathrm{R}, \cut\) follow from simple propositional tautologies.
  So we focus on the \(\interp_{\mathrm{IL}}\) case.
  Then \(\pi\) is of shape
  \[
    \AxiomC{\(\begin{bmatrix}
    		\pi_i\\
        \psi_i, (\psi_i, \Phi_{[0,i)}, \phi) \interp \bot \Rightarrow \Phi_{[0,i)}, \phi
    		\end{bmatrix}_{m...i...0}\)}
    \RightLabel{\(\interp_{\mathrm{IL}}.\)}
    \UnaryInfC{\(\{\phi_i \interp \psi_i\}_{i < m}, \Gamma \Rightarrow \psi_m \interp \phi, \Delta\)}
    \DisplayProof
  \]
  By the induction hypothesis we get
  \[
    \IL \vdash (\psi_i \interp \bot) \wedge \psi_i \wedge \bigwedge (\Phi_{[0,i)} \interp \bot) \wedge (\phi \interp \bot) \to \bigvee \Phi_{[0,i)} \vee \phi,\qquad \text{for \(i \leq m\)},
  \]
  so by L\"ob's rule we have
  \( \IL \vdash ((\psi_i \interp \bot) \wedge \psi_i) \interp \left(\bigvee \Phi_{[0,i)} \vee \phi\right),\) or equivalently \(\IL \vdash \bnec \psi_i \interp \left(\bigvee \Phi_{[0,i)} \vee \phi\right)\) \text{for \(i \leq m\)}.
  Using Lemma~\ref{lm:basics-of-IL} with (J2) we have
  \( \IL \vdash \psi_i \interp \left(\bigvee \Phi_{[0,i)} \vee \phi\right) \), for \(i \leq m\).
  By induction on \(i \leq m\) we show that \(\IL \vdash (\bigwedge_{k < m} \phi_k \interp \psi_k) \to \psi_i \interp \phi\), 
  so assume \(\IL \vdash (\bigwedge_{k < m} \phi_k \interp \psi_k) \to \psi_j \interp \phi\), for \(j < i\).
  Using (J3) we get \(\IL \vdash (\bigwedge_{k < m} \phi_k \interp \psi_k) \to (\bigvee_{j < i} \psi_j) \interp \phi\) and by (J2) \(\IL \vdash (\bigwedge_{k < m} \phi_k \interp \psi_k) \to (\bigvee \Phi_{[0,i)}) \interp \phi\).
  Also \(\IL \vdash \phi \interp \phi\),
  so we get \(\IL \vdash (\bigwedge_{k < m} \phi_k \interp \psi_k) \to (\bigvee \Phi_{[0,i)} \vee \phi) \interp \phi\).
  But \(\IL \vdash \psi_i \interp \left(\bigvee \Phi_{[0,i)}\vee \phi\right)\) so by the use of (J2) we conclude the desired \(\IL \vdash (\bigwedge_{k < m} \phi_k \interp \psi_k) \to \psi_i \interp \phi\).
\end{proof}

\subsection{From \(\fgentzenIL + \cut\) to \(\gentzenIL + \cut\)}

We show that using L\"ob's rule (formulated in the language of \(\IL\)) we can go from wellfounded proofs to non-wellfounded proofs (assuming \(\cut\)).

\begin{lemma}
  We have that the rule
  \[
    \AxiomC{\(\psi, (\psi, \Sigma) \interp \bot \Rightarrow \Sigma\)}
    \RightLabel{L\"ob}
    \UnaryInfC{\(\psi, \Sigma \interp \bot \Rightarrow \Sigma\)}
    \DisplayProof
  \]
  is admissible in \(\fgentzenIL + \cut\).
\end{lemma}
\begin{proof}
  Let \(\pi \vdash \psi, (\psi, \Sigma) \interp \bot \Rightarrow \Sigma\) in \(\fgentzenIL + \cut\).
  By admissibility of weakening we obtain a proof \(\pi' \vdash \psi, (\psi, \Sigma,\bot) \interp \bot \Rightarrow \Sigma, \bot\).
  The desired proof is
  \[
    \AxiomC{\(\pi'\)}
    \noLine
    \UnaryInfC{\(\psi, (\psi, \Sigma,\bot) \interp \bot \Rightarrow \Sigma, \bot\)}
    \AxiomC{}
    \RightLabel{\(\botL\)}
    \UnaryInfC{\(\bot, \ldots \Rightarrow \ldots\)}
    \AxiomC{\(\ldots\)}
    \RightLabel{\(\interpIL\)}
    \TrinaryInfC{\(\psi, \Sigma \interp \bot \Rightarrow \Sigma, \psi \interp \bot\)}
    \AxiomC{\(\pi\)}
    \noLine
    \UnaryInfC{\(\psi, (\psi, \Sigma) \interp \bot \Rightarrow \Sigma\)}
    \RightLabel{\(\cut\)}
    \BinaryInfC{\(\psi, \Sigma \interp \bot \Rightarrow \Sigma\)}
    \DisplayProof
  \]
  where \(\interpIL\) has been applied with any order of the formulas \(\Sigma \interp \bot\) and \(\psi \interp \bot\) as the main formula.
  The premises of that rule instance hidden in the ellipsis are proven using the \(\bot\mathrm{L}\) rule.
\end{proof}

\begin{theorem}\label{tm:fgentzenILcut-to-genteznILcut}
  Let \(S\) be a sequent.
  If \(\fgentzenIL + \cut \vdash S\), then \(\gentzenIL +\cut \vdash S\).
\end{theorem}
\begin{proof}
  We define a function \(\alpha\) from proofs in \(\fgentzenIL + \cut\) to proofs in \(\gentzenIL + \cut\) that preserves the conclusion of the proof.
  The definition is done via corecursion and case analysis on the last rule of the input proof.
  \(\alpha\) will commute with all the rules except for \(\interpIL\), i.e., if \(R\) is a rule different from \(\interpIL\) we will have that
  \[
    \AxiomC{\(\pi_0\)}
    \noLine
    \UnaryInfC{\(S_0\)}
    \AxiomC{\(\cdots\)}
    \AxiomC{\(\pi_{n-1}\)}
    \noLine
    \UnaryInfC{\(S_{n-1}\)}
    \RightLabel{\(R\)}
    \TrinaryInfC{\(S\)}
    \DisplayProof
    {\qquad \overset{\alpha}{\longmapsto} \quad}
    \AxiomC{\(\alpha(\pi_0)\)}
    \noLine
    \UnaryInfC{\(S_0\)}
    \AxiomC{\(\cdots\)}
    \AxiomC{\(\alpha(\pi_{n-1})\)}
    \noLine
    \UnaryInfC{\(S_{n-1}\)}
    \RightLabel{\(R\)}
    \TrinaryInfC{\(S\)}
    \DisplayProof
  \]
  And in case the last rule of the input is \(\interpIL\) then
  \[
    \AxiomC{\(\begin{bmatrix}\pi_i\\
      \psi_i, (\psi_i, \Phi_{[0,i)}, \phi) \interp \bot \Rightarrow \Phi_{[0,i)}, \phi
		  \end{bmatrix}_{m...i...0}\)}
    \RightLabel{\(\interpIL\)}
    \UnaryInfC{\(\{\phi_i \interp \psi_i\}_{i < m}, \Gamma \Rightarrow \Delta, \psi_m \interp \phi\)}
    \DisplayProof
    {\quad \overset{\alpha}{\longmapsto} \quad}
    \AxiomC{\(\begin{bmatrix}\alpha(\textsf{l\"ob}(\pi_i))\\
      \psi_i, (\Phi_{[0,i)}, \phi) \interp \bot \Rightarrow \Phi_{[0,i)}, \phi
		  \end{bmatrix}_{m...i...0}\)}
    \RightLabel{\(\interpIKT\)}
    \UnaryInfC{\(\{\phi_i \interp \psi_i\}_{i < m}, \Gamma \Rightarrow \Delta, \psi_m \interp \phi\)}
    \DisplayProof
  \]
  It is clear that if \(\pi\) is a proof in \(\fgentzenIL + \cut\) then \(\alpha(\pi)\) is a preproof in \(\gentzenIL + \cut\).
  We notice that it is a proof, since in each corecursive call either the height of the input tree is smaller (in case the last rule of the input proof is not \(\interpIL\)) or we introduce progress via the application of the rule \(\interpIKT\) (in case the last rule of the input proof is \(\interpIL\)).
\end{proof}

\subsection{From \(\gentzenIL\) to \(\fgentzenIL\)}
Using the subformula property we can also transform non-wellfounded proofs into wellfounded proofs.
For this step the absence of the \(\cut\) rule is fundamental, as otherwise we would lack the necessary subformula property.

\begin{theorem}\label{thm:gentzenIL-to-fgentzenIL}
  For any \(\Lambda\) finite set of formulas, 
  \( \gentzenIL \vdash \Gamma \Rightarrow \Delta\) implies \(\fgentzenIL \vdash \Lambda \interp \bot, \Gamma \Rightarrow \Delta\).
\end{theorem}
\begin{proof}
  Let \(\pi \vdash \Gamma \Rightarrow \Delta\) in \(\gentzenIL\).
  By induction on the lexicographical order \(\left(|\sub(\Gamma \Rightarrow \Delta) \setminus \Lambda |, \lheight(\pi)\right)\) and the case analysis in the last rule of \(\pi\).\footnote{Note that in the presence of \(\cut\) this measure would not work.}
  The only interesting case is when the
  last rule of \(\pi\) is \(\interp_{\mathrm{IK4}}\).
  So \(\pi\) is of shape
  \[
    \AxiomC{\(\begin{bmatrix}\pi_i\\
        \psi_i, (\Phi_{[0,i)}, \phi) \interp \bot \Rightarrow \Phi_{[0,i)}, \phi
    		\end{bmatrix}_{m...i...0}\)}
    \RightLabel{\(\interp_{\mathrm{IK4}}\)}
    \UnaryInfC{\(\{\phi_i \interp \psi_i\}_{i < m}, \Gamma \Rightarrow \psi_m \interp \phi, \Delta\)}
    \DisplayProof
  \]
  and let us denote the conclusion of \(\pi_i\) as \(S_i\) and the conclusion of \(\pi\) as \(S\).
  We want to show that \( \fgentzenIL \vdash \Lambda \interp \bot, \{\phi_i \interp \psi_i\}_{i < m}, \Gamma \Rightarrow \psi_m \interp \phi, \Delta \).
  For \(i \leq m\) we define proofs \(\tau_i \vdash \psi_i, (\psi_i, \Phi_{[0,i)}, \Lambda, \phi) \interp \bot \Rightarrow \Phi_{[0,i)}, \Lambda, \phi\)
  so the desired proof is
  \[
    \AxiomC{\(\tau_m \quad \cdots \quad \tau_0 \quad \rho_{n-1} \quad \cdots \quad \rho_0\)}
    \RightLabel{\(\interp_{\mathrm{IL}}\)}
    \UnaryInfC{\(\Lambda \interp \bot, \{\phi_i \interp \psi_i\}_{i < m}, \Gamma \Rightarrow \psi_m \interp \phi, \Delta\)}
    \DisplayProof
  \]
  where \(\Lambda = \{\chi_0,\ldots,\chi_{n-1}\}\) and \(\interp_{\IL}\) was applied with the ordering
  \[
    \chi_0 \interp \bot, \ldots, \chi_{n-1} \interp \bot, \phi_0 \interp \psi_0, \ldots, \phi_{m-1} \interp \psi_{m-1}
  \] 
  and principal formula \(\psi_m \interp \phi\).
  Let us define the \(\tau_i\)'s and \(\rho_j\)'s.

  First, we define the \(\tau_i\)'s by cases.
  Case 1.
  If \(\psi_i \in \Lambda\) then we define \(\tau_i\) as
      \[
        \AxiomC{}
        \RightLabel{\(\mathrm{Ax}\)}
        \UnaryInfC{\(\psi_i, (\psi_i, \Phi_{[0,i)}, \Lambda, \phi) \interp \bot \Rightarrow \Phi_{[0,i)}, \Lambda, \phi\)}
        \DisplayProof
      \]
      since the formula \(\psi_i\) appears on both sides of this sequent.

  Case 2.
	If \(\psi_i \not \in \Lambda\) then, since \(\psi_i \in \sub(S_i)\) and \(\sub(S_i) \subseteq \sub(S)\), we have
      \( |\sub(S_i) \setminus (\Lambda \cup \{\psi_i\})| < |\sub(S_i) \setminus \Lambda| \leq |\sub(S) \setminus \Lambda| \).
      So by induction hypothesis applied to \(\pi_i\) with set \(\Lambda \cup \{\psi_i\}\) we obtain a proof \(\pi'_i\) in \(\fgentzenIL\) such that \( \pi'_i \vdash \psi_i, (\psi_i, \Phi_{[0,i)}, \Lambda, \phi) \interp \bot\Rightarrow \Phi_{[0,i)}, \phi \).
      We define \(\tau_i\) applying Wk to \(\pi'_i\) so \(\tau_i \vdash \psi_i, (\psi_i,\Phi_{[0,i)}, \Lambda, \phi) \interp \bot \Rightarrow \Phi_{[0,i)}, \Lambda, \phi\) in \(\fgentzenIL\).
  Finally, we define \(\rho_j\) for \(j < n\) as the following proof in \(\fgentzenIL\)
  \[
    \AxiomC{}
    \RightLabel{\(\bot \text{L.}\)}
    \UnaryInfC{\(\bot, (\bot, \Lambda_{[0,j)},\phi) \interp \bot \Rightarrow \Lambda_{[0,j)}, \phi\)}
    \DisplayProof
  \]
\end{proof}

As a trivial corollary setting \(\Lambda = \varnothing\) we get

\begin{corollary}\label{cr:from-gentzen-to-fgentzen}
  If \(\gentzenIL \vdash S\) then \(\fgentzenIL \vdash S\).
\end{corollary}

\section{Cut elimination}
\label{sec:cut-elim}

Finally, to show the equivalence of all the system defined up to this stage, we need to show \(\cut\) elimination.
We will prove \(\cut\) elimination for \(\gentzenIL\), as the shape of the rule \(\interpIKT\) makes it easier to eliminate \(\cut\) than with the rule \(\interpIL\).
However, thanks to the transformations defined in the previous section \(\cut\) elimination for \(\fgentzenIL\) will just simply be a corollary.

To make the proof simpler, we will show eliminability of contraction.

\begin{lemma}\label{lm:contraction}
  Define the rule \(\ctr\) as
  \[
    \AxiomC{\(\Gamma, \Gamma', \Gamma' \Rightarrow \Delta, \Delta', \Delta'\)}
    \RightLabel{\(\ctr\)}
    \UnaryInfC{\(\Gamma, \Gamma' \Rightarrow \Delta, \Delta'\)}
    \DisplayProof
  \]
  Then \(\ctr\) is eliminable in \(\gentzenIL (+ \cut)\) preserving local \(\cut\)-freeness.
\end{lemma}
\begin{proof}
  We are going to show that \(\ctr\) is locally admissible without introducing any new cuts, obtaining the preservativity condition.
  To make it simpler we will assume that we want to contract only one formula one time, the general case can be treated similarly.
  We proceed by induction on the local height of the proof and cases in the last rule applied.
  The only interesting case\footnote{The rest of the cases are managed as usual, applying inversion if necessary.} is when \(\pi\) is of shape
  \[
    \AxiomC{\(\begin{bmatrix}\pi_i\\
      \psi_i, (\Phi_{[0,i)}, \phi) \interp \bot \Rightarrow \Phi_{[0,i)}, \phi
		  \end{bmatrix}_{m...i...0}\)}
    \RightLabel{\(\interp_{\mathrm{IK4}}\)}
    \UnaryInfC{\(\{\phi_i \interp \psi_i\}_{i < m}, \Gamma \Rightarrow \Delta, \psi_m \interp \phi\)}
    \DisplayProof
  \]
  and both formulas we desire to contract occur in the conclusion at \(\{\phi_i \interp \psi_i\}_{i < m}\).
  So there are \(j < k < m\) such that \(\phi_j \interp \psi_j = \phi_k \interp \psi_k\) and we want to show that the sequent
  \(
    \{\phi_i \interp \psi_i\}_{i < m, i \neq k}, \Gamma \Rightarrow \Delta, \psi_m \interp \phi
  \)
  is provable.
  For each \(i > k\) define the proof \(\rho_i\) in \(\gentzenIL (+\cut) + \ctr\) as
  \footnote{Remember that we will use dashed lines to reexpress sequents. This is just a notation and does not affect the structure of the proof.}
  \[
    \AxiomC{\(\pi_i\)}
    \noLine
    \UnaryInfC{\(\psi_i, (\Phi_{[0,i)}, \phi) \interp \bot \Rightarrow \Phi_{[0,i)}, \phi\)}
    \dashedLine
    \UnaryInfC{\(\psi_i, (\Phi_{[k+1,i)}, \phi_k, \Phi_{[0,k)}, \phi) \interp \bot \Rightarrow \Phi_{[k+1,i)}, \phi_k, \Phi_{[0,k)}, \phi\)}
    \RightLabel{\(\ctr\)}
    \UnaryInfC{\(\psi_i, (\Phi_{[k+1,i)}, \Phi_{[0,k)}, \phi) \interp \bot \Rightarrow \Phi_{[k+1,i)}, \phi_k, \Phi_{[0,k)}, \phi\)}
    \RightLabel{\(\ctr\)}
    \UnaryInfC{\(\psi_i, (\Phi_{[k+1,i)},\Phi_{[0,k)}, \phi) \interp \bot \Rightarrow \Phi_{[k+1,i)}, \Phi_{[0,k)}, \phi\)}
    \DisplayProof
  \]
  where in order to apply \(\ctr\) we used that \(\phi_k = \phi_j \in \Phi_{[0,k)}\).
  Then, the desired proof (which is trivially locally \(\ctr\)-free), is
  \[
    \AxiomC{\(\rho_m \quad \cdots \quad \rho_{k+1} \quad \pi_{k-1} \quad \cdots \quad \pi_0\)}
    \RightLabel{\(\interp_{\mathrm{IK4}}\)}
    \UnaryInfC{\(\{\phi_i \interp \psi_i\}_{i < m, i \neq k}, \Gamma \Rightarrow \psi_m \interp \phi, \Delta\)}
    \DisplayProof
  \]
  where \(\interp_{\mathrm{IK4}}\) has been applied with ordering \(\phi_0 \interp \psi_0, \ldots, \phi_{k-1} \interp \psi_{k-1}, \phi_{k+1} \interp \psi_{k+1}, \ldots, \phi_{m-1} \interp \psi_{m-1}\) and principal formula \(\psi_m \interp \phi\).
\end{proof}

\begin{theorem}[Local \(\cut\)-admissibility]\label{th:local-cut-admissibility}
  Assume we have proofs \(\pi \vdash \Gamma \Rightarrow \Delta, \chi\) and \(\tau \vdash \chi, \Gamma \Rightarrow \Delta\) in \(\gentzenIL + \cut\)  which are locally \(\cut\)-free.
  Then there is \(\rho \vdash \Gamma \Rightarrow \Delta\) in \(\gentzenIL + \cut\) which is locally \(\cut\)-free.
\end{theorem}
\begin{proof} 
  By induction on the lexicographic order of the pairs \(\left(|\chi|, \lheight(\pi) + \lheight(\tau)\right)\), i.e., the size of the formula and the sum of the local heights of \(\pi\) and \(\tau\).

  Case 1: either \(\pi\) or \(\tau\) is axiomatic.
  Assume \(\tau\) is axiomatic, the case where \(\pi\) is axiomatic is analogous.
  In case \(\Gamma \Rightarrow \Delta\) is an axiomatic sequent the desired proof is trivial, so assume it is not.
  This means that the cut formula must play a fundamental role in the axiomatic character of \(\tau\), i.e., in case the rule is \(\bot\mathrm{L}\) the cut formula must be \(\bot\) and in case the rule is \(\mathrm{ax}\) the cut formula must be the repeated propositional variable.
  There are two subcases.

  Subcase \(\chi = \bot\). Then \(\pi\) and \(\tau\) are respectively of shape
  \[
    \AxiomC{\(\pi\)}
    \noLine
    \UnaryInfC{\(\Gamma \Rightarrow \Delta,\bot\)}
    \DisplayProof
    \qquad
    \AxiomC{}
    \RightLabel{\(\bot\mathrm{L}\)}
    \UnaryInfC{\(\bot, \Gamma \Rightarrow \Delta\)}
    \DisplayProof
  \]
  where both are locally \(\cut\)-free.
  Then, we can use invertibility of \(\bot\mathrm{R}\) preserving local \(\cut\)-freeness in \(\pi\) to obtain the desired proof.

  Subcase \(\chi = p\).
  Then \(\pi\) and \(\tau\) are respectively of shape
  \[
    \AxiomC{\(\pi\)}
    \noLine
    \UnaryInfC{\(\Gamma \Rightarrow \Delta,p\)}
    \DisplayProof
    \qquad
    \AxiomC{}
    \RightLabel{ax}
    \UnaryInfC{\(p, \Gamma \Rightarrow \Delta_0, p\)}
    \DisplayProof
  \]
  where \(\Delta = \Delta_0, p\) and \(\pi \vdash \Gamma \Rightarrow \Delta, p, p\) locally cut-free.
  Then using Lemma~\ref{lm:contraction} on \(\pi\) we can obtain the desired proof by contracting \(p\) preserving local cut-freeness.

  Case 2: the last rule of \(\pi\) or the last rule of \(\tau\) is \(\bot\mathrm{R}\).
  We will prove the case when the last rule of \(\pi\) is \(\tau\), the other case being analogous.
  First, let us assume that the principal formula of \(\pi\) is not the cut formula.
  So \(\pi\) and \(\tau\) will be of the following shape
  \[
    \AxiomC{\(\pi_0\)}
    \noLine
    \UnaryInfC{\(\Gamma \Rightarrow \Delta_0, \chi\)}
    \RightLabel{\(\bot\mathrm{R}\)}
    \UnaryInfC{\(\Gamma \Rightarrow \Delta_0, \bot, \chi\)}
    \DisplayProof
    \qquad
    \AxiomC{\(\tau\)}
    \noLine
    \UnaryInfC{\(\chi, \Gamma \Rightarrow \Delta_0, \bot\)}
    \DisplayProof
  \]
  Then apply inversion of \(\bot\mathrm{R}\) on \(\tau\) obtaining a \(\tau' \vdash \chi, \Gamma \Rightarrow \Delta_0\) which is locally cut-free and whose local height has not increased.
  We obtain the desired proof by applying the I.H., with the same cut formila and smaller sum of local heights, on \(\pi_0\) and \(\tau'\).

  Now, assume the principal formula of \(\pi\) is the cut formula.
  So \(\pi\) and \(\tau\) will be of the following shape
  \[
    \AxiomC{\(\pi_0\)}
    \noLine
    \UnaryInfC{\(\Gamma \Rightarrow \Delta\)}
    \RightLabel{\(\bot\mathrm{R}\)}
    \UnaryInfC{\(\Gamma \Rightarrow \Delta, \bot\)}
    \DisplayProof
    \qquad
    \AxiomC{\(\tau\)}
    \noLine
    \UnaryInfC{\(\bot, \Gamma \Rightarrow \Delta\)}
    \DisplayProof
  \]
  Then, the desired proof is \(\pi_0\).

  Case 3: principal cut reduction (in \(\pi\) and \(\tau\) the cut formula is principal).
  The only formula that can be principal on the left side and on the right side of sequents are implications.
  Then \(\pi\) and \(\tau\) are of shape
  \[
    \AxiomC{\(\pi_0\)}
    \noLine
    \UnaryInfC{\(\chi_0, \Gamma \Rightarrow \Delta, \chi_1\)}
    \RightLabel{\({\to}\mathrm{R}\)}
    \UnaryInfC{\(\Gamma \Rightarrow \Delta, \chi_0 \to \chi_1\)}
    \DisplayProof
    \qquad
    \AxiomC{\(\tau_0\)}
    \noLine
    \UnaryInfC{\(\Gamma \Rightarrow \Delta, \chi_0\)}
    \AxiomC{\(\tau_1\)}
    \noLine
    \UnaryInfC{\(\chi_1, \Gamma \Rightarrow \Delta\)}
    \RightLabel{\({\to}\mathrm{L}\)}
    \BinaryInfC{\(\chi_0 \to \chi_1, \Gamma \Rightarrow \Delta\)}
    \DisplayProof
  \]
  We can apply the admissibility of weakening to obtain a proof \(\tau'_0 \vdash \Gamma \Rightarrow \Delta, \chi_1, \chi_0\) which is locally cut-free.
  Since \(|\chi_i| < |\chi_0 \to \chi_1|\) for \(i \in \{0,1\}\) we have that we can apply the induction hypothesis on \(\tau'_0, \pi_0\) with cut formula \(\chi_0\) obtaining a \(\rho_0 \vdash \Gamma \Rightarrow \Delta, \chi_1\) which is locally cut-free.
  Then we can apply the induction hypothesis on \(\rho_0, \tau_1\) with cut formula \(\chi_1\) obtaining a \(\rho_1 \vdash \Gamma \Rightarrow \Delta\) which is locally cut-free, as desired.

  Case 4: the cut formula is not principal in either \(\pi\) or \(\tau\) and the principal formula is an implication.
  We will assume that in \(\pi\) the cut formula is not principal, the case for \(\tau\) is analogous.
  There are two subcases, depending on the last rule applied to \(\pi\).

  Subcase \({\to}\mathrm{R}\).
  Then \(\pi\) and \(\tau\) are of shape
  \[
    \AxiomC{\(\pi_0\)}
    \noLine
    \UnaryInfC{\(\phi, \Gamma \Rightarrow \Delta, \psi, \chi\)}
    \RightLabel{\({\to}\mathrm{R}\)}
    \UnaryInfC{\(\Gamma \Rightarrow \Delta, \phi \to \psi, \chi\)}
    \DisplayProof
    \qquad
    \AxiomC{\(\tau\)}
    \noLine
    \UnaryInfC{\(\chi, \Gamma \Rightarrow \Delta, \phi \to \psi\)}
    \DisplayProof
  \]
  Applying Lemma~\ref{lm:invertibility-of-implication} to \(\tau\) we can obtain a proof \(\tau_0 \vdash \chi, \phi, \Gamma \Rightarrow \Delta, \psi\) which is also locally \(\cut\)-free and \(\lheight(\tau_0) \leq \lheight(\tau)\).
  Since \(\lheight(\pi_0) < \lheight(\pi)\) we can apply the induction hypothesis to \(\pi_0\) and \(\tau\) with cut formula \(\chi\), obtaining a locally \(\cut\)-free proof \(\rho \vdash \phi, \Gamma \Rightarrow \Delta, \psi\).
  The desired proof is
  \[
    \AxiomC{\(\rho\)}
    \noLine
    \UnaryInfC{\(\phi, \Gamma \Rightarrow \Delta, \psi\)}
    \RightLabel{\({\to}\mathrm{R}\)}
    \UnaryInfC{\(\Gamma \Rightarrow \Delta, \phi \to \psi\)}
    \DisplayProof
  \]

  Subcase \({\to}\mathrm{L}\).
  Then \(\pi\) and \(\tau\) are of shape
  \[
    \AxiomC{\(\pi_0\)}
    \noLine
    \UnaryInfC{\(\Gamma \Rightarrow \Delta, \phi, \chi\)}
    \AxiomC{\(\pi_1\)}
    \noLine
    \UnaryInfC{\(\psi, \Gamma \Rightarrow \Delta, \chi\)}
    \RightLabel{\({\to}\mathrm{L}\)}
    \BinaryInfC{\(\phi \to \psi, \Gamma \Rightarrow \Delta, \chi\)}
    \DisplayProof
    \qquad
    \AxiomC{\(\tau\)}
    \noLine
    \UnaryInfC{\(\chi, \phi \to \psi, \Gamma \Rightarrow \Delta\)}
    \DisplayProof
  \]
  Applying Lemma~\ref{lm:invertibility-of-implication} to \(\tau\) we can obtain a proofs \(\tau_0 \vdash \chi, \Gamma \Rightarrow \Delta, \phi\) and \(\tau_1 \vdash \chi, \psi, \Gamma \Rightarrow \Delta\) which is also locally \(\cut\)-free and \(\lheight(\tau_i) \leq \lheight(\tau)\) for \(i \in \{0,1\}\).
  Since \(\lheight(\pi_0) < \lheight(\pi)\) we can apply the induction hypothesis to \(\pi_i\) and \(\tau_i\) with cut formula \(\chi\), obtaining a locally \(\cut\)-free proofs \(\rho_0 \vdash \Gamma \Rightarrow \Delta, \phi\) and \(\rho_1 \vdash \psi, \Gamma \Rightarrow \Delta\).
  The desired proof is
  \[
    \AxiomC{\(\rho_0\)}
    \noLine
    \UnaryInfC{\(\Gamma \Rightarrow \Delta, \phi\)}
    \AxiomC{\(\rho_1\)}
    \noLine
    \UnaryInfC{\(\psi, \Gamma \Rightarrow \Delta\)}
    \RightLabel{\({\to}\mathrm{L}\)}
    \BinaryInfC{\(\phi \to \psi, \Gamma \Rightarrow \Delta\)}
    \DisplayProof
  \]

  Case 5: the cut formula is not principal in either \(\pi\) or \(\tau\) and the principal formula is an \(\interp\)-formula.
  If the cut formula is not principal in \(\pi\) and the principal formula an \(\interp\)-formula, then the cut formula belongs to the weakening part of \(\pi\).
  So \(\pi\) is of shape
  \[
    \AxiomC{\(\begin{bmatrix}\pi_i\\
        \psi_i, (\Phi_{[0,i)}, \phi) \interp \bot \Rightarrow \Phi_{[0,i)}, \phi
    		\end{bmatrix}_{m...i...0}\)}
    \RightLabel{\(\interp_{\mathrm{IK4}}\)}
    \UnaryInfC{\(\{\phi_i \interp \psi_i\}_{i < m}, \Gamma \Rightarrow \Delta,  \chi,\psi_m \interp \phi\)}
    \DisplayProof
  \]
  In this case we just modify the weakening part of \(\pi\) to eliminate the cut formula, i.e., the desired proof is simply
  \[
    \AxiomC{\(\begin{bmatrix}\pi_i\\
        \psi_i, (\Phi_{[0,i)}, \phi) \interp \bot \Rightarrow \Phi_{[0,i)}, \phi
    		\end{bmatrix}_{m...i...0}\)}
    \RightLabel{\(\interp_{\mathrm{IK4}}\)}
    \UnaryInfC{\(\{\phi_i \interp \psi_i\}_{i < m}, \Gamma \Rightarrow \Delta, \psi_m \interp \phi\)}
    \DisplayProof
  \]

  So assume the \(\cut\) formula is not principal in \(\tau\) and the principal formula is an \(\interp\)-formula.
  If the cut formula belongs to the weakening of \(\tau\) to obtain the desired proof we just need to modify the weakening part of \(\tau\).
  So We can assume that the \(\cut\) formula does not occur in the weakening part of \(\tau\).
  This implies that \(\chi = \chi_0 \interp \chi_1\) for some formulas \(\chi_0\)
  and \(\chi_1\).

  Then, if the last rule of \(\pi\) is \(\to \text{L}\) or \(\to \text{R}\) we would be in Case 4, and it cannot be \(\text{ax}\), \(\bot \text{L}\) since we would be in Case 1 or \(\bot\mathrm{R}\) since we would be in Case 2.
  The only possibility left is that the last rule of \(\pi\) is \(\interp_{\mathrm{IK4}}\), and we can assume that the cut formula is principal in \(\pi\) (otherwise it would belong to the weakening part and again we would just eliminate it from the weakening part of the rule instance).
  
  So both proofs end in an application of \(\interp_{\textsf{IK4}}\), the cut formula is principal in \(\pi\) and occurs in the ordering used in \(\tau\).
  Then \(\pi\) and \(\tau\) are of the following shape:
  \[
    \AxiomC{\(\begin{bmatrix}\pi_i\\
        \psi_i, (\Phi_{[0,i)}, \phi) \interp \bot \Rightarrow \Phi_{[0,i)}, \phi
    		\end{bmatrix}_{m...i...0}\)}
    \RightLabel{\(\interp_{\mathrm{IK4}}\)}
    \UnaryInfC{\(\{\phi_i \interp \psi_i\}_{i < m}, \Gamma_0 \Rightarrow \Delta_0, \psi_m \interp \phi\)}
    \DisplayProof
  \]
  \[
    \AxiomC{\(\begin{bmatrix}\tau_i\\
        \psi'_j, (\Phi'_{[0,j)}, \phi') \interp \bot \Rightarrow \Phi'_{[0,j)}, \phi'
    		\end{bmatrix}_{n...j...0}\)}
    \RightLabel{\(\interp_{\mathrm{IK4}}\),}
    \UnaryInfC{\(\{\phi'_j \interp \psi'_j\}_{j < n}, \Gamma_1 \Rightarrow \Delta_1, \psi'_n \interp \phi'\)}
    \DisplayProof
  \]
  where \(\chi = \chi_0 \interp \chi_1 = \psi_m \interp \phi = \phi'_k \interp \psi'_k \) for some \(k < n\)
  and  
  \[
    \tag{i}
    \big(\{\phi_i \interp \psi_i\}_{i < m}, \Gamma_0 \Rightarrow \Delta_0 \big)
    = \big(\{\phi'_j \interp \psi'_j\}_{j < n, j \neq k}, \Gamma_1 \Rightarrow \Delta_1, \psi'_n \interp \phi'\big).
  \]

  Let us write \(\Sigma = \{\phi_i \interp \psi_i\}_{i < m}\) and \(\Sigma' = \{\phi'_j \interp \psi'_j\}_{j < n, j \neq k}\).
  Define \(\Gamma_2 := \Gamma_0 \setminus (\Sigma' \setminus \Sigma) = \Gamma_1 \setminus (\Sigma \setminus \Sigma')\), where the equality holds thanks to (i).
  Then
  \[
    \Gamma_2, \Sigma \cap \Sigma', \Sigma \setminus \Sigma', \Sigma' \setminus \Sigma = \{\phi_i \interp \psi_i\}_{i < m}, \Gamma_0 = \{\phi'_j \interp \psi'_j\}_{j < n, j \neq k}, \Gamma_1.
  \]
  We also notice that contracting \(\Gamma_2, \Sigma,\Sigma' \Rightarrow \Delta_1, \psi'_n \interp \phi'_n\) we can obtain the desired sequent.
  Let us define proofs \((\rho_i)_{i < m}, (\rho'_j)_{j \leq n, j \neq k}\) such that
  \begin{align*}
    &\rho'_j \vdash \psi'_j, (\Phi'_{(k, j)}, \Phi_{[0,m)}, \Phi'_{[0,k)}, \phi') \interp \bot \Rightarrow \Phi'_{(k, j)}, \Phi_{[0,m)}, \Phi'_{[0,k)}, \phi', \text{ for }k < j \leq n,\\
    &\rho_i \vdash \psi_i, (\Phi_{[0,i)}, \Phi'_{[0,k)}, \phi') \interp \bot \Rightarrow \Phi_{[0,i)}, \Phi'_{[0,k)}, \phi', \text{ for }i < m,\\
    &\rho'_j \vdash \psi'_j, (\Phi'_{[0,j)}, \phi') \interp \bot \Rightarrow \Phi'_{[0,j)}, \phi', \text{ for }j < k.
  \end{align*}
  Then we get the following (locally \(\cut\)-free) proof \(\rho\)
  \[
    \AxiomC{\(\rho'_n \quad\cdots \quad\rho'_{k + 1} \quad \rho_{m-1} \quad\cdots \quad\rho_0 \quad \rho'_{k-1} \quad\cdots \quad\rho'_0\)}
    \RightLabel{\(\interp_{\mathrm{IK4}}\)}
    \UnaryInfC{\(\{\phi_i \interp \psi_i\}_{i < m}, \{\phi'_j \interp \psi'_j\}_{j < n, j \neq k}, \Gamma_2 \Rightarrow \Delta', \psi'_n \interp \phi'\)}
    \dashedLine
    \UnaryInfC{\(\Sigma, \Sigma', \Gamma_2 \Rightarrow \Delta_1, \psi'_n \interp \phi'_n\)}
    \DisplayProof
  \]
  where \(\interp_{\mathrm{IK4}}\) is applied with ordering
  \[
    \phi'_0 \interp \psi'_0, \ldots, \phi'_{k-1} \interp \psi'_{k-1},
    \phi_0 \interp \psi_0, \ldots, \phi_{m-1} \interp \psi_{m-1},
    \phi'_{k+1} \interp \psi'_{k+1}, \ldots, \phi'_{n-1} \interp \psi'_{n-1}
  \]
  and main formula \(\psi'_n \interp \phi'\).
  The desired proof will be obtained by applying contraction, i.e., Lemma~\ref{lm:contraction} to \(\rho\) as contraction preserves local \(\cut\)-freeness.
  Note that while defining the \(\rho_i\)s and \(\rho'_j\)s we can use \(\cut\) rule, as in the final proof it will occur outside the main local fragment.

  We define \(\rho'_j\) for \(j < k\) as \(\tau_j\), so we only need to define \(\rho'_j\) for \(k < j \leq n\) and \(\rho_i\) for \(i < m\).
  To define \(\rho'_j\) for \(k < j \leq n\) we notice we have the following proofs:
  \begin{align*}
    &\tau_j \vdash \psi'_j, (\Phi'_{(k,j)}, \chi_0 ,\Phi'_{[0,k)}, \phi') \interp \bot \Rightarrow \Phi'_{(k,j)}, \chi_0, \Phi'_{[0,k)}, \phi',\\
    &\pi_m \vdash \chi_0, (\Phi_{[0,m)}, \chi_1) \interp \bot \Rightarrow \Phi_{[0,m)}, \chi_1, \\
    &\tau_k \vdash \chi_1, (\Phi_{[0,k)}, \phi') \interp \bot \Rightarrow \Phi'_{[0,k)}, \phi'.
  \end{align*}
  Applying Lemma~\ref{lm:nec-rule} to \(\pi_m\) and to \(\tau_k\) we obtain proofs \(\pi'_m\) and \(\tau'_k\) such that
  \( \pi'_m \vdash (\Phi_{[0,m)}, \chi_1) \interp \bot \Rightarrow \chi_0 \interp \bot\), \(\tau'_k \vdash (\Phi'_{[0,k)}, \phi') \interp \bot \Rightarrow \chi_1 \interp \bot \).
  Then the desired proof \(\rho'_j\) is (where \(\mathrm{wk}\) indicates an application of admissibility of weakening)
  {\footnotesize
  \[
    \AxiomC{\(\wk(\tau'_k)\)}
    \AxiomC{\(\wk(\pi'_m)\)}
    \AxiomC{\(\wk(\tau_j)\)}
    \LeftLabel{\(\chi_0 \interp \bot\)}
    \RightLabel{\(\cut\)}
    \BinaryInfC{\(\psi'_j, (\Phi, \chi_1)\interp \bot \Rightarrow \Phi, \chi_0, \chi_1\)}
    \AxiomC{\(\wk(\pi_m)\)}
    \LeftLabel{\(\chi_0\)}
    \RightLabel{\(\cut\)}
    \BinaryInfC{\(\psi'_j, (\Phi, \chi_1)\interp \bot \Rightarrow \Phi, \chi_1\)}
    \LeftLabel{\(\chi_1 \interp \bot\)}
    \RightLabel{\(\cut\)}
    \BinaryInfC{\(\psi'_j, \Phi\interp \bot \Rightarrow \Phi, \chi_1\)}
    \AxiomC{\(\wk(\tau_k)\)}
    \LeftLabel{\(\chi_1\)}
    \RightLabel{\(\cut\)}
    \BinaryInfC{\(\psi'_j, \Phi\interp \bot \Rightarrow \Phi\)}
    \DisplayProof
  \]
}
  where we denoted \(\Phi'_{(k,j)},\Phi_{[0,m)}, \Phi'_{[0,k)},\phi'\) as \(\Phi\) and annotated the cut formula at the left of the rule application.

  All that is left is to define proofs \(\rho_i\) for \(i < m\).
  We remember that we have the following proofs:
  \begin{align*}
    &\pi_i \vdash \psi_i, (\Phi_{[0,i)}, \chi_1) \interp \bot \Rightarrow \Phi_{[0,i)}, \chi_1, \\
    &\tau_k \vdash \chi_1, (\Phi'_{[0,k)}, \phi') \interp \bot \Rightarrow \Phi'_{[0,k)}, \phi'.
  \end{align*}
  Applying Lemma~\ref{lm:nec-rule} we obtain \( \tau'_k \vdash (\Phi'_{[0,k)}, \phi') \interp \bot \Rightarrow \chi_1 \interp \bot \).
  Then the desired proof \(\rho_i\) is defined as
  \[
  	\AxiomC{\(\wk(\tau'_k)\)}
    \noLine
    \UnaryInfC{\(\psi_i, \Phi \interp \bot \Rightarrow \Phi, \chi_1, \chi_1 \interp \bot\)}
  	\AxiomC{\(\wk(\pi_i)\)}
    \noLine
    \UnaryInfC{\(\chi_1 \interp \bot, \psi_i, \Phi \interp \bot \Rightarrow \Phi, \chi_1\)}
  	\RightLabel{\(\cut\)}
    \BinaryInfC{\(\psi_i, \Phi \interp \bot \Rightarrow \Phi, \chi_1\)}
  	\AxiomC{\(\wk(\tau_k)\)}
    \noLine
    \UnaryInfC{\(\chi_1, \psi_i, \Phi \interp \bot \Rightarrow \Phi\)}
  	\RightLabel{\(\cut\)}
    \BinaryInfC{\(\psi_i, \Phi \interp \bot\Rightarrow \Phi\)}
    \DisplayProof
  \]
  where we denoted \(\phi_{[0,i)}, \phi'_{[0,k)}, \phi'\) as \(\Phi\).
\end{proof}

Thanks to the previous lemma we can conclude the two desired \(\cut\) elimination results.

\begin{corollary}[\(\gentzenIL\) \(\cut\) elim.]\label{tm:gentzenILcut-to-gentzenIL}
  We have the following:
  \begin{enumerate}
    \item \(\cut\) is eliminable in \(\gentzenIL\).
    \item \(\cut\) is eliminable in \(\fgentzenIL\).
  \end{enumerate}
\end{corollary}
\begin{proof}
  That \(\cut\) is eliminable in \(\gentzenIL\) follows straightforwardly from Lemma~\ref{lm:admissibility-and-eliminability} and Theorem~\ref{th:local-cut-admissibility}.
  Finally, that \(\cut\) is eliminable in \(\fgentzenIL\) follows from \(\cut\) eliminability in \(\gentzenIL\) together with the transformations of Theorem~\ref{tm:fgentzenILcut-to-genteznILcut} and Corollary~\ref{cr:from-gentzen-to-fgentzen}.
\end{proof}

\section{Regularizing proofs}
\label{sub:cyclic-systems}

Let \(\mathcal{G} = (\mathcal{R}, L)\) be a local-progress calculus.
We are going to define an alternative notion of proof called \emph{cyclic proof}.
A cyclic preproof in \(\mathcal{G}\) is a pair \(\pi =(\tau, w \mapsto w^\circ)\) such that
\begin{enumerate}
  \item \(\tau\) is a finite tree generated by the rules of \(\mathcal{R}\), where some leaves are sequents marked with a rule denoted \(\mathrm{Repeat}\).
    These leaves are called \emph{repeat nodes}.
  \item \(w \mapsto w^\circ\) is a function whose domain is the set of repeat nodes of \(\tau\) and additionally, it fulfills that \(w^\circ < w\) and \(S_{w^\circ} = S_w\), where \(S_{w^\circ}\) is the sequent at \(w^\circ\) and \(S\) the sequent at \(w\).
\end{enumerate}
A \emph{(cyclic) proof} is a preproof where for any repeat leaf \(w\) in the path from \(w^\circ\) to \(w\) there is progress, i.e., there is a node \(v\) with children \(v0,\ldots,v(n-1)\) such that \(w^\circ \leq v < vi < w\) and \(i \in L_{R}(S_{0},\ldots,S_{n-1}, S)\) where \(R\) is the rule at \(v\), \(S\) is the sequent at \(v\) and \(S_j\) is the sequent at \(vj\) for \(j < n\).

We will write \(\cgentzenIL\) to denote the local-progress proof system \(\gentzenIL\) with the notion of cyclic proof instead of non-wellfounded proof.
So we will write \(\cgentzenIL \vdash S\) to mean that there is a cyclic proof of \(S\) in \(\gentzenIL\).
In this section we will show that \(\gentzenIL \vdash S\) implies \(\cgentzenIL \vdash S\).
In order to do that we will introduce a local progress calculus in the middle of \(\gentzenIL\) and \(\cgentzenIL\) called \(\slimgentzenIL\).

\begin{definition}
  We define the rule \(\interp^{\mathrm{slim}}_{\mathrm{IK4}}\) as
  \[
    \AxiomC{\([\psi_i, (\Phi_{[0,i)}, \phi) \interp \bot \Rightarrow \Phi_{[0,i)}, \phi]_{m...i...0}\)}
    \RightLabel{\(\interp^{\mathrm{slim}}_{\mathrm{IK4}}\)}
    \UnaryInfC{\(\{\phi_i \interp \psi_i\}_{i < m}, \Gamma \Rightarrow \Delta, \psi_m \interp \phi\)}
    \DisplayProof
  \]
  where there are no repetitions in \(\{\phi_i\}_{i < m}\) (equivalently, \(\Phi_{[0,m)}\) is a set instead of a multiset).

  We define the sequent calculus \(\slimgentzenIL\) as the local-progress sequent calculus given by the rules of Figure~\ref{fig:rule} without rules \(\interp_{\mathrm{IL}}\), \(\interp_{\mathrm{IK4}}\) and \(\cut\) adding the rule \(\interp^{\mathrm{slim}}_{\mathrm{IK4}}\).
  Progress only occurs at the premises of \(\interp^{\mathrm{slim}}_{\mathrm{IK4}}\).
\end{definition}

\begin{theorem}\label{thm:gentzenIL-to-slimgentzenIL}
  If \(\gentzenIL \vdash S\) then \(\slimgentzenIL \vdash S\).
\end{theorem}
\begin{proof}
  Say that a proof is \emph{locally slim} if all the applications of \(\interp_{\mathrm{IK4}}\) in its main fragment are instances of \(\interp^{slim}_{\mathrm{IK4}}\).
  We are going to show that every proof of a sequent can be transformed into a locally slim proof of the same sequent.
  To obtain a translation from \(\gentzenIL\) to \(\slimgentzenIL\) is suffices to use the translation method of Subsubsection~\ref{subsub:translations}.
  So assume that \(\pi \vdash S\) in \(\gentzenIL\), we proceed by induction on the local height of \(\pi\) and cases in the last rule applied.
  The only non-trivial case is when \(\pi\) is of shape
  \[
    \AxiomC{\(\begin{bmatrix}\pi_i\\
        \psi_i, (\Phi_{[0,i)}, \phi) \interp \bot \Rightarrow \Phi_{[0,i)}, \phi
    		\end{bmatrix}_{m...i...0}\)}
    \RightLabel{\(\interp_{\mathrm{IK4}}.\)}
    \UnaryInfC{\(\{\phi_i \interp \psi_i\}_{i < m}, \Gamma \Rightarrow \Delta, \psi_m \interp \phi\)}
    \DisplayProof
  \]
  We proceed by a subinduction in the number of repetitions in \(\Phi_{[0,m)}\).
  If there are no repetitions in \(\Phi_{[0,m)}\), then it is clear that \(\Phi_{[0,m)}\) is a set.
  So we only need to change the rule label from \(\interp_{\mathrm{IK4}}\) to \(\interp^{slim}_{\mathrm{IK4}}\).
  Now assume that there is a repeated formula \(\phi_k = \phi_j\) for \(k > j\).
  Then, for each \(i > k\) define \(\rho_i\) as the proof in \(\gentzenIL\) obtain from eliminating \(\ctr\) from
	\[
		\AxiomC{\(\pi_i\)}
		\noLine
    \UnaryInfC{\(\psi_i, (\Phi_{[0,i)}, \phi) \interp\bot\Rightarrow\Phi_i,\phi\)}
		\RightLabel{\(\ctr\)}
    \UnaryInfC{\(\psi_i,(\Phi_{[0,i)}, \phi)\interp\bot \Rightarrow \Phi_{[0,k)},\Phi_{(k,i)},\phi\)}
		\RightLabel{\(\ctr\).}
    \UnaryInfC{\(\psi_i, (\Phi_{[0,k)},\Phi_{(k,i)},\phi)\interp\bot\Rightarrow \Phi_{[0,k)},\Phi_{(k,i)},\phi\)}
    		\DisplayProof
	\]
  Then we define the proof \(\rho\) as
  \[
    \AxiomC{\(\rho_m \quad \cdots \quad \rho_{k+1} \quad \pi_{k-1} \cdots \quad \cdots \phi_0\)}
    \RightLabel{\(\interp_{\mathrm{IK4}}\)}
    \UnaryInfC{\(\{\phi_i \interp \psi_i\}_{i < m, i \neq k}, \Gamma, \phi_k \interp \psi_k \Rightarrow \Delta, \psi_m \interp \phi\)}
    \DisplayProof
  \]
  where \(\interp_{\mathrm{IK4}}\) have been applied with the ordering
  \[
    \phi_0 \interp \psi_0, \ldots, \phi_{k-1} \interp \psi_{k-1}, \phi_{k+1} \interp \psi_{k+1}, \ldots, \phi_{m-1} \interp \psi_{m-1}
  \]
    and principal formula \(\psi_m \interp \phi\).
  Since the number of repetitions have decreased, we can apply the induction hypothesis.
\end{proof}

Finally, we are going to see how to transform a non-wellfounded proof in \(\slimgentzenIL\) into a cyclic proof in \(\cgentzenIL\).

\begin{definition}
  Let \(\pi\) be a proof in \(\slimgentzenIL\) and \(w \in \nodes(\pi)\).
  A node \(w\) is called \emph{finite} if for any \(v < u < w\) we have \(S^\pi_v \neq S^\pi_u\) and a finite node \(w\) is called \emph{cyclic} if there is a \(v < w\) such that \(S^\pi_v = S^\pi_w\).
      We notice that this \(v\) must be unique and is called the \emph{cyclic companion of \(w\)}, denoted \(w^\circ\).
\end{definition}

\begin{theorem}\label{thm:slimgentzenIL-to-cgentzenIL}
  If \(\slimgentzenIL \vdash S\) then \(\cgentzenIL \vdash S\).
\end{theorem}
\begin{proof}
  Let \(\tau \vdash S\) in \(\slimgentzenIL\).
Using the subformula property, each premise of the modal rule in a proof in \(\slimgentzenIL\) is determined by a finite set of formulas in \(\sub(S)\) (since the application of the rule is slim we can assume it is a subset) and two formulas of \(\sub(S)\), for example given subset \(\Phi\) and formulas \(\psi,\phi\) the associated premise would be \(\psi, (\Phi, \phi) \interp \bot \Rightarrow \Phi, \phi\).
We can see then, that the possible number of premise sequents of modal rules is bounded by \(2^kk^2\) where \(k = |\sub(S)|\).

We define the tree \(\tau'\) as \(\tau \restricts_N\) where
\[
  N = \{w \in \nodes(\tau) \mid w \text{ is finite}\},
\]
and the rules at the cyclic nodes has been replaced for Repeat.
Define cyclic tree \(\rho = (\tau', w \text{ cyclic} \mapsto w^\circ)\) and let us show that it is the desired cyclic proof.
Clearly, \(\tau'\) is generated by the rules.
It must also fulfill the branch condition, since the premises of all the rules in \(\slimgentzenIL\) have a smaller size than the conclusion, except for \(\interp^{\mathrm{slim}}_{\mathrm{IK4}}\).

All left to show is that \(\tau'\) is finite.
Assume otherwise, then by K\"onig's Lemma (as \(\tau'\) is finitely branching), it must have an infinite branch \(b\).
This is also an infinite branch of \(\tau\) so it must go through the rule \(\interp^{\mathrm{slim}}_{\mathrm{IK4}}\) infinitely many times, so \(\{i \in \mathbb{N} \mid S^\tau_{b \restricts i} \text{ is a premise of \(\interp^{\mathrm{slim}}_{\mathrm{IK4}}\)}\}\) is infinite.
However, there is only a finite amount of possible sequents for a premise of \(\interp^{\mathrm{slim}}_{\mathrm{IK4}}\) in \(\tau\) so there are \(i < j\) such that \(S^\tau_{w_i} = S^\tau_{w_j}\).
This implies that \(w_{j+1}\) is not a node of \(\tau'\), since it is not finite, a contradiction.
\end{proof}

The following result is obtained directly from Theorem~\ref{thm:gentzenIL-to-slimgentzenIL} and Theorem~\ref{thm:slimgentzenIL-to-cgentzenIL}.

\begin{corollary}
If \(\gentzenIL \vdash S\) then \(\cgentzenIL \vdash S\).
\end{corollary}

\section{Uniform interpolation}
\label{sec:interpolation}

In this section we are going to show how to prove the existence of uniform interpolation for \(\IL\) using the Fixpoint Theorem and non-wellfounded proofs.
This was inspired by the proof of the same result in \(\mu\)-calculus from \cite{uniform-interp-mu}).
First we need to show how to solve modal equational systems in \(\IL\).
Then, using a modal equational system and a proof search tree in \(\cgentzenIL\) we will construct a candidate of uniform interpolant.
We will prove that the candidate of uniform interpolant is indeed the uniform interpolant by corecursively constructing proofs in \(\gentzenIL\).
Finally, we will lift this result also to \(\ILP\) using a strong interpretation of \(\ILP\) in \(\IL\).

For definiteness, let us formulate what uniform interpolation for a logic \(L\) means.
We define it for any logic \(L\), although we are not going to define what a logic is.
In practice, in this paper \(L\) will be either \(\IL\) or \(\ILP\) (defined at Subsection~\ref{subsec:ILP}).
\begin{definition}
  Let \(L\) be a logic.
  For any formula \(\phi\) and vocabulary \(V \subseteq \vocab(\phi)\) we say that \(\iota\) is the \emph{\(L\)-uniform interpolant} of \(\phi\) if
  \begin{enumerate}
    \item \(\vocab(\iota) \subseteq V\),
    \item \(\IL \vdash \phi \to \iota\)
    \item For any \(\psi\) with \(\vocab(\psi) \subseteq V\) such that \(\IL \vdash \phi \to \psi\) we have that \(\IL \vdash \iota \to \psi\).
  \end{enumerate}
  We say that \(L\) \emph{has uniform interpolation} if any formula has an \(L\)-uniform interpolant.
\end{definition}

\subsection{Modal equational systems}
Our first step into uniform interpolation will be to study equations systems in \(\IL\).
In particular, we are interested in finding sufficient conditions under which an equation system will have an unique solution modulo equivalence in \(\IL\).
Thanks to the fixpoint theorem of \(\IL\) this study will be analogous to the case of the logic \(\GL\).
Nevertheless, due to the difference between \(\IL\) and \(\GL\) (particularly, \(\IL\) has an extra binary modality) we feel the need to write the adapted proofs here.

\begin{definition}
  We say that \(\phi\) is \emph{modalized} in a variable \(p\) if every occurrence of \(p\) in \(\phi\) is under the scope of a \(\interp\) connective.
  We say that \(\phi\) is \emph{propositional} in a variable \(p\) if no occurrence of \(p\) in \(\phi\) is under the scope of a \(\interp\) connective.
\end{definition}

Given a formula \(\phi\) we define the \emph{vocabulary of \(\phi\)} as the set of propositional variables occuring in \(\phi\), usually denoted as \(\vocab(\phi)\).
We start by formulating the fixpoint theorem in \(\IL\), a proof of this theorem for \(\IL\) can be found in \cite{Areces1998}.

\begin{theorem}[Fixpoint Theorem]
  Let \(\phi(p)\) be a formula such that \(p\) is modalized in \(\phi\).
  Then, there is a formula \(\psi\) with \(\vocab(\psi) \subseteq \vocab(\phi) \setminus \{p\}\) and
  \[
    \IL \vdash \psi \leftrightarrow \phi(\psi).
  \]
\end{theorem}

We turn to the definition of modal equation system.
Note that given a \emph{substitution} \(f\) (i.e.\ a function from propositional variables to formulas) and a formula \(\phi\) we will write \(\phi[f]\) to mean the simultaneous substitution in \(\phi\) of each variable \(p\) for \(f(p)\).

\begin{definition}
  Let \(B\) and \(V\) be finite disjoint sets of propositional variables.
  A \((B,V)\)-modal equational system is a finite set \(\mathcal{E}\) of formulas of shape
  \[
    \{x \leftrightarrow \phi_x \mid x \in B\}
  \]
  such that for each \(x \in B\), \(\vocab(\phi_x) \subseteq B \cup V\).
  The elements of \(B\) are called the \emph{bound variables} of \(\mathcal{E}\), while the variables in \(V\) are called the \emph{free variables} of \(\mathcal{E}\).

  We say that a \((B,V)\)-modal equational system is \emph{orderable} if there is an enumeration \(x_0,\ldots,x_{n}\) of \(B\) such that for any \(j\) and \(i \leq j\), \(\phi_{x_j}\) is modalized in \(x_i\).

  A \emph{solution in \(\mathrm{IL}\)} of \(\mathcal{E}\) is a function \(y \in B \mapsto \psi_y\) such that for any \(x \in B\) we have that
  \begin{enumerate}
    \item \(\vocab(\psi_x) \subseteq V\), and
    \item \(\mathrm{IL} \vdash \psi_x \leftrightarrow \phi_x[y \in B \mapsto \psi_y]\).
  \end{enumerate}
\end{definition}

We want to show that any solvable equation system has an unique solution (modulo equivalence in \(\IL\)).
We will start with some lemmas, which are just restatements in \(\IL\) and generalizations of lemmas from \cite{smorynski}, that will guarantee the uniqueness (modulo equivalence).
We start with a simple lemma that will allow us to not reprove things twice.

\begin{lemma}[Simple Formalization Lemma]
   \(\IL \vdash \dnec \phi \to \psi\) implies \(\IL \vdash \nec \phi \to \nec \psi\).
\end{lemma}
\begin{proof}
  By necessitation and axioms \((K)\) and \((4)\).
\end{proof}

We need to establish a substitution lemma for \(\IL\).

\begin{lemma}
  \label{lm:equivalence}
  We have that
  \[
    \IL \vdash \nec(\phi_0 \leftrightarrow \phi_1) \wedge \nec(\psi_0 \leftrightarrow \psi_1) \to ((\phi_0 \interp \psi_0) \leftrightarrow (\phi_1 \interp \psi_1)).
  \]
\end{lemma}
\begin{proof}
  Note that \(\IL \vdash \nec(\phi_0 \leftrightarrow \phi_1) \to \nec(\phi_0 \to \phi_1)\), so by (J1) we obtain \(\IL \vdash \nec(\phi_0 \leftrightarrow \phi_1) \to \phi_0 \interp \phi_1\).
  It is easy to see then that
  \[
    \IL \vdash \nec(\phi_0 \leftrightarrow \phi_1) \wedge \nec(\psi_0 \leftrightarrow \psi_1) \to (\phi_0 \interp \phi_1) \wedge (\phi_1 \interp \phi_0) \wedge (\psi_0 \interp \psi_1) \wedge (\psi_1 \interp \psi_0).
  \]
  Then by (J2) we obtain the desired
  \[
    \IL \vdash \nec(\phi_0 \leftrightarrow \phi_1) \wedge \nec(\psi_0 \leftrightarrow \psi_1) \to ((\phi_0 \interp \psi_0) \leftrightarrow (\phi_1 \interp \psi_1)).
    \qedhere
  \]
\end{proof}

\begin{lemma}[Substitution Lemma]
  Given a formula \(\phi(p)\) we have that
  \begin{enumerate}
    \item (Propositional) If \(\phi\) is propositional in \(p\) then \(\IL \vdash (\psi \leftrightarrow \chi) \to (\phi(\psi) \leftrightarrow \phi(\chi))\).
    \item (First) \(\IL \vdash \dnec(\psi \leftrightarrow \chi) \to (\phi(\psi) \leftrightarrow \phi(\chi))\).
    \item (Second) \(\IL \vdash \nec(\psi \leftrightarrow \chi) \to \nec(\phi(\psi) \leftrightarrow \phi(\chi))\).
  \end{enumerate}
\end{lemma}
\begin{proof}
  The Propositional Substitution Lemma is proven by induction on the complexity of \(\phi\) using propositional (non-modal) reasoning.
  The Second Substitution Lemma is a consequence of the first by applying the Simple Formalization Lemma, so we just prove the First Substitution Lemma.

  By induction on the complexity of \(\phi\), the only interesting case is when \(\phi(p) = \phi_0(p) \interp \phi_1(p)\).
  Using the induction hypothesis we have that 
  \[
    \IL \vdash \dnec (\psi \leftrightarrow \chi) \to (\phi_0(\psi) \leftrightarrow \phi_0(\chi)) \wedge (\phi_1(\psi) \leftrightarrow \phi_1(\chi)).
  \]
  By using properties of \(\dnec\) and of \(\nec\) we obtain
  \[
    \IL \vdash \dnec (\psi \leftrightarrow \chi) \to \nec (\phi_0(\psi) \leftrightarrow \phi_0(\chi)) \wedge \nec (\phi_1(\psi) \leftrightarrow \phi_1(\chi)).
  \]
  Finally, the desired
  \[
    \IL \vdash \dnec(\psi \leftrightarrow \chi) \to ((\phi_0(\psi) \interp \phi_1(\psi)) \leftrightarrow (\phi_0(\chi) \interp \phi_1(\chi)))
  \]
  is obtained using by Lemma~\ref{lm:equivalence}.
\end{proof}

The Substitution Lemma allow us to show a generalized version of the uniqueness of fixpoints.
This generalized version establishes the uniqueness of solution (modulo equivalence) for orderable modal equation systems.

\begin{lemma}[Generalized uniqueness of fixpoints]
  \label{lm:uniqueness}
  Let \(\phi_0(p_0,\ldots,p_n), \ldots, \phi_n(p_0,\ldots,p_n)\) be formulas such that \(\phi_i\) is modalized in \(p_0,\ldots, p_i\) and \(q_0,\ldots,q_n\) be new variables.
  Define the set \(H\) as containing the formulas
  \[
    \dnec(p_i \leftrightarrow \phi_i(p_0,\ldots,p_n)) 
    \qquad
    \dnec(q_i \leftrightarrow \phi_i(q_0,\ldots,q_n))
  \]
  for \(i \leq n\).
  Then
  \[
    \IL \vdash \bigwedge H \to \left(\bigwedge_{i \leq n} p_i \leftrightarrow q_i\right).
  \]
\end{lemma}
\begin{proof}
  We are going to show that for \(i \leq n\)
  \( \IL \vdash \bigwedge H \wedge \left(\bigwedge_{j < i} \nec(p_j \leftrightarrow q_j)\right) \to \dnec (p_{i} \leftrightarrow q_{i}) \).
  Then, the desired result follows using these formulas via propositional reasoning.

  We proceed by induction on the reverse natural order on \(\{0,\ldots,n\}\)
  So, we have to show that
  \(
    \IL \vdash \bigwedge H \wedge \left(\bigwedge_{j < i } \nec(p_j \leftrightarrow q_j)\right) \to \dnec (p_{i} \leftrightarrow q_{i})
  \)
  assuming that for \(i < k \leq n\) we have
  \(
    \IL \vdash \bigwedge H \wedge \left(\bigwedge_{j < k } \nec(p_j \leftrightarrow q_j)\right) \to \dnec (p_{k} \leftrightarrow q_{k})
  \)
  Using these assumptions we can obtain that
  \[
    \tag{i}
    \IL \vdash \bigwedge H \wedge \left(\bigwedge_{j \leq i} \nec(p_j \leftrightarrow q_j)\right) \to \bigwedge_{i < j \leq n}\dnec (p_{j} \leftrightarrow q_{j}).
  \]
  Since \(\phi_i\) is modalized in \(p_0,\ldots,p_{i}\), there is a formula \(\phi'_i(r_0,\ldots,r_m, p_{i+1},\ldots,p_{n})\) without occurrences of \(p_0,\ldots,p_{i}\) nor \(\interp\) and formulas \(\psi_0(p_0,\ldots,p_n), \ldots, \psi_m(p_0,\ldots,p_n), \chi_0(p_0,\ldots,p_n), \ldots, \chi_m(p_0,\ldots,p_n)\) such that
  \[
    \tag{ii}
    \phi_i = \phi'_i(\psi_0 \interp \chi_0, \ldots, \psi_m \interp \chi_m, p_{i+1}, \ldots, p_n).
  \]
  Using the Second Substitution Lemma we get that for \(i \leq m\) \newline
  \(
    \IL \vdash \left(\bigwedge_{j \leq n} \nec(p_j \leftrightarrow q_j)\right) \to \nec(\psi_i(p_0,\ldots,p_n) \leftrightarrow \psi_i(q_0,\ldots,q_n)) \wedge \nec(\chi_i(p_0,\ldots,p_n) \leftrightarrow \chi_i(q_0,\ldots,q_n)),
  \)
  and then, by Lemma~\ref{lm:equivalence}, we have
  \newline
  \(
    \IL \vdash \left(\bigwedge_{j \leq n} \nec(p_j \leftrightarrow q_j)\right) \to ((\psi_i(p_0,\ldots,p_n) \interp \chi_i(p_0,\ldots,p_n))\leftrightarrow (\psi_i(q_0,\ldots,q_n) \interp \chi_i(q_0,\ldots,q_n)))
  \).
  \newline
  So using Propositional Substitution Lemma and remembering the shape of \(\phi_i\) displayed at (ii), we obtain
  \[
    \tag{iii}
    \IL \vdash \left(\bigwedge_{j \leq i} \nec(p_j \leftrightarrow q_j)\right) \wedge \left(\bigwedge_{i < j \leq n}\dnec(p_j \leftrightarrow q_j)\right) \to (\phi_i(p_0,\ldots,p_n) \leftrightarrow \phi_i(q_0,\ldots,q_n))
  \]
  (i) and (iii) gives
  \(
    \IL \vdash \bigwedge H \wedge \left(\bigwedge_{j < i}\nec(p_j \leftrightarrow q_j)\right) \to (\nec(p_i \leftrightarrow q_i) \to (\phi_i(p_0,\ldots, p_n)\leftrightarrow \phi_i(q_0,\ldots,q_n))).
  \)
  Then, by definition of \(H\), we also obtain
  \[
    \tag{iv}
    \IL \vdash \bigwedge H \wedge \left(\bigwedge_{j < i}\nec(p_j \leftrightarrow q_j)\right) \to (\nec(p_i \leftrightarrow q_i) \to (p_i \leftrightarrow q_i)).
  \]
  So by applying necessitation and using the properties of \(\nec,\dnec\) together with axiom (4), we conclude that
  \(
    \IL \vdash \bigwedge H \wedge \left(\bigwedge_{j < i}\nec(p_j \leftrightarrow q_j)\right) \to \nec (\nec(p_i \leftrightarrow q_i) \to (p_i \leftrightarrow q_i)).
  \)
  By L\"ob's axiom we obtain that
  \(
    \IL \vdash \bigwedge H \wedge \left(\bigwedge_{j < i}\nec(p_j \leftrightarrow q_j)\right) \to \nec (p_i \leftrightarrow q_i)
  \).
  Finally, by (iv) we can conclude 
  \[
    \IL \vdash \bigwedge H \wedge \left(\bigwedge_{j < i}\nec(p_j \leftrightarrow q_j)\right) \to \dnec (p_i \leftrightarrow q_i)
  \]
  as desired.
\end{proof}

The uniqueness (modulo equivalence) of solution in orderable modal equational systems will follow from the previous lemma.
In the following theorem we also show the existence of such a solution.

\begin{theorem}
  Let \(\mathcal{E}\) be an orderable \((B,V)\)-modal equational system.
  Then \(\mathcal{E}\) has an unique solution (up to logical equivalence) in \(\mathrm{IL}\).
\end{theorem}
\begin{proof}
  Proof of existence.
  Assume we have the enumeration \(x_0, \ldots, x_n\) of \(B\) such that for any \(j\) and \(i \leq j\), \(\phi_{x_j}\) is modalized in \(x_i\), let us denote \(\phi_{x_j}\) as \(\phi_j\).
  By recursion on \(j\) define formulas \(\psi_j\) for \(j \leq n\) and \(\chi^j_i\) for \(i \leq j \leq n\) such that
  \begin{enumerate}
    \item \(\vocab(\psi_j) \subseteq \{x_{j + 1}, \ldots, x_n\} \cup V\) and \(\vocab(\chi^j_i) \subseteq \{x_{j + 1}, \ldots, x_n\} \cup V\) for \(i \leq j\),
    \item We have that for \(j \leq n\)
      \[
        \IL \vdash \phi_j(\chi^{j-1}_0[x_j \mapsto \psi_j],\ldots, \chi^{j-1}_{j-1}[x_j \mapsto \psi_j], \psi_j, x_{j+1}, \ldots, x_{n}) \leftrightarrow \psi_j
      \]
      and for \(i \leq j \leq n\)
      \[
        \IL \vdash \phi_i(\chi^j_0,\ldots, \chi^j_j, x_{j+1}, \ldots, x_n) \leftrightarrow \chi^j_i.
      \]
  \end{enumerate}
  We define \(\psi_0\) as the fixpoint of \(\phi_0(x_0,\ldots,x_{n})\) with respect to modalized variable \(x_0\) and \(\chi^0_0 := \psi_0\).
  It is clear that \(\psi_0\) and \(\chi^0_0\) fulfill the conditions.
  Assume we have defined up to stage \(j\), and let us define stage \(j+1\).
  Notice that the formula \(\phi_{j+1}(\chi^{j}_0, \ldots, \chi^{j}_j,x_{j+1},\ldots,x_n)\) is modalized in \(x_{j+1}\), as \(\phi_{j+1}(x_0,\ldots,x_{n})\) is modalized in \(x_0,\ldots,x_{j+1}\); and its vocabulary is contained in \(\{x_{j+1},\ldots,x_n\} \cup V\)
  We define \(\psi_{j+1}\) to be the fixpoint of that formula at \(x_{j+1}\) so
  \begin{multline*}
    \vocab(\psi_{j+1}) \subseteq \{x_{j+2}, \ldots, x_n\} \cup V \text{ and}\\
    \IL \vdash \phi_{j+1}(\chi^{j}_0[x_{j+1} \mapsto \psi_{j+1}], \ldots, \chi^{j}_j[x_{j+1} \mapsto \psi_{j+1}],\psi_{j+1}, x_{j+2},\ldots,x_n) \leftrightarrow \psi_{j+1}.
  \end{multline*}
  Define \(\chi^{j+1}_i := \chi^j_i[x_{j+1} \mapsto \psi_{j+1}]\) for \(i < j+1\) and \(\chi^{j+1}_{j+1} := \psi_{j+1}\).
  Then, using the induction hypothesis and that the set of \(\IL\)-theorems is closed under substitution, it is easy to check that both of the needed properties are true.

  Proof of uniqueness.
  This follows straightforwardly from Lemma~\ref{lm:uniqueness}.
\end{proof}

\subsection{Construction of the interpolant}

During this section we will fix a set of propositional variables, a vocabulary, \(V\).
The idea is that we want to build an uniform interpolant with respect this vocabulary \(V\).

Given a sequent we want to construct a proof search in \(\cgentzenIL\) from which the interpolant will be defined.
Depending on the shape of the proof search the exact definition of the uniform interpolant will vary slightly, for this reason we will call this proof search an \emph{interpolation template}.

\begin{figure}
  \[
    \AxiomC{\([\psi^\epsilon, \Phi \interp \bot \Rightarrow \Phi]_{\Phi, \psi^\epsilon}\)}
    \RightLabel{\(\interp^*_{\mathrm{IK4}}\)}
    \UnaryInfC{\(\Sigma, \Gamma \Rightarrow \Lambda , \Delta\)}
    \DisplayProof
  \]
  where
  \begin{enumerate}
    \item \(\Gamma, \Delta\) are sets of propositional variables.
    \item \(\Sigma,\Lambda\) are sets of \(\interp\)-formulas.
    \item \(\Phi \in \{\Phi \subseteq \an(\Sigma) \cup \su(\Lambda) \mid |\Phi \setminus \an(\Sigma)| \leq 1\}\), in words, \(\Phi\) is a multiset with some antecessors of \(\Sigma\) and \emph{at most} one succedent in \(\Lambda\),
    \item \(\psi^\epsilon \in \su(\Sigma) \cup \an(\Lambda) \cup \{\epsilon\}\). We understand that \(\psi\) can be either a formula (in the corresponding set) or nothing. The second option is represented via the \(\epsilon\). In order to denote a formula or nothing we will usually write the superscript \(^\epsilon\), to remember that it may be nothing.
  \end{enumerate}

  \caption{\(\interp^*_{\mathrm{IK4}}\) rule}
\end{figure}

\begin{definition}
  An interpolation template is a cyclic preproof \(T\) (see Subsection~\ref{sub:cyclic-systems}) constructed using the rules of \(\ax\), \(\botL\), \(\botR\), \(\toL\), \(\toR\), \(\ninv\) and
  \[
    \AxiomC{}
    \RightLabel{Empty}
    \UnaryInfC{\(\Rightarrow\)}
    \DisplayProof
    \qquad
    \AxiomC{\(\Gamma \Rightarrow \Delta\)}
    \RightLabel{Wk}
    \UnaryInfC{\(\Gamma, \Gamma' \Rightarrow \Delta, \Delta'\)}
    \DisplayProof
  \]
  In addition, an interpolation template must fulfill the following conditions.
  \begin{enumerate}
    \item (Determinism) Any two non-repeat nodes labelled by the same sequent are instances of the same rule instantiation.
    \item (Axiomatic termination) Every node with an axiomatic set-sequent\footnote{A set-sequent is a sequent \(\Gamma \Rightarrow \Delta\) where any formula occurs at most once in \(\Gamma\) and at most once in \(\Delta\).} is a leaf.
    \item (Cyclic termination) For any non-axiomatic node \(w\) with a set-sequent, if there is a node \(v\) below \(w\) with the same sequent has \(w\) then \(w\) is a repeat.
    \item (Weakening condition)
      If some formula occurs more than once in the sequent of \(w\) (i.e., if the sequent at \(w\) is not a set-sequent), then \(w\) is obtained by an application of Wk of shape
      \[
        \AxiomC{\(\Gamma^s \Rightarrow \Delta^s\)}
        \RightLabel{Wk}
        \UnaryInfC{\(\Gamma \Rightarrow \Delta\)}
        \DisplayProof
      \]
      where \(\Gamma^s\) and \(\Delta^s\) are the sets of formulas obtained from the multisets \(\Gamma\) and \(\Delta\) by contracting all the repetitions.
      No other forms of \(\mathrm{Wk}\) appear in \(\tau\).
  \end{enumerate}
  We say that \(T\) is an interpolation template of a sequent \(S\) if \(S\) is at the root of \(T\).
\end{definition}

Since we need to construct an uniform interpolant for every formula it is necessary to show that every sequent has an interpolation template.

\begin{lemma}
Every sequent has an interpolation template.
\end{lemma}
\begin{proof}
  We informally describe the process of, given a sequent \(\Gamma \Rightarrow \Delta\), construct a interpolation template for \(\Gamma \Rightarrow \Delta\) by stages. We will also guarantee that at some stage we will stop.
  We assume we are given an enumeration \(\{\phi_i\}_{i \in \mathbb{N}}\) of the formulas of \(\IL\). The enumeration will help us with the determinism condition.

  Stage \(0\). 
  We construct a tree whose root has the sequent \(\Gamma \Rightarrow \Delta\).
  In case it is a set-sequent we do not do anything else in this stage, in particular no rule will be attached to the root.
  If it is not a set-sequent we annotate the root with the rule \(\mathrm{Wk}\) and we add a node on top of it with the sequent \(\Gamma^s \Rightarrow \Delta^s\) and we do not attach any rule to it (yet).

  Stage \(n+1\).
  We make a list of all the leaves of the tree that do not have a rule attached to it. If the list is empty we finish the procedure.
  Otherwise traverse the list doing the following to each of its elements.
  \begin{enumerate}
    \item Look at the sequent \(\Gamma_w \Rightarrow \Delta_w\) attached to the leaf \(w\).
      If it is not a set-sequent apply \(\mathrm{Wk}\) to it obtaining a new leaf \(w'\) with sequent \(\Gamma^s_w \Rightarrow \Delta^s_w\) and no rule.
      Apply the next step to \(w'\).
      In case \(\Gamma_w \Rightarrow \Delta_w\) is already a set-sequent apply the next step directly in \(w\).
    \item Apply the first instruction possible from the following list, depending on the shape of the sequent.
      \begin{enumerate}
        \item If \(\bot\) occurs at the left side of the sequent annotate the node with the rule \(\bot\mathrm{L}\).
        \item If a propositional variable occurs on both sides of the sequent annotate the node with the rule \(\mathrm{ax}\).
        \item If the sequent is empty annotate the rule \(\mathrm{Empty}\) to the node.
        \item If there is a node below with the same sequent create annotate the node with the rule \(\mathrm{Repeat}\) and create a cycle to the (unique) node below with the same sequent.
        \item If \(\bot\) occurs at the right side of the sequent annotate the node with the rule \(\bot\mathrm{R}\) and create a new leaf above it with the corresponding premise.
        \item If there is an implication on the left side of the sequent look for the first one that occurs in the enumeration \(\{\phi_i\}_{i \in \mathbb{N}}\), let it be \(\phi \to \psi\).
          Annotate the node with the rule \(\to\mathrm{L}\) and add two leaves above the node with the corresponding premises of applying \(\to\mathrm{L}\) with principal formula \(\phi \to \psi\).
        \item If there is an implication on the right side of the sequent look for the first one that occurs in the enumeration \(\{\phi_i\}_{i \in \mathbb{N}}\), let it be \(\phi \to \psi\).
          Annotate the node with the rule \(\to\mathrm{R}\) and add one leaf above the node with the corresponding premise of applying \(\to\mathrm{R}\) with principal formula \(\phi \to \psi\).
        \item Otherwise, annotate the node with the rule \(\interp^*_{\IKT}\) and add as many leaves as necessary above the node to have all the needed premises for the application of the rule.
          Annotate each leaf with a different premise and with no rule.
      \end{enumerate}
  \end{enumerate}
  We want to argue that this process finishes, i.e., that at some stage all the leaves are annotated with a rule.
  Assume otherwise, note that after \(\omega\)-stages we would have construced an infinite finitely-branching tree.
  By K\"onig's lemma we will have an infinite branch.
  We can look at the set-sequents in this infinite branch, there must be infinitely many.
  However, all the rules we applied fulfill the subformula property, so the number of possible set-sequents appearing on the branch is finite.
  This implies that there must be a repeated set-sequent.
  However, in the second repetition of this set-sequent the branch should have been closed using the \(\mathrm{Repeat}\) rule, so no infinite branch would have been produced.
\end{proof}

Once we have interpolation templates for any sequent we are going to use the finite tree structure of the template (i.e., the template without cycles) to build a formula at each node \(w\).
This formula is called the \emph{pre-interpolant at \(w\)}.
We note it is not yet the interpolant, in particular because we are ignoring the cycles in its construction and each \(\mathrm{Repeat}\) node will introduce a variable that we will have to eliminate, as it will not belong to the vocabulary \(V\).
In particular, we are going to assume that for each \(\mathrm{Repeat}\) leaf \(w\) in the interpolation template we adjoin a new variable \(x_w\) to the set of propositional variables \(\mathrm{Var}\) (not to \(V\)).

\begin{definition}[Construction of pre-interpolant]
  Given an interpolation template \(T\) we construct a pre-interpolant \(\rho_w\) at each node \(w\) of \(\tau\) by induction in the (acyclic) tree structure of \(T\).
  If \(\Gamma_w \Rightarrow \Delta_w\) is the sequent at node \(w\) we will write \(\rho : \Gamma_w \Rightarrow \Delta_w\) to mean that \(\rho\) is the preinterpolant at node \(w\).
  Then, the preinterpolant is built using the following rules
  \[
    \AxiomC{}
    \UnaryInfC{\(\bot : p, \Gamma \Rightarrow p, \Delta \)}
    \DisplayProof
    \quad
    \AxiomC{}
    \UnaryInfC{\(\bot : \bot, \Gamma\Rightarrow \Delta\)}
    \DisplayProof
    \quad
    \AxiomC{}
    \UnaryInfC{\(\top : \varnothing \Rightarrow \varnothing\)}
    \DisplayProof
  \]

  \[
    \AxiomC{\(\rho_0 : \Gamma \Rightarrow \phi, \Delta\)}
    \AxiomC{\(\rho_1 : \psi,\Gamma \Rightarrow \Delta\)}
    \RightLabel{\(\to \mathrm{L}\)}
    \BinaryInfC{\(\rho_0 \vee \rho_1 : \phi \to \psi, \Gamma \Rightarrow \Delta\)}
    \DisplayProof
    \quad
    \AxiomC{\(\rho : \phi, \Gamma \Rightarrow \psi, \Delta\)}
    \RightLabel{\(\to \mathrm{R}\)}
    \UnaryInfC{\(\rho : \Gamma \Rightarrow \phi \to \psi, \Delta\)}
    \DisplayProof
  \]

  \[
    \AxiomC{}
    \LeftLabel{Current node is \(w\)}
    \RightLabel{Repeat}
    \UnaryInfC{\(x_w : \Gamma \Rightarrow \Delta\)}
    \DisplayProof
    \qquad
    \AxiomC{\(\rho :  \Gamma^s \Rightarrow \Delta^s\)}
    \RightLabel{Wk}
    \UnaryInfC{\(\rho :\Gamma \Rightarrow \Delta\)}
    \DisplayProof
  \]

  \[
    \AxiomC{\([\rho_{\Phi,\psi^\epsilon} : \psi^\epsilon, \Phi \interp \bot\Rightarrow \Phi]_{\Phi, \psi^\epsilon}\)}
    \RightLabel{\(\interp^*_{\mathrm{IK4}}\)}
    \UnaryInfC{\(\left(\bigwedge_{s \in S} \rho_s\right) \wedge \bigwedge (\Gamma \cap V) \wedge \bigwedge \neg (\Delta \cap V) : \Sigma, \Gamma \Rightarrow \Lambda , \Delta\)}
    \DisplayProof
  \]
  where \(S\) is the set of pairs of a sequence in \(\Sigma\) and either a formula in \(\Lambda\) or \(\epsilon\),\footnote{Since \(\Sigma\) is a multiset with a sequent in \(\Sigma\) we also need to take care to not repeat an element more times than its muiltiplicity in \(\Sigma\).} then given \(s = ((\phi_i \interp \psi_i)_{i < m}, \sigma^\epsilon) \in S\) we define \(\rho_s\) as follows:
  \begin{enumerate}
    \item If \(\sigma^\epsilon = \epsilon\)
      \[
        \neg \rho_{\Phi_{[0,m)}, \epsilon} \interp \bigvee_{i < m} \rho_{\Phi_{[0,i)}, \psi_i},
      \]

    \item If \(\sigma^\epsilon = \psi_m \interp \phi\) we define \(\rho_s\) as
      \[
        \bigvee_{i \leq m} \neg \left(
          \rho_{\Phi_{[0,i)} \cup \{\phi\}, \psi_i} \interp \neg \rho_{\Phi_{[0,i)} \cup \{\phi\}, \epsilon}
        \right).
      \]
  \end{enumerate}

\end{definition}

We state some easy properties of the pre-interpolant.

\begin{lemma}
  Let \(T\) be an interpolation template, \(w\) be a \(\mathrm{Repeat}\) node of \(T\) and \(v\) be a node of \(T\).
  We have that
  \begin{enumerate}
    \item If \(x_w\) occurs in \(\rho_v\), then \(v \leq w\).
    \item If \(v_0,\ldots, v_{k-1}\) are the children nodes of \(v\), \(\rho_{v_0}, \ldots, \rho_{v_{k-1}}\) are modalized in \(x_w\) and \(v\) is not a \(\mathrm{Repeat}\) node then \(\rho_v\) is modalized in \(x_w\).
    \item If \(v \leq w\) and the path from \(v\) to \(w\) goes through \(\interp^*_{\mathrm{IK4}}\), then \(\rho_v\) is modalized in \(x_w\).
  \end{enumerate}
\end{lemma}
\begin{proof}
  Proof of 1.
  By induction in the height of the subtree generated at \(w\).
  If \(w\) is axiomatic it must be the case that \(w\) is the repeat where \(x_w\) is introduced, i.e., \(v = w\).
  Otherwise, let \(v\) have children \(v_0,\ldots,v_{k-1}\) for \(k>0\).
  Since \(x_w\) occurs at \(\rho_v\) it must be the case (by looking at the definiton of \(\rho_v\) for cases \({\to}\mathrm{L}\), \({\to}\mathrm{R}\), \(\mathrm{Wk}\), \(\interp^*_{\mathrm{IK4}}\)) that \(x_w\) occurs in at least one of \(\rho_{v_0},\ldots,\rho_{v_{k-1}}\) (as only Repeat introduces bound variables), let us assume it occurs in \(v_i\).
  By the induction hypothesis, we obtain that \(v_i \leq w\) and then \(v \leq v_i\) give us the desired \(v \leq w\).

  Proof of 2.
  Trivial by looking at the possible definitions of \(\rho_v\) (they all preserve modalized bounded variables except Repeat).

  Proof of 3.
  By induction in the distance of \(v\) to the last application of a \(\interp^*_{\mathrm{IK4}}\) in the path from \(v\) to \(w\).
  If \(v\) is the conclusion of \(\interp^*_{\mathrm{IK4}}\) itself, note that all the bound variables are modalized in \(\rho_v\) by definition.
  Otherwise, \(v\) must be a non-axiomatic node (as in the path from \(v\) to \(w\) a \(\interp^*_{\mathrm{IK4}}\) should occur), let \(v_0,\ldots,v_{k-1}\) be its children (for \(k > 0\)).
  Note that there is an unique \(i < k\) such that \(v_i \leq w\) while for \(j \neq i\) we have that \(v_j\) and \(w\) are incomparable.
  Then \(\rho_{v_j}\) for \(j \neq i\) is modalized in \(x_w\), as \(v_j\) and \(w\) are incomparable we indeed have that \(x_w\) does not occur in \(\rho_{v_j}\) by the first point of this lemma.
  Also, \(x_w\) is modalized in \(\rho_{v_i}\) by the induction hypothesis, so \(\rho_{v_0}, \ldots, \rho_{v_{k-1}}\) are modalized in \(x_w\).
  As \(v\) is not a repeat (it is non-axiomatic) we can conclude, by the second point of this lemma, that \(\rho_v\) is modalized in \(x_w\), as desired.
\end{proof}

Finally, we are prepared to define the interpolant given by an interpolation template \(T\).

\begin{lemma}[Definition of interpolant]
  Let \(T\) be a interpolation template.
  Then the set
  \[
    \mathcal{E}_T := \{x_w = \rho_{w^\circ} \mid w \text{ Repeat node of \(T\)}\}
  \]
  is an orderable \((B,V)\)-modal equational system, where \(B = \{x_w \mid w \text{ Repeat node of \(T\)}\}\).
  The application of the unique solution (up to logical equivalence) of \(\mathcal{E}_T\) to \(\rho_w\) will be denoted \(\iota_w\).
  The \emph{interpolant of \(T\)} is defined as \(\iota_T := \iota_\epsilon\).
\end{lemma}
\begin{proof}
  The first step is to give an enumeration of the nodes of \(T\), \(w_0,\ldots,w_{k-1}\) such that if \(i \leq j\) then either \(w_i\) and \(w_j\) are incomparable or \(w_j \leq w_i\).\footnote{For example, it suffices to enumerate first all the leaves, then all the nodes whose generated subtree has height 1, then all nodes which generated subtree has height 2, and so on.}
  Then we obtain an enumeration \(x_0,\ldots,x_{m-1}\) of the variables in \(B\) as follows: traverse the list \(w_0,\ldots,w_{k-1}\) and whenever we are at \(w^\circ\), the cyclic companion of a node \(w\), add \(x_w\) to the end of the enumeration (in particular, if we are at the cyclic companion of multiple nodes \(w_{i_0},\ldots,w_{i_{n-1}}\) just add \(x_{w_{i_0}},\ldots,x_{w_{i_{n-1}}}\) at the end of the enumeration in an arbitrary order).
  Let \(v_i\) be the Repeat node corresponding to the variable \(x_i\), i.e., \(x_i = x_{v_i}\).
  Then the enumeration \(x_0,\ldots,x_{m-1}\) has the following property: for any \(i \leq j < m\) either
  \begin{enumerate}
    \item \(v^\circ_i\), \(v^\circ_j\) are incomparable, or
    \item \(v^\circ_j \leq v^\circ_i\).
  \end{enumerate}
  Let \(i \leq j < m\), we have to show that \(\rho_{v^\circ_j}\) is modalized in \(x_i\).
  If \(v^\circ_i\) and \(v^\circ_j\) are incomparable, then \(v_i\) and \(v^\circ_j\) are also incomparable (as \(v^\circ_i \leq v_i\) and \(v_i\) is a leaf) so \(x_i\) does not occur in \(\rho_{v^\circ_j}\) and then it is modalized.
  If \(v^\circ_j \leq v^\circ_i\), then \(v^\circ_j \leq v_i\) and the path from \(v^\circ_j\) to \(v_i\) must go through a \(\interp^*_{\mathrm{IK4}}\) rule, as the path from \(v^\circ_i\) to \(v_i\) must go through a \(\interp^*_{\mathrm{IK4}}\) in order to hit a repeat (all the other rules, when read bottom up, lower the size of the sequent, i.e., the size of the conclusion is strictly bigger than the size of the premises).
  Then, we know that \(\rho_{v^\circ_j}\) must be modalized in \(x_i\), as desired.
\end{proof}

\subsection{Verification}

We continue using the vocabulary \(V\) fixed at the start of the previous subsection.

We need to verify that our definition of uniform interpolant works.
In order to do this we need to show that certain proofs in \(\gentzenIL\) exists, this will be covered by Lemma~\ref{lm:first-verification} and Lemma~\ref{lm:second-verification}.
Instead of constructing the proofs directly in \(\gentzenIL\) we will use an auxiliary system \(\auxgentzenIL\), where the proofs are easier to construct.
The first thing we will need to show is that any proof in \(\auxgentzenIL\) can be transformed into a proof in \(\gentzenIL\), thus justifying the use of the auxiliary system.

\begin{definition}
We define the system \(\auxgentzenIL\) as \(\gentzenIL\) adding the rules the following rules
\[
  \AxiomC{\(\Gamma \Rightarrow \Delta\)}
  \RightLabel{Wk}
  \UnaryInfC{\(\Gamma, \Gamma' \Rightarrow \Delta, \Delta'\)}
  \DisplayProof
  \qquad
  \AxiomC{\(\Gamma \Rightarrow \Delta\)}
  \RightLabel{\(\equiv\)}
  \UnaryInfC{\(\Gamma' \Rightarrow \Delta'\)}
  \DisplayProof
  \qquad
  \AxiomC{\(\Gamma, \Gamma', \Gamma' \Rightarrow \Delta, \Delta', \Delta'\)}
  \RightLabel{Ctr}
  \UnaryInfC{\(\Gamma, \Gamma' \Rightarrow \Delta, \Delta'\)}
  \DisplayProof
\]
where
\begin{enumerate}
  \item In \(\mathrm{Wk}\) we have that \(\Gamma', \Delta'\) may be empty sets.
    In this particular case \(\mathrm{Wk}\) can be also be denoted as \(\mathrm{Eq}\).
  \item In \(\equiv\) either \(\Gamma = \Gamma'\), \(\Delta = \Delta_0, \phi\) and \(\Delta' = \Delta_0, \psi\) or  \(\Gamma = \Gamma_0, \phi\), \(\Gamma' = \Gamma_0, \psi\) and \(\Delta = \Delta'\), where \(\IL \vdash \phi \leftrightarrow \psi\).
  \item None of the new rules make progress.
\end{enumerate}
\end{definition}

\begin{lemma}
  If \(\auxgentzenIL \vdash S\) then \(\gentzenIL \vdash S\).
\end{lemma}
\begin{proof}
  First, we notice that we can corecursively define a translation \(\alpha\) of proofs in \(\auxgentzenIL\) into \(\gentzenIL + \cut + \Wk\).
  \(\alpha\) commutes with all the rules different from \(\equiv\) and \(\Ctr\).

  Assume \(\pi\) has the shape
  \[
    \AxiomC{\(\pi_0\)}
    \noLine
    \UnaryInfC{\(\phi, \Gamma \Rightarrow \Delta\)}
    \RightLabel{\(\equiv\)}
    \UnaryInfC{\(\psi, \Gamma \Rightarrow \Delta\)}
    \DisplayProof
  \]
  where \(\IL \vdash \phi \leftrightarrow \psi\), so in particular \(\IL \vdash \psi \to \phi\).
  Then, there is a proof \(\tau \vdash \psi \Rightarrow \phi\) in \(\gentzenIL\).
  Then the translation is defined as
  \[
    \AxiomC{\(\tau\)}
    \noLine
    \UnaryInfC{\(\psi \Rightarrow \phi\)}
    \RightLabel{\(\Wk\)}
    \UnaryInfC{\(\psi, \Gamma \Rightarrow \Delta, \phi\)}
    \AxiomC{\(\alpha(\pi_0)\)}
    \noLine
    \UnaryInfC{\(\phi, \Gamma \Rightarrow \Delta\)}
    \RightLabel{\(\Wk\)}
    \UnaryInfC{\(\phi, \psi, \Gamma \Rightarrow \Delta\)}
    \RightLabel{\(\cut\)}
    \BinaryInfC{\(\psi, \Gamma \Rightarrow \Delta\)}
    \DisplayProof
  \]
  The cases of other instances of the \(\equiv\) rule are handled similarly.

  Finally, we treat the contraction case.
  We will assume that we only contract one formula at the right side of the sequent, if we contract more than one formula (also on the left side) it can be handled similarly.
  So assume \(\pi\) has the following shape
  \[
    \AxiomC{\(\pi_0\)}
    \noLine
    \UnaryInfC{\(\Gamma \Rightarrow \phi, \phi, \Delta\)}
    \RightLabel{Ctr}
    \UnaryInfC{\(\Gamma \Rightarrow \phi, \Delta\)}
    \DisplayProof
  \]
  Then the translation is defined as
  \[
    \AxiomC{\(\pi_0\)}
    \noLine
    \UnaryInfC{\(\Gamma \Rightarrow \phi, \phi, \Delta\)}
    \AxiomC{}
    \RightLabel{\(\Ax\)}
    \UnaryInfC{\(\phi, \Gamma \Rightarrow \phi, \Delta\)}
    \BinaryInfC{\(\Gamma \Rightarrow \phi, \Delta\)}
    \DisplayProof
  \]

  Then, any proof in \(\auxgentzenIL\) can be transformed into a proof in \(\gentzenIL + \cut + \Wk\), since \(\Wk\) is eliminable in \(\gentzenIL + \cut\) we obtain a proof in \(\gentzenIL + \cut\) and then since \(\cut\) is eliminable in \(\gentzenIL\) we can obtain the desired in \(\gentzenIL\).
\end{proof}

In the following two lemmas, we claim that a proof of \(\gentzenIL\) exists.
In fact, we are going to show that proofs in \(\auxgentzenIL\) exists, but the previous lemma let us bridge the two systems.

\begin{lemma}
  \label{lm:first-verification}
  Let \(T\) be a interpolation template for \(\Gamma \Rightarrow \Delta\).
  Then \(\gentzenIL\vdash \Gamma \Rightarrow \Delta, \iota_T\).
\end{lemma}
\begin{proof}
  Given a node \(w\) of \(T\) let us write \(\Gamma_w \Rightarrow \Delta_w\) for the sequent at \(w\) in \(T\).
  We are going to define a function \(\alpha\) that given a node \(w\) of \(T\) returns a proof in \(\auxgentzenIL\) of \(\Gamma_w \Rightarrow \Delta_w, \iota_w\), where \(\iota_w\) is the interpolant at \(w\) in \(T\).
  We define \(\alpha\) corecursively in such a way that \(\alpha(w)\) is a preproof of \(\Gamma_w \Rightarrow \Delta_w, \iota_w\).
  Later, we will argue that this preproof is indeed a proof.
  We proceed by cases on the shape of \(w\).

  Case \(w\) is \(\ax\).
  We have that \(w\) is
  \[
    \AxiomC{}
    \RightLabel{\(\ax\)}
    \UnaryInfC{\(\bot : p, \Gamma \Rightarrow p, \Delta\)}
    \DisplayProof
  \]
  where \(\Gamma_w = p, \Gamma\), \(\Delta_w = p, \Delta\) and \(\rho_w = \bot\).
  Then \(\iota_w = \bot\) and the desired preproof is
  \[
    \AxiomC{}
    \RightLabel{\(\ax\)}
    \UnaryInfC{\(p, \Gamma \Rightarrow p, \Delta, \bot\)}
    \DisplayProof
  \]

  Case \(w\) is \(\botL\).
  We have that \(w\) is
  \[
    \AxiomC{}
    \RightLabel{\(\botL\)}
    \UnaryInfC{\(\bot : \bot, \Gamma \Rightarrow \Delta\)}
    \DisplayProof
  \]
  where \(\Gamma_w = \bot, \Gamma\), \(\Delta_w = \Delta\) and \(\rho_w = \bot\).
  Then \(\iota_w = \bot\) and the desired preproof is
  \[
    \AxiomC{}
    \RightLabel{\(\botL\)}
    \UnaryInfC{\(\bot, \Gamma \Rightarrow \Delta, \bot\)}
    \DisplayProof
  \]

  Case \(w\) is \(\emp\).
  We have that \(w\) is
  \[
    \AxiomC{}
    \RightLabel{\(\emp\)}
    \UnaryInfC{\(\top :  {\Rightarrow }\)}
    \DisplayProof
  \]
  where \(\Gamma_w = \varnothing\), \(\Delta_w = \varnothing\) and \(\rho_w = \top\).
  Then \(\iota_w = \top\) and the desired preproof is
  \[
    \AxiomC{}
    \RightLabel{\(\botL\)}
    \UnaryInfC{\(\bot \Rightarrow \bot\)}
    \RightLabel{\(\toR\)}
    \UnaryInfC{\(\Rightarrow \top\)}
    \DisplayProof
  \]

  Case \(w\) is \(\rep\).
  We have that \(w\) is
  \[
    \AxiomC{}
    \RightLabel{\(\rep\)}
    \UnaryInfC{\(x_w : \Gamma \Rightarrow \Delta\)}
    \DisplayProof
  \]
  where \(\Gamma_w = \Gamma\), \(\Delta_w = \Delta\), \(\rho_w = x_w\) and the cyclic companion of \(w\), \(w^\circ\) has sequent \(\Gamma \Rightarrow \Delta\).
  Let \(x_w \mapsto \chi_w\) be the solution of \(\mathcal{E}_T\), so \(\iota_w = \chi_w\) and \(\IL \vdash \chi_w \leftrightarrow \iota_{w^\circ}\).
  Then the desired preproof is
  \[
    \AxiomC{\(\alpha(w^\circ)\)}
    \noLine
    \UnaryInfC{\(\Gamma \Rightarrow \Delta, \iota_{w^\circ}\)}
    \RightLabel{\(\equiv\)}
    \UnaryInfC{\(\Gamma \Rightarrow \Delta, \chi_w\)}
    \DisplayProof
  \]

  Case \(w\) is \(\Wk\).
  Then \(w\) is of shape
  \[
    \AxiomC{\(w0\)}
    \noLine
    \UnaryInfC{\(\rho : \Gamma^s \Rightarrow \Delta^s\)}
    \RightLabel{\(\Wk\)}
    \UnaryInfC{\(\rho : \Gamma \Rightarrow \Delta\)}
    \DisplayProof
  \]
  where \(\Gamma_w = \Gamma\), \(\Delta_w = \Delta\), \(\Gamma_{w0} = \Gamma^s\), \(\Delta_{w0} = \Delta^s\) and \(\rho_{w} = \rho_{w0} = \rho\).
  Then \(\iota_{w} = \iota_{w0} = \iota\) and the desired preproof is
  \[
    \AxiomC{\(\alpha(w0)\)}
    \noLine
    \UnaryInfC{\(\Gamma^s \Rightarrow \Delta^s, \iota\)}
    \RightLabel{\(\Wk\)}
    \UnaryInfC{\(\Gamma \Rightarrow \Delta, \iota\)}
    \DisplayProof
  \] 

  Case \(w\) is \(\botR\).
  Then \(w\) is of shape
  \[
    \AxiomC{\(w0\)}
    \noLine
    \UnaryInfC{\(\rho : \Gamma \Rightarrow \Delta\)}
    \RightLabel{\(\Wk\)}
    \UnaryInfC{\(\rho : \Gamma \Rightarrow \Delta, \bot\)}
    \DisplayProof
  \]
  where \(\Gamma_w = \Gamma_{w0} = \Gamma\), \(\Delta_w = \Delta, \bot\),  \(\Delta_{w0} = \Delta\) and \(\rho_{w} = \rho_{w0} = \rho\).
  Then \(\iota_{w} = \iota_{w0} = \iota\) and the desired preproof is
  \[
    \AxiomC{\(\alpha(w0)\)}
    \noLine
    \UnaryInfC{\(\Gamma \Rightarrow \Delta, \iota\)}
    \RightLabel{\(\botR\)}
    \UnaryInfC{\(\Gamma \Rightarrow \Delta, \bot, \iota\)}
    \DisplayProof
  \]
  
  Case \(w\) is \(\toL\).
  Then \(w\) is of shape
  \[
    \AxiomC{\(w0\)}
    \noLine
    \UnaryInfC{\(\rho_0 : \Gamma \Rightarrow \Delta, \phi\)}
    \AxiomC{\(w1\)}
    \noLine
    \UnaryInfC{\(\rho_1 : \psi, \Gamma \Rightarrow \Delta\)}
    \RightLabel{\(\toL\)}
    \BinaryInfC{\(\rho_0 \vee \rho_1 : \phi \to \psi, \Gamma \Rightarrow \Delta\)}
    \DisplayProof
  \]
  where \(\Gamma_w = \phi \to \psi, \Gamma\) and \(\Delta_w = \Delta\).
  Then the desired preproof is
  \[
    \AxiomC{\(\alpha(w0)\)}
    \noLine
    \UnaryInfC{\(\Gamma \Rightarrow \Delta, \phi, \iota_{0}\)}
    \RightLabel{\(\Wk\)}
    \UnaryInfC{\(\Gamma \Rightarrow \Delta, \phi, \iota_{0}, \iota_1\)}
    \RightLabel{\(\veeR\)}
    \UnaryInfC{\(\Gamma \Rightarrow \Delta, \phi, \iota_{0}\vee \iota_1\)}
    \AxiomC{\(\alpha(w1)\)}
    \noLine
    \UnaryInfC{\(\psi, \Gamma \Rightarrow \Delta, \iota_{1}\)}
    \RightLabel{\(\Wk\)}
    \UnaryInfC{\(\psi, \Gamma \Rightarrow \Delta, \iota_0, \iota_{1}\)}
    \RightLabel{\(\veeR\)}
    \UnaryInfC{\(\psi, \Gamma \Rightarrow \Delta, \iota_0 \vee \iota_{1}\)}
    \RightLabel{\(\toL\)}
    \BinaryInfC{\(\phi \to \psi, \Gamma \Rightarrow \Delta, \iota_0 \vee \iota_1\)}
    \DisplayProof
  \]

  Case \(w\) is \(\toR\).
  Then \(w\) is of shape
  \[
    \AxiomC{\(w0\)}
    \noLine
    \UnaryInfC{\(\rho : \phi, \Gamma \Rightarrow \psi, \Delta\)}
    \RightLabel{\(\toR\)}
    \UnaryInfC{\(\rho : \Gamma \Rightarrow \phi \to \psi, \Delta\)}
    \DisplayProof
  \]
  where \(\Gamma_w = \Gamma\) and \(\Delta_{w} = \phi \to \psi, \Delta\).
  Then the desired preproof is
  \[
    \AxiomC{\(\alpha(w0)\)}
    \noLine
    \UnaryInfC{\(\phi, \Gamma \Rightarrow \psi, \Delta, \iota\)}
    \RightLabel{\(\toR\)}
    \UnaryInfC{\(\Gamma \Rightarrow \phi \to \psi, \Delta, \iota\)}
    \DisplayProof
  \]

  Case \(w\) is \(\ninv\).
  Then \(w\) is of shape
  \[
    \AxiomC{\(\left[
        \begin{matrix}
          w_{\Phi, \psi^\epsilon} \\
          \rho_{\Phi, \psi^\epsilon} : \psi^\epsilon, \Phi \interp \bot \Rightarrow \Phi
        \end{matrix}
    \right]_{\Phi, \psi^\epsilon}\)}
    \RightLabel{\(\ninv\)}
    \UnaryInfC{\(\rho : \Sigma, \Gamma \Rightarrow \Lambda, \Delta\)}
    \DisplayProof
  \]
  where \(\Gamma_w = \Sigma, \Gamma\), \(\Delta_w = \Lambda, \Delta\), \(\rho_w = \rho\), \(\Sigma, \Lambda\) are multisets of \(\interp\)-formulas, \(\Gamma, \Delta\) are multisets of propositional variables.

  By definition we know that \(\rho = \left(\bigwedge_{s \in S} \rho_s\right) \wedge  \bigwedge (\Gamma \cap V) \bigwedge \neg (\Delta \cap V)\), where \(S\) is the set of ordered pairs of a sequence of \(\Sigma\) and either a formula in \(\Lambda\) or \(\epsilon\).
  So \(\iota_w = \left(\bigwedge_{s \in S} \iota_s\right) \wedge  \bigwedge (\Gamma \cap V) \bigwedge \neg (\Delta \cap V)\).
  The desired preproof will be
  \[
    \AxiomC{\(\left[
      \begin{matrix}
        \tau_s \\
        \Sigma, \Gamma \Rightarrow \Lambda, \Delta, \iota_s
      \end{matrix}
    \right]_{s \in S}\)}
    \AxiomC{\(\left[
      \begin{matrix}
        \tau_{p} \\
        \Sigma, \Gamma \Rightarrow \Lambda, \Delta, p
      \end{matrix}
    \right]_{p \in \Gamma \cap V}\)}
    \AxiomC{\(\left[
      \begin{matrix}
        \tau_{\neg p} \\
        \Sigma, \Gamma \Rightarrow \Lambda, \Delta, \neg p
      \end{matrix}
    \right]_{p \in \Delta \cap V}\)}
    \doubleLine
    \RightLabel{\(\wedgeR\)}
    \TrinaryInfC{\(\Sigma, \Gamma \Rightarrow \Lambda, \Delta, \left(\bigwedge_{s \in S} \rho_s\right) \wedge  \bigwedge (\Gamma \cap V) \bigwedge \neg (\Delta \cap V)\)}
    \DisplayProof
  \]
  where we define the \(\tau_s\)'s, \(\tau_p\)'s and \(\tau_{\neg p}\)'s as follows.
  \begin{itemize}
    \item Definition of \(\tau_s\) for \(s = (\langle \phi_0 \interp \psi_0, \ldots, \phi_{m-1} \interp \psi_{m-1}\rangle, \epsilon)\).
      Then
      \[
        \iota_s = \neg \iota_{\Phi_{[0,m)}, \epsilon} \interp \bigvee_{i < m} \iota_{\Phi_{[0,i)}, \psi_i}.
      \]
      We define the proofs \(\tau_{s,i}\) for \(i \leq m\) as follows.
      First we define \(\tau_{s,m}\) as
      \[
        \AxiomC{\(\alpha(w_{\Phi_{[0,m)}, \epsilon})\)}
        \noLine
        \UnaryInfC{\(\Phi_{[0,m)} \interp \bot \Rightarrow \Phi_{[0,m)}, \iota_{\Phi_{[0,m)}, \epsilon}\)}
        \RightLabel{\(\negL\)}
        \UnaryInfC{\(\neg \iota_{\Phi_{[0,m)}, \epsilon}, \Phi_{[0,m)} \interp \bot \Rightarrow \Phi_{[0,m)}\)}
        \RightLabel{\(\Wk\)}
        \UnaryInfC{\(\neg \iota_{\Phi_{[0,m)}, \epsilon}, \left(\Phi_{[0,m)}, \bigvee_{i < m} \iota_{\Phi_{[0,i)}, \psi_i}\right) \interp \bot \Rightarrow \Phi_{[0,m)}, \bigvee_{i < m} \iota_{\Phi_{[0,i)}, \psi_i}\)}
        \DisplayProof
      \]
      and for \(j < m\) we define \(\tau_{s,j}\) as
      \[
        \AxiomC{\(\alpha(w_{\Phi_{[0,j)}, \psi_j})\)}
        \noLine
        \UnaryInfC{\(\psi_j, \Phi_{[0,j)} \interp \bot \Rightarrow \Phi_{[0,j)}, \iota_{\Phi_{[0,j)}, \psi_j}\)}
        \RightLabel{\(\Wk\)}
        \UnaryInfC{\(\psi_j, \left(\Phi_{[0,j)}, \bigvee_{i < m} \iota_{\Phi_{[0,i)}, \psi_i}\right) \interp \bot \Rightarrow \Phi_{[0,j)}, \left\{\iota_{\Phi_{[0,i)}, \psi_i}\right\}_{i < m}\)}
        \doubleLine
        \RightLabel{\(\veeR\)}
        \UnaryInfC{\(\psi_j, \left(\Phi_{[0,j)}, \bigvee_{i < m} \iota_{\Phi_{[0,i)}, \psi_i}\right) \interp \bot \Rightarrow \Phi_{[0,j)}, \bigvee_{i < m} \iota_{\Phi_{[0,i)}, \psi_i}\)}
        \DisplayProof
      \]
      Finally, we define \(\tau_s\) as follows
      \[
        \AxiomC{\(\tau_{s,m} \quad \cdots \quad \tau_{s,0}\)}
        \RightLabel{\(\interpIKT\)}
        \UnaryInfC{\(\phi_0 \interp \psi_0, \ldots, \phi_{m-1} \interp \psi_{m-1}, \Sigma_0, \Gamma \Rightarrow \Lambda, \Delta, \neg \iota_{\Phi_{[0,m)}, \epsilon} \interp \bigvee_{i < m} \iota_{\Phi_{[0,i)}, \psi_i}\)}
        \dashedLine
        \UnaryInfC{\(\Sigma, \Gamma \Rightarrow \Lambda, \Delta, \neg \iota_{\Phi_{[0,m)}, \epsilon} \interp \bigvee_{i < m} \iota_{\Phi_{[0,i)}, \psi_i}\)}
        \DisplayProof
      \]
      where \(\interpIKT\) was applied with order
      \[
        \phi_0 \interp \psi_0, \ldots, \phi_{m-1} \interp \psi_{m-1}
      \]
      and principal formula \(\neg \iota_{\Phi_{[0,m)}, \epsilon} \interp \bigvee_{i < m} \iota_{\Phi_{[0,i)}, \psi_i}\).

    \item Definition of \(\tau_s\) for \(s = (\langle \phi_0 \interp \psi_0, \ldots, \phi_{m-1} \interp \psi_{m-1}\rangle, \psi_m \interp \phi)\).
      Then 
      \[
        \iota_s = \bigvee_{i \leq m} \neg \left(\iota_{\Phi_{[0,i)} \cup \{\phi\}, \psi_i} \interp \neg \iota_{\Phi_{[0,i)} \cup \{\phi\},\epsilon}\right)
      \]
      Define \(\sigma_i = \iota_{\Phi_{[0,i)} \cup \{\phi\}, \psi_i} \interp \neg \iota_{\Phi_{[0,i)} \cup \{\phi\},\epsilon}\) for \(i \leq m\), so
      \[
        \iota_s = \bigvee_{i \leq m} \neg \sigma_i.
      \]
      Define the preproof \(\tau_{s,i}\) for \(i \leq m\) as
      \[
        \AxiomC{\(\alpha(w_{\Phi_{[0,i)} \cup \{\phi\}, \psi_i})\)}
        \noLine
        \UnaryInfC{\(\psi_i, \left(\Phi_{[0,i)}, \phi\right) \interp \bot \Rightarrow \Phi_{[0,i)}, \phi, \iota_{\Phi_{[0,i)} \cup \{\phi\}, \psi_i}\)}
        \RightLabel{\(\Wk\)}
        \UnaryInfC{\(\psi_i, \left(\Phi_{[0,i)}, \phi, \left\{\iota_{\Phi_{[0,j)} \cup \{\phi\}, \psi_j} \right\}_{j \leq i}\right) \interp \bot \Rightarrow \Phi_{[0,i)}, \phi, \left\{\iota_{\Phi_{[0,j)} \cup \{\phi\}, \psi_j}\right\}_{j \leq i}\)}
        \DisplayProof
      \]
      and the preproof \(\tau'_{s,i}\) for \(i \leq m\) as
      \[
        \AxiomC{\(\alpha(w_{\Phi_{[0,i)} \cup \{\phi\}, \epsilon})\)}
        \noLine
        \UnaryInfC{\(\left(\Phi_{[0,i)}, \phi\right) \interp \bot \Rightarrow \Phi_{[0,i)}, \phi, \iota_{\Phi_{[0,i)} \cup\{\phi\}, \epsilon}\)}
        \RightLabel{\(\negL\)}
        \UnaryInfC{\(\neg \iota_{\Phi_{[0,i)} \cup\{\phi\}, \epsilon},\left(\Phi_{[0,i)}, \phi\right) \interp \bot \Rightarrow \Phi_{[0,i)}, \phi\)}
        \RightLabel{\(\Wk\)}
        \UnaryInfC{\(\neg \iota_{\Phi_{[0,i)} \cup\{\phi\}, \epsilon},\left(\Phi_{[0,i)}, \phi, \left\{\iota_{\Phi_{[0,j)} \cup \{\phi\}, \psi_j} \right\}_{j < i}\right) \interp \bot \Rightarrow \Phi_{[0,i)}, \phi, \left\{\iota_{\Phi_{[0,j)} \cup \{\phi\}, \psi_j}\right\}_{j < i}\)}
        \DisplayProof
      \]
      Then the desired preproof is
      \[
        \AxiomC{\(\tau_{s, m} \quad \tau'_{s,m} \quad \cdots \quad \tau_{s, 0} \quad \tau'_{s, 0}\)}
        \RightLabel{\(\interpIKT\)}
        \UnaryInfC{\(\sigma_0, \phi_0 \interp \psi_0, \sigma_1, \phi_1 \interp \psi_1, \ldots, \phi_{m-1} \interp \psi_{m-1}, \sigma_m, \Sigma_0, \Gamma \Rightarrow \psi_m \interp \phi, \Lambda_0, \Delta\)}
        \doubleLine
        \RightLabel{\(\negR\)}
        \UnaryInfC{\(\Sigma, \Gamma \Rightarrow \Lambda, \Delta,  \left\{\neg \sigma_i\right\}_{i \leq m}\)}
        \doubleLine
        \RightLabel{\(\veeR\)}
        \UnaryInfC{\(\Sigma, \Gamma \Rightarrow \Lambda, \Delta, \bigvee_{i \leq m} \neg \sigma_i\)}
        \DisplayProof
      \]
      where \(\interpIKT\) was applied with order
      \[
        \sigma_0, \phi_0 \interp \psi_0, \sigma_1, \phi_1 \interp \psi_1, \ldots, \phi_{m-1} \interp \psi_{m-1}, \sigma_m,
      \]
      and principal formula \(\psi_m \interp \phi\).

    \item Definition of \(\tau_p\). \(\tau_p\) is just an application of \(\ax\), as \(p \in \Gamma\).
    \item Definition of \(\tau_{\neg p}\).
      \(\tau_{\neg p}\) is defined as
      \[
        \AxiomC{}
        \RightLabel{\(\ax\)}
        \UnaryInfC{\(p, \Sigma, \Gamma \Rightarrow \Lambda, \Delta\)}
        \RightLabel{\(\negR\)}
        \UnaryInfC{\(\Sigma, \Gamma \Rightarrow \Lambda, \Delta, \neg p\)}
        \DisplayProof
      \]
  \end{itemize}

  Let us argue that \(\alpha(w)\) is a proof and not only a preproof.
  To each pair \(w\) node of \(T\) assign a measure \(\omega |\Gamma_w \Rightarrow \Delta_w| + \ell(w)\), where \(|\Gamma_w \Rightarrow \Delta_w|\) is the size of the sequent and \(\ell(w)\) is the length of \(w\) as a sequence of \(\mathbb{N}\).
  We notice that every corecursive call of \(\alpha\) strictly decreases this measure except in the case \(w\) is \(\ninv\).
  However, in those cases between the conclusion of the preproof and the corecursive calls an instance of \(\interpIKT\) occurs.
  This fact guarantees that any infinite branch of \(\alpha(w)\) will go through the premise of a \(\interpIKT\) rule infinitely often, as desired.
\end{proof}

\begin{lemma}
  \label{lm:second-verification}
  Let \(T\) be a interpolation template for \(\Gamma \Rightarrow \Delta\) and \(\Gamma',\Delta'\) contain formulas only in vocabulary \(V\).
  Then \(\gentzenIL \vdash \Gamma, \Gamma'\Rightarrow \Delta, \Delta'\) implies \(\gentzenIL\vdash \iota_T, \Gamma' \Rightarrow \Delta'\).
\end{lemma}
\begin{proof}
  Given a node \(w\) of \(T\) let us write \(\Gamma_w \Rightarrow \Delta_w\) for the sequent at \(w\) in \(T\).
  We are going to define a function \(\beta\) that given a node \(w\) of \(T\) and a proof \(\pi \vdash \Gamma_w, \Gamma' \Rightarrow \Delta_w, \Delta'\) in \(\gentzenIL\) returns a proof in \(\auxgentzenIL\) of \(\iota_w, \Gamma' \Rightarrow \Delta'\).
  We define \(\beta\) corecurisvely in such a way that by definition \(\beta(w,\pi)\) is a preproof of \(\iota_w, \Gamma' \Rightarrow \Delta'\).
  Later, we will argue that this preproof is indeed a proof.
  We proceed by cases on the shape of \(w\).

  Case \(w\) is \(\ax\).
  Then \(w\) and \(\pi\) are of shape
  \[
    \AxiomC{}
    \RightLabel{\(\ax\)}
    \UnaryInfC{\(\bot : p, \Gamma \Rightarrow p, \Delta\)}
    \DisplayProof
    \qquad
    \AxiomC{\(\pi\)}
    \noLine
    \UnaryInfC{\(p, \Gamma, \Gamma' \Rightarrow p, \Delta, \Delta'\)}
    \DisplayProof
  \]
  where \(\Gamma_w = p, \Gamma\), \(\Delta_w = p, \Delta\) and \(\rho_w = \bot\).
  Then \(\iota_w = \bot\) and the desired proof is
  \[
    \AxiomC{}
    \RightLabel{\(\botL\)}
    \UnaryInfC{\(\bot, \Gamma' \Rightarrow \Delta'\)}
    \DisplayProof
  \]

  Case \(w\) is \(\botL\).
  Then \(w\) and \(\pi\) are of shape
  \[
    \AxiomC{}
    \RightLabel{\(\botL\)}
    \UnaryInfC{\(\bot : \bot, \Gamma \Rightarrow \Delta\)}
    \DisplayProof
    \qquad
    \AxiomC{\(\pi\)}
    \noLine
    \UnaryInfC{\(\bot, \Gamma, \Gamma' \Rightarrow \Delta, \Delta'\)}
    \DisplayProof
  \]
  where \(\Gamma_w = \bot, \Gamma\), \(\Delta_w = \Delta\) and \(\rho_w = \bot\).
  Then \(\iota_w = \bot\) and the desired proof is
  \[
    \AxiomC{}
    \RightLabel{\(\botL\)}
    \UnaryInfC{\(\bot, \Gamma' \Rightarrow \Delta'\)}
    \DisplayProof
  \]

  Case \(w\) is \(\emp\).
  Then \(w\) and \(\pi\) are of shape
  \[
    \AxiomC{}
    \RightLabel{\(\emp\)}
    \UnaryInfC{\(\top :  {\Rightarrow }\)}
    \DisplayProof
    \qquad
    \AxiomC{\(\pi\)}
    \noLine
    \UnaryInfC{\(\Gamma' \Rightarrow \Delta'\)}
    \DisplayProof
  \]
  where \(\Gamma_w = \varnothing\), \(\Delta_w = \varnothing\) and \(\rho_w = \top\).
  Then \(\iota_w = \top\) and the desired proof is
  \[
    \AxiomC{\(\pi\)}
    \noLine
    \UnaryInfC{\(\Gamma' \Rightarrow \Delta'\)}
    \RightLabel{\(\Wk\)}
    \UnaryInfC{\(\top, \Gamma' \Rightarrow \Delta'\)}
    \DisplayProof
  \]

  Case \(w\) is \(\rep\).
  Then \(w\) and \(\pi\) are of shape
  \[
    \AxiomC{}
    \RightLabel{\(\rep\)}
    \UnaryInfC{\(x_w :  \Gamma \Rightarrow \Delta\)}
    \DisplayProof
    \qquad
    \AxiomC{\(\pi\)}
    \noLine
    \UnaryInfC{\(\Gamma, \Gamma' \Rightarrow \Delta, \Delta'\)}
    \DisplayProof
  \]
  where \(\Gamma_w = \Gamma\), \(\Delta_w = \Delta\), \(\rho_w = x_w\) and the cyclic companion of \(w\), \(w^\circ\) has sequent \(\Gamma \Rightarrow \Delta\).
  Let \(x_w \mapsto \chi_w\) be the solution of \(\mathcal{E}_T\), so \(\iota_w = \chi_w\) and \(\IL \vdash \chi_w \leftrightarrow \iota_{w^\circ}\).
  Then the desired preproof is
  \[
    \AxiomC{\(\beta(w^\circ, \pi)\)}
    \noLine
    \UnaryInfC{\(\iota_{w^\circ}, \Gamma' \Rightarrow \Delta'\)}
    \RightLabel{\(\equiv\)}
    \UnaryInfC{\(\chi_w, \Gamma' \Rightarrow \Delta'\)}
    \DisplayProof
  \]

  Case \(w\) is \(\Wk\).
  Then \(w\) and \(\pi\) are of shape
  \[
    \AxiomC{\(w0\)}
    \noLine
    \UnaryInfC{\(\rho : \Gamma^s \Rightarrow \Delta^s\)}
    \RightLabel{\(\Wk\)}
    \UnaryInfC{\(\rho : \Gamma \Rightarrow \Delta\)}
    \DisplayProof
    \qquad
    \AxiomC{\(\pi\)}
    \noLine
    \UnaryInfC{\(\Gamma, \Gamma' \Rightarrow \Delta, \Delta'\)}
    \DisplayProof
  \]
  where \(\Gamma_w = \Gamma\), \(\Delta_w = \Delta\).
  We apply contraction to obtain a proof \(\pi' \vdash \Gamma^s, \Gamma' \Rightarrow \Delta^s, \Delta'\).
  The desired preproof is
  \[
    \AxiomC{\(\beta(w0, \pi')\)}
    \noLine
    \UnaryInfC{\(\iota, \Gamma' \Rightarrow \Delta'\)}
    \RightLabel{\(\Eq\)}
    \UnaryInfC{\(\iota, \Gamma' \Rightarrow \Delta'\)}
    \DisplayProof
  \]

  Case \(w\) is \(\botR\).
  Then \(w\) and \(\pi\) are of shape
  \[
    \AxiomC{\(w0\)}
    \noLine
    \UnaryInfC{\(\rho : \Gamma \Rightarrow \Delta\)}
    \RightLabel{\(\botR\)}
    \UnaryInfC{\(\rho : \Gamma \Rightarrow \Delta, \bot\)}
    \DisplayProof
    \qquad
    \AxiomC{\(\pi\)}
    \noLine
    \UnaryInfC{\(\Gamma, \Gamma' \Rightarrow \Delta, \bot, \Delta'\)}
    \DisplayProof
  \]
  where \(\Gamma_w = \Gamma\), \(\Delta_w = \Delta, \bot\).
  We apply invertibility of \(\botR\) to obtain a proof \(\pi' \vdash \Gamma, \Gamma' \Rightarrow \Delta, \Delta'\).
  The desired preproof is
  \[
    \AxiomC{\(\beta(w0, \pi')\)}
    \noLine
    \UnaryInfC{\(\iota, \Gamma' \Rightarrow \Delta'\)}
    \RightLabel{\(\Eq\)}
    \UnaryInfC{\(\iota, \Gamma' \Rightarrow \Delta'\)}
    \DisplayProof
  \]

  Case \(w\) is \(\toL\).
  Then \(w\) and \(\pi\) are of shape
  \[
    \AxiomC{\(w0\)}
    \noLine
    \UnaryInfC{\(\rho_0 : \Gamma \Rightarrow \phi, \Delta\)}
    \AxiomC{\(w1\)}
    \noLine
    \UnaryInfC{\(\rho_1 : \psi, \Gamma \Rightarrow \Delta\)}
    \RightLabel{\(\botR\)}
    \BinaryInfC{\(\rho_0 \vee \rho_1 : \phi \to \psi, \Gamma \Rightarrow \Delta\)}
    \DisplayProof
    \qquad
    \AxiomC{\(\pi\)}
    \noLine
    \UnaryInfC{\(\phi \to \psi, \Gamma, \Gamma' \Rightarrow \Delta, \Delta'\)}
    \DisplayProof
  \]
  where \(\Gamma_w = \phi \to \psi, \Gamma\), \(\Delta_w = \Delta\).
  We apply invertibility of \(\toL\) to obtain proofs \(\pi'_0 \vdash \Gamma, \Gamma' \Rightarrow \phi, \Delta, \Delta'\) and \(\pi'_1 \vdash  \psi, \Gamma, \Gamma' \Rightarrow \Delta, \Delta'\).
  The desired preproof is
  \[
    \AxiomC{\(\beta(w0, \pi'_0)\)}
    \noLine
    \UnaryInfC{\(\iota_0, \Gamma' \Rightarrow \Delta'\)}
    \AxiomC{\(\beta(w1, \pi'_1)\)}
    \noLine
    \UnaryInfC{\(\iota_1, \Gamma' \Rightarrow \Delta'\)}
    \RightLabel{\(\veeL\)}
    \BinaryInfC{\(\iota_0 \vee \iota_1, \Gamma' \Rightarrow \Delta'\)}
    \DisplayProof
  \]

  Case \(w\) is \(\toR\).
  Then \(w\) and \(\pi\) are of shape
  \[
    \AxiomC{\(w0\)}
    \noLine
    \UnaryInfC{\(\rho : \phi, \Gamma \Rightarrow \Delta, \psi\)}
    \RightLabel{\(\toR\)}
    \UnaryInfC{\(\rho : \Gamma \Rightarrow \Delta, \phi \to \psi\)}
    \DisplayProof
    \qquad
    \AxiomC{\(\pi\)}
    \noLine
    \UnaryInfC{\(\Gamma, \Gamma' \Rightarrow \Delta, \phi \to \psi, \Delta'\)}
    \DisplayProof
  \]
  where \(\Gamma_w = \Gamma\), \(\Delta_w = \Delta, \phi \to \psi\).
  We apply invertibility of \(\toR\) obtaining a proof \(\pi' \vdash \phi, \Gamma, \Gamma' \Rightarrow \Delta, \psi, \Delta'\).
  The desired proof is
  \[
    \AxiomC{\(\beta(w0, \pi')\)}
    \noLine
    \UnaryInfC{\(\iota, \Gamma' \Rightarrow \Delta'\)}
    \RightLabel{\(\Eq\)}
    \UnaryInfC{\(\iota, \Gamma' \Rightarrow \Delta'\)}
    \DisplayProof
  \]

  Case \(w\) is \(\ninv\).
  Then \(w\) is of shape
  \[
    \AxiomC{\(\left[
        \begin{matrix}
          w_{\Phi, \psi^\epsilon} \\
          \rho_{\Phi, \psi^\epsilon} : \psi^\epsilon, \Phi \interp \bot \Rightarrow \Phi
        \end{matrix}
    \right]_{\Phi, \psi^\epsilon}\)}
    \RightLabel{\(\ninv\)}
    \UnaryInfC{\(\rho : \Sigma, \Gamma \Rightarrow \Lambda, \Delta\)}
    \DisplayProof
  \]
  where \(\Gamma_w = \Sigma, \Gamma\), \(\Delta_w = \Lambda, \Delta\), \(\rho_w = \rho\), \(\Sigma, \Lambda\) are multisets of \(\interp\)-formulas, \(\Gamma, \Delta\) are multisets of propositional variables.

  By definition we know that \(\rho = \left(\bigwedge_{s \in S} \rho_s\right) \wedge  \bigwedge (\Gamma \cap V) \bigwedge \neg (\Delta \cap V)\), where \(S\) is the set of ordered pairs of a sequence of \(\Sigma\) and either a formula in \(\Lambda\) or \(\epsilon\).
  So \(\iota_w = \left(\bigwedge_{s \in S} \iota_s\right) \wedge  \bigwedge (\Gamma \cap V) \bigwedge \neg (\Delta \cap V)\).
  We proceed by cases analysis in the last rule applied to \(\pi\).
  
  Subcase last rule of \(\pi\) is \(\ax\).
  Then \(\pi\) is of shape
  \[
    \AxiomC{}
    \RightLabel{\(\ax\)}
    \UnaryInfC{\(\Sigma, \Gamma, \Gamma' \Rightarrow \Lambda, \Delta, \Delta'\)}
    \DisplayProof
  \]
  where a propositional variable \(p\) must appear on the left side and right side of the sequent.
  It cannot be the case that \(p \in \Gamma \cap \Delta\), since then \(w\) would have rule \(\ax\) instead of \(\ninv\).
  This leaves three options left:
  \begin{enumerate}
    \item If \(p \in \Gamma \cap \Delta'\)  then \(p \in V\) as the vocabulary of \(\Delta'\) is \(V\), so \(p\) is a conjunct of \(\iota\).
      The desired proof is
      \[
        \AxiomC{}
        \RightLabel{\(\ax\)}
        \UnaryInfC{\(p, \Gamma' \Rightarrow \Delta'\)}
        \doubleLine
        \RightLabel{\(\Wk + \wedgeL\)}
        \UnaryInfC{\(\iota, \Gamma' \Rightarrow \Delta'\)}
        \DisplayProof
      \]
    \item If \(p \in \Gamma' \cap \Delta\) then \(p \in V\) as the vocabulary of \(\Gamma'\) is \(V\), so \(\neg p\) is a conjunct of \(\iota\).
      The desired proof is
      \[
        \AxiomC{}
        \RightLabel{\(\ax\)}
        \UnaryInfC{\(\Gamma' \Rightarrow \Delta', p\)}
        \RightLabel{\(\negL\)}
        \UnaryInfC{\(\neg p, \Gamma' \Rightarrow \Delta'\)}
        \doubleLine
        \RightLabel{\(\Wk + \wedgeL\)}
        \UnaryInfC{\(\iota, \Gamma' \Rightarrow \Delta'\)}
        \DisplayProof
      \]
    \item If \(p \in \Gamma' \cap \Delta'\) then the desired proof is
      \[
        \AxiomC{}
        \RightLabel{\(\ax\)}
        \UnaryInfC{\(\iota, \Gamma' \Rightarrow \Delta'\)}
        \DisplayProof
      \]
  \end{enumerate}

  Subcase last rule of \(\pi\) is \(\botL\).
  Then \(\pi\) is of shape
  \[
    \AxiomC{}
    \RightLabel{\(\botL\)}
    \UnaryInfC{\(\Sigma, \Gamma, \Gamma' \Rightarrow \Lambda, \Delta, \Delta'\)}
    \DisplayProof
  \]
  where \(\bot \in \Gamma'\) (\(\Sigma\) consists of \(\interp\)-formulas only and \(\Gamma\) of propositional variables).
  The desired proof is
  \[
    \AxiomC{}
    \RightLabel{\(\botL\)}
    \UnaryInfC{\(\iota, \Gamma' \Rightarrow \Delta'\)}
    \DisplayProof
  \]

  Subcase last rule of \(\pi\) is \(\botR\).
  Then \(\pi\) is of shape
  \[
    \AxiomC{\(\pi_0\)}
    \noLine
    \UnaryInfC{\(\Sigma, \Gamma, \Gamma' \Rightarrow \Lambda, \Delta, \Delta'_0\)}
    \RightLabel{\(\botR\)}
    \UnaryInfC{\(\Sigma, \Gamma, \Gamma' \Rightarrow \Lambda, \Delta, \bot, \Delta'_0\)}
    \DisplayProof
  \]
  where \(\Delta' = \bot, \Delta'_0\) (as \(\bot\) cannot occur in \(\Lambda\) nor in \(\Delta\)).
  Then the desired proof is
  \[
    \AxiomC{\(\beta(w,\pi_0)\)}
    \noLine
    \UnaryInfC{\(\iota, \Gamma' \Rightarrow \Delta'_0\)}
    \RightLabel{\(\botR\)}
    \UnaryInfC{\(\iota, \Gamma' \Rightarrow \bot, \Delta'_0\)}
    \DisplayProof
  \]

  Subcase last rule of \(\pi\) is \(\toL\).
  Then \(\pi\) is of shape
  \[
    \AxiomC{\(\pi_0\)}
    \noLine
    \UnaryInfC{\(\Sigma, \Gamma, \Gamma'_0 \Rightarrow \Lambda, \Delta, \phi', \Delta'\)}
    \AxiomC{\(\pi_0\)}
    \noLine
    \UnaryInfC{\(\Sigma, \Gamma, \psi', \Gamma'_0 \Rightarrow \Lambda, \Delta, \Delta'\)}
    \RightLabel{\(\toL\)}
    \BinaryInfC{\(\Sigma, \Gamma, \phi' \to \psi', \Gamma'_0 \Rightarrow \Lambda, \Delta, \Delta'\)}
    \DisplayProof
  \]
  where \(\Gamma' = \phi' \to \psi', \Gamma'_0\) (as an implication cannot occur in \(\Sigma\) nor in \(\Gamma\)).
  Then the desired proof is
  \[
    \AxiomC{\(\beta(w,\pi_0)\)}
    \noLine
    \UnaryInfC{\(\iota,  \Gamma'_0 \Rightarrow \phi', \Delta'\)}
    \AxiomC{\(\beta(w,\pi_0)\)}
    \noLine
    \UnaryInfC{\(\iota, \psi', \Gamma'_0 \Rightarrow  \Delta'\)}
    \RightLabel{\(\toL\)}
    \BinaryInfC{\(\iota, \phi' \to \psi', \Gamma'_0 \Rightarrow  \Delta'\)}
    \DisplayProof
  \]

  Subcase last rule of \(\pi\) is \(\toR\).
  Then \(\pi\) is of shape
  \[
    \AxiomC{\(\pi_0\)}
    \noLine
    \UnaryInfC{\(\Sigma, \Gamma, \phi', \Gamma' \Rightarrow \Lambda, \Delta, \psi', \Delta'_0\)}
    \RightLabel{\(\botR\)}
    \UnaryInfC{\(\Sigma, \Gamma, \Gamma' \Rightarrow \Lambda, \Delta, \phi' \to \psi', \Delta'_0\)}
    \DisplayProof
  \]
  where \(\Delta' = \phi' \to \psi', \Delta'_0\) (as an implication cannot occur in \(\Lambda\) nor in \(\Delta\)).
  Then the desired proof is
  \[
    \AxiomC{\(\beta(w,\pi_0)\)}
    \noLine
    \UnaryInfC{\(\iota, \phi', \Gamma' \Rightarrow \psi', \Delta'_0\)}
    \RightLabel{\(\toR\)}
    \UnaryInfC{\(\iota, \Gamma' \Rightarrow \phi' \to \psi', \Delta'_0\)}
    \DisplayProof
  \]

  Subcase last rule of \(\pi\) is \(\interpIKT\).
  Then \(\pi\) is of shape
  \[
    \AxiomC{\(\left[
        \begin{matrix}
          \pi_i \\
          \psi_i, (\Phi_{[0,i)}, \phi) \interp \bot \Rightarrow \Phi_{[0,i)}, \phi
        \end{matrix}
    \right]_{i \leq k}\)}
    \RightLabel{\(\interpIKT\)}
    \UnaryInfC{\(\Sigma, \Gamma, \Gamma' \Rightarrow \Lambda, \Delta, \Delta'\)}
    \DisplayProof
  \]
  where \(\interpIKT\) has been applied with ordering \(\phi_0 \interp \psi_0, \ldots, \phi_{k-1} \interp \psi_{k-1}\) and principal formula \(\psi_{k} \interp \phi\).
  We know that each \(\phi_i \interp \psi_i\) belongs to \(\Sigma \cup \Gamma'\).
  We can divide the sequence in two as
  \begin{align*}
    &\phi_{i_0} \interp \psi_{i_0}, \ldots, \phi_{i_{m-1}} \interp \psi_{i_{m-1}}, \\
    &\phi_{j_0} \interp \psi_{j_0}, \ldots, \phi_{j_{n-1}} \interp \psi_{j_{n-1}},
  \end{align*}
  such that
  \begin{enumerate}
    \item if \(x' < x < m\) then \(i_{x'} < i_x\) and if \(y' < y < n\) then \(j_{y'} < j_y\),
    \item \(\langle \phi_{i_x} \interp \psi_{i_x}\rangle_{x < m}\) is a sequence of \(\Sigma\)  and \(\langle \phi_{j_y} \interp \psi_{j_y}\rangle_{y < n}\) is a sequence of \(\Gamma'\).
      We note that we take into account repetitions, i.e., if \(\phi_{i_{x'}} \interp \psi_{i_{x'}}\) cannot occur more in the sequence \(\langle \phi_{i_x} \interp \psi_{i_x}\rangle_{x < m}\) than it occurs in \(\Sigma\).
      Similarly for the sequence \(\langle \phi_{j_y} \interp \psi_{j_y}\rangle_{y < n}\).
  \end{enumerate}
  We will need the following auxiliary definitons 
  \begin{align*}
    &\Phi^i_{I} = \{\phi_{i_x} \mid x \in I\}, &&\Phi^j_{I} = \{\phi_{j_y} \mid y \in I\}, \\
    &x_y = \max \left(\{x < m \mid i_x < j_y\} \cup \{-1\}\right),
    &&y_x = \max \left(\{y < n \mid j_y < i_x\} \cup \{-1\}\right).
  \end{align*}
  In words, \(x_y\) is just the biggest \(x\) such that \(\phi_{i_x} \interp \psi_{i_x}\) occurs before \(\phi_{j_y} \interp \psi_{j_y}\) in the original order, or \(-1\) in such a \(x\) does not exists.
  For \(y_x\) the situation is analogous.
  This definitions allow us given a set \(\Phi_{[0,i)}\) split it into its \(i_x\)-part and its \(j_x\)-part.
  More precisely, for \(x < m\) we have that \(\Phi_{[0,i_x)} = \Phi^i_{[0,x)}, \Phi^j_{[0,y_x]}\) and for \(y < n\) we have that \(\Phi_{[0,j_y)} = \Phi^i_{[0,x_y]}, \Phi^j_{[0,y)}\).
  Finally, we can turn to the definition of the preproof \(\beta(w,\pi)\). 
  It depends on \(\psi_m \interp \phi \in \Lambda\) or not.
  \begin{itemize}
    \item Assume \(\psi_k \interp \phi \in \Lambda\).
      In this case we define \(\psi_{i_m} = \psi_k\) and \(s = (\langle \phi_{i_x} \interp \psi_{i_x} \rangle_{x < m}, \psi_{i_m} \interp \phi)\).
      The desired preproof is
      \[
        \AxiomC{\(\cdots\)}
        \AxiomC{\(\tau^x_{y_x + 1} \quad \cdots \quad \tau^x_0\)}
        \RightLabel{\(\interpIKT\)}
        \UnaryInfC{\(\Gamma' \Rightarrow \Delta', \iota_{\Phi^i_{[0,x)} \cup \{\phi\}, \psi_{i_x}} \interp \neg \iota_{\Phi^i_{[0,x)} \cup \{\phi\}, \epsilon}\)}
        \UnaryInfC{\(\neg (\iota_{\Phi^i_{[0,x)} \cup \{\phi\}, \psi_{i_x}} \interp \neg \iota_{\Phi^i_{[0,x)} \cup \{\phi\}, \epsilon}), \Gamma' \Rightarrow \Delta'\)}
        \AxiomC{\(\cdots\)}
        \doubleLine
        \RightLabel{\(\veeL\)}
        \TrinaryInfC{\(\bigvee_{x \leq m} \neg (\iota_{\Phi^i_{[0,x)} \cup \{\phi\}, \psi_{i_x}} \interp \neg \iota_{\Phi^i_{[0,x)} \cup \{\phi\}, \epsilon}), \Gamma' \Rightarrow \Delta'\)}
        \dashedLine
        \UnaryInfC{\(\iota_s, \Gamma' \Rightarrow \Delta'\)}
        \doubleLine\RightLabel{\(\Wk + \wedgeL\)}
        \UnaryInfC{\(\iota, \Gamma' \Rightarrow \Delta'\)}
        \DisplayProof
      \]
      where the displayed \(\interpIKT\) has been applied with ordering \(\phi_{j_0} \interp \psi_{j_0}, \ldots, \phi_{j_{y_x}} \interp \psi_{j_{y_x}}\) and principal formula \(\iota_{\Phi^i_{[0,x)} \cup \{\phi\}, \psi_{i_x}} \interp \neg \iota_{\Phi^i_{[0,x)} \cup \{\phi\}, \epsilon}\).
      To fully define the preproof, we still need to define \(\tau^x_y\) for \(y \leq y_{x} + 1\).
      First, to define \(\tau^x_{y_x+1}\) we remember that
      \[
        \pi_{i_x} \vdash \psi_{i_x}, (\Phi_{[0,i_x)}, \phi) \interp \bot \Rightarrow \Phi_{[0,i_x)}, \phi,
      \]
      or equivalently
      \[
        \pi_{i_x} \vdash \psi_{i_x}, (\Phi^i_{[0,x)}, \Phi^j_{[0,y_x]}, \phi) \interp \bot \Rightarrow \Phi^i_{[0,x)}, \Phi^j_{[0,y_x]},\phi.
      \]
      Then definition of \(\tau^x_{y_x + 1}\) is
      \[
        \AxiomC{\(\beta(w_{\Phi^{i}_{[0,x)} \cup \{\phi\}, \psi_{i_x}}, \pi_{i_x})\)}
        \noLine
        \UnaryInfC{\(\iota_{\Phi^{i}_{[0,x)} \cup \{\phi\}, \psi_{i_x}}, \Phi^j_{[0,y_x]} \interp \bot \Rightarrow \Phi^j_{[0,y_x]}\)}
        \RightLabel{\(\Wk\)}
        \UnaryInfC{\(\iota_{\Phi^{i}_{[0,x)} \cup \{\phi\}, \psi_{i_x}}, (\Phi^j_{[0,y_x]}, \neg \iota_{\Phi^i_{[0,x)} \cup \{\phi\}, \epsilon}) \interp \bot \Rightarrow \Phi^j_{[0,y_x]}, \neg \iota_{\Phi^i_{[0,x)} \cup \{\phi\}, \epsilon}\)}
        \DisplayProof
      \]
      Finally, we are going to define \(\tau^x_y\) for \(0 \leq y \leq y_x\).
      We remember that
      \[
        \pi_{j_y} \vdash \psi_{j_y}, (\Phi_{[0,j_y)}, \phi) \interp \bot \Rightarrow \Phi_{[0,j_y)}, \phi
      \]
      or equivalently
      \[
        \pi_{j_y} \vdash \psi_{j_y}, (\Phi^i_{[0,x_y]}, \Phi^j_{[0,y)}, \phi) \interp \bot \Rightarrow \Phi^i_{[0,x_y]}, \Phi^j_{[0,y)}, \phi.
      \]
      We notice that \(\Phi^i_{[0,x_y]} \subseteq \Phi^i_{[0,x)}\) since \(0 \leq y \leq y_x\).
      By an applying of weakening we obtain a proof
      \[
        \pi'_{j_y} \vdash \psi_{j_y}, (\Phi^i_{[0,x)}, \Phi^j_{[0,y)}, \phi) \interp \bot \Rightarrow \Phi^i_{[0,x)}, \Phi^j_{[0,y)}, \phi.
      \]
      Then the definiton of \(\tau^x_y\) is
      \[
        \AxiomC{\(\beta(w_{\Phi^i_{[0,x)} \cup \{\phi\}, \epsilon}, \pi'_{j_y})\)}
        \noLine
        \UnaryInfC{\(\iota_{\Phi^i_{[0,x)} \cup \{\phi\}, \epsilon}, \psi_{j_y}, \Phi^j_{[0,y)} \interp \bot \Rightarrow \Phi^j_{[0,y)}\)}
        \RightLabel{\(\negR\)}
        \UnaryInfC{\(\psi_{j_y}, \Phi^j_{[0,y)} \interp \bot \Rightarrow \Phi^j_{[0,y)}, \neg \iota_{\Phi^i_{[0,x)} \cup \{\phi\}, \epsilon}\)}
        \RightLabel{\(\Wk\)}
        \UnaryInfC{\(\psi_{j_y}, (\Phi^j_{[0,y)}, \neg \iota_{\Phi^i_{[0,x)} \cup \{\phi\}, \epsilon}) \interp \bot \Rightarrow \Phi^j_{[0,y)}, \neg \iota_{\Phi^i_{[0,x)} \cup \{\phi\}, \epsilon}\)}
        \DisplayProof
      \]

    \item Assume \(\psi_k \interp \phi \not \in \Lambda\), so \(\psi_m \interp \phi \in \Delta'\).
      In this case we define \(\psi_{j_n} = \psi_k\) and for \(y \leq n\) define \(s_y = (\langle \phi_{i_{x}} \interp \psi_{i_{x}} \rangle_{x < x_y + 1}, \epsilon)\), i.e., \(s_y\) is the sequence of \(\phi_{i_{x}} \interp \psi_{i_x}\) that occurs before position \(j_y\) in the original sequence.
      Then
      \[
        \iota_{s_y} = \neg\iota_{\Phi^i_{[0,x_y]},\epsilon} \interp \bigvee_{x \leq x_y} \iota_{\Phi^i_{[0,x)},\psi_{i_{x}}}
      \]
      where we used that \([0,x_y] = [0, x_y + 1)\) and \(x \leq x_y\) iff \(x < x_y + 1\).
      The desired preproof is
      \[
        \AxiomC{\(\tau_{n} \quad \tau'_n \quad \cdots \quad \tau_0 \quad \tau'_0\)}
        \RightLabel{\(\interpIKT\)}
        \UnaryInfC{\(\iota_{s_{0}}, \ldots, \iota_{s_{n}}, \Gamma' \Rightarrow \Delta'\)}
        \doubleLine\RightLabel{\(\Ctr + \Wk + \wedgeL\)}
        \UnaryInfC{\(\iota, \Gamma' \Rightarrow \Delta'\)}
        \DisplayProof
      \]
      where \(\interpIKT\) has been applied with ordering \(\iota_{s_{0}}, \phi_{j_0} \interp \psi_{j_0}, \iota_{s_{1}} \phi_{j_1} \interp \psi_{j_1}, \ldots, \phi_{j_{n-1}} \interp \psi_{j_{n-1}}, \iota_{s_{n}}\) and principal formula \(\psi_{j_n} \interp \phi\).
      To fully define the preprof, we still need to define the \(\tau_y\) and \(\tau'_y\) for \(y \leq n\).
      First, to define \(\tau_y\), we remember that
      \[
        \pi_{j_y} \vdash \psi_{j_y}, (\Phi_{[0,j_y)}, \phi) \interp \bot \Rightarrow \Phi_{[0,j_y)}, \phi,
      \]
      or equivalently
      \[
        \pi_{j_y} \vdash \psi_{j_y}, (\Phi^i_{[0,x_y]}, \Phi^j_{[0,y)}, \phi) \interp \bot \Rightarrow \Phi^i_{[0,x_y]}, \Phi^j_{[0,y)}, \phi.
      \]
      Then the definition of \(\tau_y\) is
      \[
        \AxiomC{\(\beta(w_{\Phi^i_{[0,x_{y}]}, \epsilon}, \pi_{j_y})\)}
        \noLine
        \UnaryInfC{\(\iota_{\Phi^i_{[0,x_{y}]}, \epsilon}, \psi_{j_y}, \left(\Phi^j_{[0,y)}, \phi\right) \interp \bot \Rightarrow\Phi^j_{[0,y)}, \phi\)}
        \UnaryInfC{\(\psi_{j_y}, \left(\Phi^j_{[0,y)}, \phi\right) \interp \bot \Rightarrow\Phi^j_{[0,y)}, \phi, \neg \iota_{\Phi^i_{[0,x_{y}]}, \epsilon}\)}
        \RightLabel{\(\Wk\)}
        \UnaryInfC{\(\psi_{j_y}, \left(\Phi^j_{[0,y)}, \phi, \left\{\neg \iota_{\Phi^i_{[0,x_{y'}]}, \epsilon}\right\}_{y' \leq y}\right) \interp \bot \Rightarrow\Phi^j_{[0,y)}, \phi, \left\{\neg \iota_{\Phi^i_{[0,x_{y'}]}, \epsilon}\right\}_{y' \leq y}\)}
        \DisplayProof
      \]

      Finally, to define \(\tau'_y\), we remember that
      \[
        \pi_{i_x} \vdash \psi_{i_x}, (\Phi_{[0,i_x)}, \phi) \interp \bot \Rightarrow \Phi_{[0,i_x)}, \phi,
      \]
      or equivalently
      \[
        \pi_{i_x} \vdash \psi_{i_x}, (\Phi^i_{[0,x)},\Phi^j_{[0,y_x]}, \phi) \interp \bot \Rightarrow \Phi^i_{[0,x)},\Phi^j_{[0,y_x]}, \phi.
      \]
      Notice that \(\Phi^j_{[0,y_x]} \subseteq \Phi^j_{[0,y)}\) for \(x \leq x_y\).
      Then the definition of \(\tau'_y\) is
      \[
        \AxiomC{\(\cdots\)}
        \AxiomC{\(\beta(w_{\Phi^i_{[0,x)}, \psi_{i_x}}, \pi_{i_x})\)}
        \noLine
        \UnaryInfC{\(\iota_{\Phi^i_{[0,x)}, \psi_{i_x}}, \left(\Phi^j_{[0,y_x]}, \phi\right) \interp \bot \Rightarrow\Phi^j_{[0,y_x]}, \phi\)}
        \RightLabel{\(\Wk\)}
        \UnaryInfC{\(\iota_{\Phi^i_{[0,x)}, \psi_{i_x}}, \left(\Phi^j_{[0,y)}, \phi\right) \interp \bot \Rightarrow\Phi^j_{[0,y)}, \phi\)}
        \AxiomC{\(\cdots\)}
        \doubleLine\RightLabel{\(\veeL\)}
        \TrinaryInfC{\(\bigvee_{x \leq x_y} \iota_{\Phi^i_{[0,x)}, \psi_{i_x}}, \left(\Phi^j_{[0,y)}, \phi\right) \interp \bot \Rightarrow\Phi^j_{[0,y)}, \phi\)}
        \RightLabel{\(\Wk\)}
        \UnaryInfC{\(\bigvee_{x \leq x_y} \iota_{\Phi^i_{[0,x)}, \psi_{i_x}}, \left(\Phi^j_{[0,y)}, \phi, \left\{\neg \iota_{\Phi^i_{[0,x_{y'}]}, \epsilon}\right\}_{y' < y}\right) \interp \bot \Rightarrow\Phi^j_{[0,y)}, \phi, \left\{\neg \iota_{\Phi^i_{[0,x_{y'}]}, \epsilon}\right\}_{y' < y}\)}
        \DisplayProof
      \]
  \end{itemize}
  With this we finish the definition of the function \(\beta\).

  Let us argue that \(\beta(w,\pi)\) is a proof and not only a preproof.
  To each pair \(\langle w, \pi\rangle\) of a node \(T\) and a proof \(\pi\) assign a meausre \(\omega^2 |\Gamma_w \Rightarrow \Delta_w| + \omega \ell(w) + \lheight(\pi)\), where \(|\Gamma_w \Rightarrow \Delta_w|\) is the size of the sequent, \(\ell(w)\) is the length of \(w\) as a sequence of \(\mathbb{N}\) and \(\lheight(\pi)\) is the local height of \(\pi\).
  We notice that every corecursive call of \(\beta\) strictly decreases this measure except in the case \(w\) is \(\ninv\) and the last rule of \(\pi\) is \(\interpIKT\).
  However, in those cases between the conclusion of the preproof and the corecursive calls an instance of \(\interpIKT\) occurs.
  This fact guarantees that any infinite branch of \(\beta(w,\pi)\) will go through the premise of a \(\interpIKT\) rule infinitely often, as desired.
\end{proof}

\begin{theorem}
  \label{cor:uniform-interpolation-IL}
  \(\IL\) has uniform interpolation.
\end{theorem}
\begin{proof}
  Let \(T\) be a \(V\)-interpolation scheme of \(\phi \Rightarrow\) and \(\iota\) its interpolant, i.e., \(\iota = \iota_T\).
  It is clear that \(\vocab(\iota) \subseteq V\) and thanks to the first veritification property we have that \(\gentzenIL \vdash \phi \Rightarrow \iota\) so \(\IL \vdash \phi \to \iota\).
  Given \(\psi\) with \(\vocab(\psi) \subseteq V\) and \(\IL \vdash \phi \to \psi\), so we would have that \(\gentzenIL \vdash \phi \Rightarrow \psi\).
  Then, by the second veritification property, we obtain \(\gentzenIL \vdash \iota \Rightarrow \psi\) so we can conclude \(\IL \vdash \iota \to \psi\).
\end{proof}

\subsection{Uniform interpolation of \(\ILP\)}
\label{subsec:ILP}

\begin{definition}
  \(\ILP\) is obtained from \(\IL\) by adding the axiom (scheme)
  \[
    \tag{P}
    \phi \interp \psi \to \nec(\phi \interp \psi).
  \]
\end{definition}

In \cite{Areces1998} it is proven that \(\ILP\) can be strongly interpreted in \(\IL\).
Then the interpretation is used to lift the result of Craig interpolation in \(\IL\) to Craig interpolation in \(\ILP\).
Let us define the interpretation from \(\IL\) to \(\ILP\) and show that it also suffices to prove uniform interpolation in \(\ILP\) from uniform interpolation in \(\IL\).

\begin{definition}
  We define the function \((\bullet)^\sharp\) as
  \begin{align*}
    &p^\sharp := p, \\
    &\bot^\sharp := \bot, \\
    &(\phi \to \psi)^\sharp := \phi^\sharp \to \psi^\sharp, \\
    &(\phi \interp \psi)^\sharp := \dnec(\phi^\sharp \interp \psi^\sharp).
  \end{align*}
\end{definition}

Notice that then
\begin{align*}
  &(\nec \phi)^\sharp
  = (\neg \phi \interp \bot)^\sharp
  = \dnec (\neg \phi^\sharp \interp \bot)
  = \dnec\nec \phi^\sharp, \\
  &(\pos \phi)^\sharp
  = (\neg (\phi \interp \bot))^\sharp
  = \neg \dnec (\phi^\sharp \interp \bot)
  = \pos \phi^\sharp \vee \pos \pos \phi^\sharp.
\end{align*}

\begin{proposition}
  \label{prop:properties-of-interpretation}
  We have that for any formula \(\phi\)
  \begin{enumerate}
    \item \(\vocab(\phi) = \vocab(\phi^\sharp)\).
    \item \(\ILP \vdash \phi \leftrightarrow \phi^\sharp\),
    \item \(\ILP \vdash \phi\) implies \(\IL \vdash \phi^\sharp\).
  \end{enumerate}
\end{proposition}
\begin{proof}
  Proof of 1. Trivial by induction on \(\phi\).

  Proof of 2.
  By induction on \(\phi\), if \(\phi\) is atomic or an implication the result is trivial.
  Finally, assume that \(\phi = \phi_0 \interp \phi_1\), and by I.H.\ we have that \(\ILP \vdash \phi_i \leftrightarrow \phi^\sharp_i\) for \(i = 0,1\).
  It is easy to show that
  \[
    \ILP \vdash (\phi_0 \interp \phi_1) \leftrightarrow (\phi^\sharp_0 \interp \phi^\sharp_1).
  \]
  By propositional reasoning we have that \(\ILP \vdash \dnec(\phi^\sharp_0 \interp \phi^\sharp_1) \to (\phi^\sharp_0 \interp \phi^\sharp_1)\).
  By axiom \((P)\) we obtain the other direction so
  \[
    \ILP \vdash (\phi^\sharp_0 \interp \phi^\sharp_1) \leftrightarrow  \dnec (\phi^\sharp_0 \interp \phi^\sharp_1).
  \]
  Putting both displayed equivalences together we obtain the desired result.

  Proof of 3.
  We proceed by induction on the length of proof of \(\ILP \vdash \phi\).
  First, we prove that the translation of the axioms of \(\ILP\) are provable in \(\IL\).
  The translation of the axioms \(\mathrm{(K)}\), \(\mathrm{(4)}\), \(\mathrm{(L)}\) \((\mathrm{J1})\), \((\mathrm{J2})\), \((\mathrm{J3})\) is straightforward to show.
  Let us see how to show \(\mathrm{(J4)}\), \(\mathrm{(J5)}\) and \((\mathrm{P})\).

  The translation of \(\mathrm{(J4)}\) is \(\dnec(\phi^\sharp \interp \psi^\sharp) \to (\pos \phi^\sharp \vee \pos\pos \phi^\sharp \to \pos \psi^\sharp \vee \pos\pos\psi^\sharp)\).
  Using that for any \(\chi\), \(\IL \vdash \pos\pos \chi \to \pos\chi\) it suffices to show that \(\IL \vdash \dnec(\phi^\sharp \interp \psi^\sharp) \to (\pos \phi^\sharp \to \pos \psi^\sharp \vee \pos\pos\psi^\sharp)\), but this easily follows from axiom \(\mathrm{(J4)}\) itself.

  The translation of \(\mathrm{(J5)}\) is \(\dnec((\pos \phi^\sharp \vee \pos\pos \phi^\sharp) \interp \phi^\sharp)\).
  By necessitation it suffices to show that \(\IL \vdash(\pos \phi^\sharp \vee \pos\pos \phi^\sharp) \interp \phi^\sharp\).
  By axiom \(\mathrm{(J5)}\) we know that \(\IL \vdash \pos \phi^\sharp \interp \phi^\sharp\) and \(\IL \vdash \pos\pos \phi^\sharp \interp \pos \phi^\sharp\).
  Using axiom \(\mathrm{(J2)}\) we obtain that \(\IL \vdash \pos\pos\phi^\sharp \interp \phi^\sharp\) and then by axiom \(\mathrm{(J3)}\) we conclude the desired \(\IL \vdash (\pos \phi^\sharp \vee \pos\pos \phi^\sharp) \interp \phi^\sharp\).

  The translation of axiom \(\mathrm{(P)}\) is \(\dnec(\phi^\sharp \interp \psi^\sharp) \to \dnec \nec \dnec(\phi^\sharp \interp \psi^\sharp)\), but this is provable by using axiom \(\mathrm{(4)}\).

  Assume that \(\ILP \vdash \phi\) since there are shorter proofs of \(\ILP \vdash \psi \to \phi\) and \(\ILP \vdash \psi\).
  By the induction hypothesis we obtain that \(\IL \vdash \psi^\sharp \to \phi^\sharp\) and \(\IL \vdash \psi^\sharp\).
  We can conclude then the desired \(\IL \vdash \phi^\sharp\).

  Finally, assume that \(\phi = \nec \psi\) and there is a shorter proof of \(\ILP \vdash \psi\).
  By the induction hypothesis we have that \(\IL \vdash \psi^\sharp\), then using necessitation once we obtain \(\IL \vdash \nec \psi^\sharp\) and using it twice we obtain \(\IL \vdash \nec\nec\psi^\sharp\).
  We can conclude that \(\IL \vdash \nec \psi^\sharp \wedge \nec\nec\psi^\sharp\), i.e., \(\IL \vdash \dnec \nec \psi^\sharp.\)
\end{proof}

Thanks to the previous Lemma, we can provide uniform interpolation for \(\ILP\).

\begin{theorem}
  \(\ILP\) has uniform interpolation.
\end{theorem}
\begin{proof}
  Let \(\phi\) be a formula and \(V\) a vocabulary.
  Define \(\iota\) as the uniform \(\IL\)-interpolant of \(\phi^\sharp\) in \(V\), which exists by Theorem~\ref{cor:uniform-interpolation-IL}.
  In particular we have that
  \begin{enumerate}
    \item \(\vocab(\iota) \subseteq V\),
    \item \(\IL \vdash \phi^\sharp \to \iota\), and
    \item for any \(\psi\) with \(\vocab(\psi) \subseteq V\) if \(\IL \vdash \phi^\sharp \to \psi\) then \(\IL \vdash \iota \to \psi\).
  \end{enumerate}
  From \(\IL \vdash \phi^\sharp \to \iota\) we obtain that \(\ILP \vdash \phi^\sharp \to \iota\).
  By Proposition~\ref{prop:properties-of-interpretation} we get \(\ILP \vdash \phi \leftrightarrow \phi^\sharp\) so \(\ILP \vdash \phi \to \iota\).

  Finally, let \(\psi\) be such that \(\vocab(\psi) \subseteq V\) and assume that \(\ILP \vdash \phi \to \psi\).
  Using Proposition~\ref{prop:properties-of-interpretation} we get that \(\IL \vdash \phi^\sharp \to \psi^\sharp\).
  Since \(\vocab(\psi^\sharp) = \vocab(\psi) \subseteq V\) we can use that \(\iota\) is the uniform \(\IL\)-interpolant of \(\phi^\sharp\) in \(V\) to obtain that \(\IL \vdash \iota \to \psi^\sharp\).
  Then \(\ILP \vdash \iota \to \psi^\sharp\) and by Proposition~\ref{prop:properties-of-interpretation} once again we get \(\ILP \vdash \iota \to \psi\), as desired.
\end{proof}

\section{Future work}
\label{sec:future-work}
There are multiple directions for future work.
First, not much has been done with respect to uniform interpolation of bimodal (an exception being \cite{kogure2023interpolationpropertiesbimodalprovability}) and unary interpretability logics.
Due to the generality of our method,  it should be possible, once the correct sequent calculi for these logics are found,  to adapt the presented results  for these logics.

In a different direction, it should be possible to extend our techniques  to show Lyndon uniform interpolation.
This will imply that multiple provability logics would have one of the strongest possible interpolation property.

Finally, it is worth noting that attempts to demonstrate uniform interpolation of~\(\IL\) by semantic methods have been unsuccessful.
Hence, our new techniques based on non-wellfounded proofs have created a gap between what can be achieved by semantic and syntactic methods.
It will be interesting to find a semantic proof of uniform interpolation and close this gap, as this would yield semantics tools that corresponds to our non-wellfounded proofs.
%

\begin{appendices}
\section{Completeness of \(\fgentzenIL\)}
\label{sec:completeness-of-gil}

We show that \(\fgentzenIL\) proves all modal axioms of \(\IL\)

(K): $\Box(\phi\rightarrow\psi)\rightarrow(\Box\phi\rightarrow\Box\psi)$

\begin{prooftree}
	\AxiomC{}
  \RightLabel{Ax}
	\UnaryInfC{$\phi, \ldots \Rightarrow \phi, \ldots$}
	\AxiomC{}
  \RightLabel{Ax}
	\UnaryInfC{$\psi, \ldots \Rightarrow \psi, \ldots$}
  \RightLabel{\({\to}\mathrm{L}\)}
	\BinaryInfC{$\phi \to \psi, \phi, (\neg \psi, \neg(\phi\rightarrow\psi), \neg \phi,\bot) \interp \bot \Rightarrow  \bot, \psi$}
  \RightLabel{\(\neg\mathrm{R}\)}
  \doubleLine
	\UnaryInfC{$(\neg \psi, \neg(\phi\rightarrow\psi), \neg \phi,\bot) \interp \bot \Rightarrow \neg (\phi \to \psi), \neg \phi, \bot, \psi$}
  \RightLabel{\(\neg\mathrm{L}\)}
	\UnaryInfC{$\neg \psi, (\neg \psi, \neg(\phi\rightarrow\psi), \neg \phi,\bot) \interp \bot \Rightarrow \neg (\phi \to \psi), \neg \phi, \bot$}
	\AxiomC{}
  \RightLabel{\(\bot\mathrm{L}\)}
	\UnaryInfC{$\bot, \ldots \Rightarrow \ldots$}
	\AxiomC{}
  \RightLabel{\(\bot\mathrm{L}\)}
	\UnaryInfC{$\bot, \ldots \Rightarrow \ldots$}
  \RightLabel{\(\interp_{\IL}\)}
	\TrinaryInfC{$\neg (\phi\rightarrow\psi) \interp \bot, \neg\phi \interp \bot \Rightarrow \neg\psi \interp \bot$}
  \dashedLine
	\UnaryInfC{$\Box(\phi\rightarrow\psi), \Box\phi \Rightarrow \Box\psi$}
  \RightLabel{\({\to}\mathrm{R}\)}
  \doubleLine
	\UnaryInfC{$\Rightarrow\Box(\phi\rightarrow\psi)\rightarrow\Box\phi\rightarrow\Box\psi$}
\end{prooftree}

(4): $\Box\phi\rightarrow\Box\Box\phi$

\begin{prooftree}
	\AxiomC{}
	\RightLabel{Ax}
	\UnaryInfC{$(\neg \nec \phi, \bot, \neg \phi) \interp \bot\Rightarrow\bot,\neg\phi, \neg\phi\interp \bot$}
	\RightLabel{$\neg$L}
	\UnaryInfC{$\neg(\neg\phi\interp \bot),(\neg \nec \phi, \bot, \neg \phi) \interp \bot\Rightarrow\bot,\neg\phi$}
  \dashedLine
	\UnaryInfC{$\neg\Box\phi,(\neg \nec \phi, \bot, \neg \phi) \interp \bot\Rightarrow\bot,\neg\phi$}
	\AxiomC{}
	\RightLabel{$\bot$L}
	\UnaryInfC{$\bot, \ldots \Rightarrow \ldots$}
	\RightLabel{$\rhd_{\IL}$}
	\BinaryInfC{$\neg\phi\rhd\bot\Rightarrow\neg\Box\phi\rhd\bot$}
	\RightLabel{$\rightarrow$R}
	\UnaryInfC{$\Rightarrow\neg \phi \interp \bot\to \neg\Box\phi \interp \bot$}
  \dashedLine
	\UnaryInfC{$\Rightarrow\Box\phi\rightarrow\Box\Box\phi$}
\end{prooftree}

(L): \(\nec(\nec \phi \to \phi) \to \nec \phi\)

\begin{prooftree}
  \AxiomC{}
  \RightLabel{Ax}
  \UnaryInfC{\(\neg \phi \interp \bot, \ldots \Rightarrow \neg \phi \interp \bot, \ldots\)}
  \AxiomC{}
  \RightLabel{Ax}
  \UnaryInfC{\(\phi, \ldots \Rightarrow \phi, \ldots\)}
  \RightLabel{\({\to}\mathrm{L}\)}
  \BinaryInfC{\((\neg \phi \interp \bot) \to \phi, (\neg \phi, \neg(\nec \phi \to \phi), \bot) \interp \bot \Rightarrow  \bot, \phi\)}
  \dashedLine
  \UnaryInfC{\(\nec \phi \to \phi, (\neg \phi, \neg(\nec \phi \to \phi), \bot) \interp \bot \Rightarrow  \bot, \phi\)}
  \RightLabel{\(\neg\mathrm{R}\)}
  \UnaryInfC{\((\neg \phi, \neg(\nec \phi \to \phi), \bot) \interp \bot \Rightarrow \neg(\nec \phi \to \phi), \bot, \phi\)}
  \RightLabel{\(\neg\mathrm{L}\)}
  \UnaryInfC{\(\neg \phi, (\neg \phi, \neg(\nec \phi \to \phi), \bot) \interp \bot \Rightarrow \neg(\nec \phi \to \phi), \bot\)}
  \AxiomC{}
  \RightLabel{\(\bot\mathrm{L}\)}
  \UnaryInfC{\(\bot, \ldots \Rightarrow \ldots\)}
  \RightLabel{\(\interp_{\IL}\)}
  \BinaryInfC{\(\neg (\nec \phi \to \phi) \interp \bot \Rightarrow \neg \phi \interp \bot\)}
  \dashedLine
  \UnaryInfC{\(\nec(\nec \phi \to \phi) \Rightarrow \nec \phi\)}
  \RightLabel{\({\to}\mathrm{R}\)}
  \UnaryInfC{\( \Rightarrow \nec(\nec \phi \to \phi) \to \nec \phi\)}
\end{prooftree}

(J1): \(\nec(\phi \to \psi) \to \phi \interp \psi\)

\begin{prooftree}
  \AxiomC{}
  \RightLabel{Ax}
  \UnaryInfC{\(\phi, \ldots \Rightarrow \phi, \ldots\)}
  \AxiomC{}
  \RightLabel{Ax}
  \UnaryInfC{\(\psi, \ldots \Rightarrow \psi\)}
  \RightLabel{\({\to}\mathrm{L}\)}
  \BinaryInfC{\(\phi \to \psi, \phi, (\phi, \neg(\phi \to \psi), \psi) \interp \bot \Rightarrow \psi\)}
  \RightLabel{\(\neg\mathrm{R}\)}
  \UnaryInfC{\(\phi, (\phi, \neg(\phi \to \psi), \psi) \interp \bot \Rightarrow \neg(\phi \to \psi), \psi\)}
  \AxiomC{}
  \RightLabel{\(\bot\mathrm{L}\)}
  \UnaryInfC{\(\bot, \ldots \Rightarrow \bot\)}
  \RightLabel{\(\interp_{\IL}\)}  
  \BinaryInfC{\(\neg (\phi \to \psi) \interp \bot \Rightarrow \phi \interp \psi\)}
  \dashedLine
  \UnaryInfC{\( \nec(\phi \to \psi) \Rightarrow \phi \interp \psi\)}
  \RightLabel{\({\to}\mathrm{R}\)}
  \UnaryInfC{\( \Rightarrow\nec(\phi \to \psi) \to \phi \interp \psi\)}
\end{prooftree}

(J2): \(\phi \interp \psi \to \psi \interp \chi \to \phi \interp \chi\)

\begin{prooftree}
  \AxiomC{}
  \RightLabel{Ax}
  \UnaryInfC{\(\phi, (\phi, \phi, \psi, \chi) \interp \bot \Rightarrow \phi, \psi, \chi\)}
  \AxiomC{}
  \RightLabel{Ax}
  \UnaryInfC{\(\psi, (\psi, \psi, \chi) \interp \bot \Rightarrow \psi, \chi\)}
  \AxiomC{}
  \RightLabel{Ax}
  \UnaryInfC{\(\chi, (\chi, \chi) \interp \bot \Rightarrow \chi\)}
  \RightLabel{\(\interp_{\IL}\)}
  \TrinaryInfC{\( \phi \interp \psi, \psi \interp \chi \Rightarrow \phi \interp \chi\)}
  \RightLabel{\({\to}\mathrm{R}\)}
  \doubleLine
  \UnaryInfC{\( \Rightarrow \phi \interp \psi \to \psi \interp \chi \to \phi \interp \chi\)}
\end{prooftree}
where \(\interp_{\IL}\) has been applied with ordering \(\psi \interp \chi, \phi \interp \psi\) and main formula \(\phi \interp \chi\).

(J3): \((\phi \interp \chi) \to (\psi \interp \chi) \to (\phi \vee \psi) \interp \chi\)

\begin{prooftree}
  \AxiomC{}
  \RightLabel{Ax}
  \UnaryInfC{\(\phi, \ldots \Rightarrow \phi, \ldots\)}
  \AxiomC{}
  \RightLabel{Ax}
  \UnaryInfC{\(\psi, \ldots \Rightarrow \psi, \ldots\)}
  \RightLabel{\(\vee\mathrm{L}\)}
  \BinaryInfC{\(\phi \vee \psi, (\phi \vee \psi, \phi, \psi, \chi) \interp \bot \Rightarrow \phi, \psi, \chi\)}
  \AxiomC{}
  \RightLabel{Ax}
  \UnaryInfC{\(\chi, \ldots \Rightarrow \chi,\ldots\)}
  \AxiomC{}
  \RightLabel{Ax}
  \UnaryInfC{\(\chi, \ldots \Rightarrow \chi,\ldots\)}
  \RightLabel{\(\interp_{\IL}\)}
  \TrinaryInfC{\(\phi \interp \chi, \psi \interp \chi \Rightarrow (\phi \vee \psi) \interp \chi\)}
  \doubleLine\RightLabel{\({\to}\mathrm{R}\)}
  \UnaryInfC{\( \Rightarrow (\phi \interp \chi) \to (\psi \interp \chi) \to (\phi \vee \psi) \interp \chi\)}
\end{prooftree}

(J4): \(\phi \interp \psi \to (\pos \phi \interp \pos \psi)\)

\begin{prooftree}
  \AxiomC{}
  \RightLabel{(J3)}
  \UnaryInfC{\( \phi \interp \psi, \psi \interp \bot\Rightarrow \phi \interp \bot\)}
  \RightLabel{\(\neg\mathrm{R}\)}
  \UnaryInfC{\( \phi \interp \psi \Rightarrow \neg(\psi \interp \bot), \phi \interp \bot\)}
  \RightLabel{\(\neg\mathrm{L}\)}
  \UnaryInfC{\( \phi \interp \psi, \neg( \phi \interp \bot) \Rightarrow \neg(\psi \interp \bot)\)}
  \dashedLine
  \UnaryInfC{\( \phi \interp \psi, \pos \phi \Rightarrow \pos \psi\)}
  \RightLabel{\({\to}\mathrm{R}\)}
  \doubleLine
  \UnaryInfC{\( \Rightarrow \phi \interp \psi \to (\pos \phi \to \pos \psi)\)}
\end{prooftree}

(J5): $\Diamond\phi\rhd\phi$

\begin{prooftree}
  \AxiomC{}
  \RightLabel{Ax}
  \UnaryInfC{\((\pos \phi, \phi) \interp \bot \Rightarrow \phi, \phi \interp \bot \)}
  \RightLabel{\(\neg\mathrm{L}\)}
  \UnaryInfC{\(\neg(\phi \interp \bot), (\pos \phi, \phi) \interp \bot \Rightarrow \phi\)}
  \dashedLine
  \UnaryInfC{\(\pos \phi, (\pos \phi, \phi) \interp \bot \Rightarrow \phi\)}
	\RightLabel{$\interp_{\IL}$}
	\UnaryInfC{$\Rightarrow\pos\phi\interp\phi$}
\end{prooftree}

\end{appendices}

\printbibliography[title={Bibliography}]

\end{document}